\documentclass{amsart}

\usepackage{graphicx}
\usepackage{color}

\usepackage[T2A]{fontenc}

\newtheorem{theorem}{Theorem}[section]
\newtheorem{lemma}[theorem]{Lemma}
\newtheorem{property}[theorem]{Property}

\theoremstyle{definition}

\theoremstyle{remark}
\newtheorem{remark}[theorem]{Remark}
\newtheorem{assumption}{Assumption} % special theorem style

\numberwithin{equation}{section}

\newcommand{\sinush}{\mathop{\rm sinh}}
\newcommand{\cosinh}{\mathop{\rm cosh}}

\newcommand{\arch}{\mathop{\rm arcosh}}

\newcommand{\cl}{\mathop{\rm cl}}

\begin{document}

\title[Boundary metrics of convex compact domains in quasi-Fuchsian manifolds]{Compact domains with prescribed convex boundary metrics in quasi-Fuchsian manifolds}

\author{Dmitriy Slutskiy}
%\address{Institut de Recherche Math\'ematique Avanc\'ee, Universit\'e de Strasbourg, France}
%\curraddr{Department of Mathematics and Statistics,
%Case Western Reserve University, Cleveland, Ohio 43403}
\email{slutskiy@math.unistra.fr}
\thanks{The author acknowledges support from U.S. National Science Foundation grants DMS 1107452,
1107263, 1107367 "RNMS: Geometric Structures and Representation Varieties" (the GEAR Network)."}

%\author{Author Two}
%\address{Mathematical Research Section, School of Mathematical Sciences,
%Australian National University, Canberra ACT 2601, Australia}

%    General info
\subjclass[2010]{Primary 53C45, 20H10; Secondary 53C42, 30F40, 57M10, 57M40}

%\date{January 1, 2001 and, in revised form, June 22, 2001.}


\keywords{quasi-Fuchsian manifold, convex compact domain, Alexandrov space, induced metric}

\begin{abstract}
We show the existence of a convex compact domain in a quasi-Fuchsian manifold such that the induced metric on its boundary coincides with a prescribed surface metric of curvature $K\geq-1$ in the sense of A.~D.~Alexandrov.
\end{abstract}

\maketitle

%% The correct journal style for \specialsection is all uppercase; a known bug
%% in amsart.cls prevents this, so input must be uppercase until it is fixed.
%\specialsection*{This is a Special Section Head}

%\specialsection*{THIS IS A SPECIAL SECTION HEAD}
%This is an example of a special section head%

%%%%%%%%%%%%%%%%%%%%%%%%%%%%%%%%%%%%%%%%%%%%%%%%%%%%%%%%%%%%%%%%%%%%%%%%
%\footnote{Here is an example of a footnote. Notice that this footnote
%text is running on so that it can stand as an example of how a footnote
%with separate paragraphs should be written.
%\par
%And here is the beginning of the second paragraph.}%
%%%%%%%%%%%%%%%%%%%%%%%%%%%%%%%%%%%%%%%%%%%%%%%%%%%%%%%%%%%%%%%%%%%%%%%%

\section{Construction of a quasi-Fuchsian manifold containing a compact convex domain with a prescribed Alexandrov metric of curvature $K\geq-1$ on the boundary}
\label{ch_aaa}

The problem of existence and uniqueness of an isometric realization of a surface with a prescribed metric in a given ambient space is classical in the metric geometry. Initially stated in the Euclidean case, it can be posed for surfaces in other spaces, in particular, in hyperbolic $3$-space $\mathbb{H}^3$.

One of the first fundamental results in this theory is due to A.~D.~Alexandrov. It concerns the realization of polyhedral surfaces in the spaces of constant curvature.

As in~\cite{chaaa_Shi1993}, we denote by $M^m(K)$ the $m$-dimensional complete simply connected space of constant sectional curvature $K$. So, $M^3(K)$ stands for spherical $3$-space of curvature $K$ in the case $K>0$; $M^3(K)$ stands for hyperbolic $3$-space of curvature $K$ when $K<0$; and in the case $K=0$, $M^3(K)$ denotes Euclidean $3$-space.

Then the result of A.~D.~Alexandrov reads as follows:

\begin{theorem} [\cite{chaaa_ADA2006}] \label{chaaa_thm_alexandrov_polyhedra}
Let $h$ be a metric of a constant sectional curvature $K$ with cone singularities on a sphere $S^2$ such that the total angle around every singular point of $h$ do not exceed $2\pi$. Then there exists a closed convex polyhedron in $M^3(K)$ equipped with the metric $h$ which is unique up to the isometries of $M^3(K)$. Here we include the doubly covered convex polygons, which are plane in $M^3(K)$, in the set of convex polyhedra.
\end{theorem}

Later, A.~D.~Alexandrov and A.~V.~Pogorelov proved the following statement in $\mathbb{H}^3$ \cite{chaaa_Pogor1973}:
\begin{theorem} \label{chaaa_thm_alexandrov-pogorelov}
Let $h$ be a $C^{\infty}$-regular metric of a sectional curvature which is strictly greater than $-1$ on a sphere $S^2$. Then there exists an isometric immersion of the sphere $(S^2,h)$ into hyperbolic $3$-space $\mathbb{H}^3$ which is unique up to the isometries of $\mathbb{H}^3$. Moreover, this immersion bounds a convex domain in $\mathbb{H}^3$.
\end{theorem}

\textbf{Definition.} \cite[p.~30]{chaaa_MaTa1998}, \cite[p.~11]{chaaa_Otal1996}  A discrete finitely generated subgroup $\Gamma_{F}\subset {PSL}_{2}(\mathbb{R})$ without torsion and such that the quotient $\mathbb{H}^2/\Gamma_{F}$ has a finite volume, is called a \emph{Fuchsian group}.

Given a hyperbolic plane $\mathcal{P}$ in $\mathbb{H}^3$ and a Fuchsian group $\Gamma_{\mathcal{P}}\subset {PSL}_{2}(\mathbb{R})$ acting on $\mathcal{P}$, we can canonically extend the action of the group $\Gamma_{\mathcal{P}}$ on the whole space $\mathbb{H}^3$.

Here we recall another result on the above-mentioned problem considered for a special type of hyperbolic manifolds, namely, for Fuchsian manifolds, which is due to M.~Gromov \cite{chaaa_Gro1986}:
\begin{theorem} \label{chaaa_thm_gromov}
Let $S$ be a compact surface of genus greater than or equal to $2$, equipped with a $C^{\infty}$-regular metric $h$ of a sectional curvature which is greater than $-1$ everywhere. Then there exists a Fuchsian group $\Gamma_{F}$ acting on $\mathbb{H}^3$, such that the surface $(S,h)$ is isometrically embedded in $\mathbb{H}^3/\Gamma_{F}$.
\end{theorem}
\begin{remark}\label{chaaa_remark_fuchsian_manifold}
The hyperbolic manifold $\mathbb{H}^3/\Gamma_{F}$ from the statement of Theorem~\ref{chaaa_thm_gromov} is called Fuchsian. Note also that the limit set $\Lambda(\Gamma_{F})\subset\partial_{\infty}\mathbb{H}^3$ of a Fuchsian group $\Gamma_{F}$ is a geodesic circle in projective space $\mathbb{CP}^{1}$ regarded as the boundary at infinity $\partial_{\infty}\mathbb{H}^3$ of the Poincar\'e ball model of hyperbolic $3$-space $\mathbb{H}^3$.
\end{remark}

\textbf{Definition.} \cite{chaaa_Lab1992} A compact hyperbolic manifold $M$
%with the boundary $\partial M$
is said to be \emph{strictly convex} if any two points in $M$ can be joined by a minimizing geodesic which lies inside the interior of $M$. This condition implies that the intrinsic curvature of $\partial M$ is greater than $-1$ everywhere (the term "hyperbolic" means for us "of a constant curvature equal to $-1$ everywhere").

In 1992 F.~Labourie \cite{chaaa_Lab1992} obtained the following result which can be considered as a generalization of Theorems~\ref{chaaa_thm_alexandrov-pogorelov} and~\ref{chaaa_thm_gromov}:
\begin{theorem} \label{chaaa_thm_labourie}
Let $M$ be a compact manifold with boundary (different from the solid torus) which admits a structure of a strictly convex hyperbolic manifold. Let $h$ be a $C^{\infty}$-regular metric on $\partial M$ of a sectional curvature which is strictly greater than $-1$ everywhere. Then there exists a convex hyperbolic metric $g$ on $M$ which induces $h$ on $\partial M$:
\begin{equation*}\label{chaaa_frm_labourie_theorem}
g\mid_{\partial M}=h.
\end{equation*}
\end{theorem}

\textbf{Definition.} \cite[p.~120]{chaaa_MaTa1998} A \emph{quasi-Fuchsian space} is the quasiconformal deformation space $QH(\Gamma_{F})$ of a Fuchsian group $\Gamma_{F}\subset {PSL}_{2}(\mathbb{R})$.

In other words, the quasi-Fuchsian manifold $QH(\Gamma_{F})$ is a quotient $\mathbb{H}^3/\Gamma_{qF}$ of $\mathbb{H}^3$ by a discrete finitely generated group $\Gamma_{qF}\subset {PSL}_{2}(\mathbb{R})$ of hyperbolic isometries of $\mathbb{H}^3$ such that the limit set $\Lambda(\Gamma)\subset\partial_{\infty}\mathbb{H}^3$ of $\Gamma$ is a Jordan curve which can be obtained from the circle $\Lambda(\Gamma_{F})\subset\partial_{\infty}\mathbb{H}^3$ by a quasiconformal deformation of $\partial_{\infty}\mathbb{H}^3$.

In geometric terms, a quasi-Fuchsian manifold is a complete hyperbolic manifold homeomorphic to $\mathcal{S}\times\mathbb{R}$, where $\mathcal{S}$ is a closed connected surface of genus at least $2$, which contains a convex compact subset.

Let us also recall the A.~D.~Alexandrov notion of curvature which does not require a metric of a surface to be regular.

Let $X$ be a complete locally compact length space and let $\mathrm{d}_{X}(\cdot,\cdot)$ stands for the distance between points in $X$. For a triple of points $p,q,r\in X$ a geodesic triangle $\triangle(pqr)$ is a triple of geodesics joining these three points.
For a geodesic triangle $\triangle(pqr)\subset X$ we denote by $\triangle(\tilde{p}\tilde{q}\tilde{r})$ a geodesic triangle sketched in $M^2(K)$ whose corresponding edges have the same lengths as $\triangle(pqr)$.

\textbf{Definition.} \cite[p.~7]{chaaa_Shi1993} $X$ is said to have \emph{curvature bounded below by} $K$ iff every point $x\in X$ has an open neighborhood $U_x\subset X$ such that for every geodesic triangle $\triangle(pqr)$ whose edges are contained entirely in $U_x$ the corresponding geodesic triangle $\triangle(\tilde{p}\tilde{q}\tilde{r})$ sketched in $M^2(K)$ has the following property: for every point $z\in qr$ and for $\tilde{z}\in \tilde{q}\tilde{r}$ with $\mathrm{d}_{X}(q,z)=\mathrm{d}_{M^2(K)}(\tilde{q},\tilde{z})$
we have
\begin{equation*}\label{chaaa_def_alexandrov_space}
\mathrm{d}_{X}(p,z)\geq\mathrm{d}_{M^2(K)}(\tilde{p},\tilde{z}).
\end{equation*}

Our main goal is to prove the following extension of Theorem~\ref{chaaa_thm_labourie}:

\begin{theorem}\label{chaaa_theorem_manifolds_with_convex_alexandrov_metric_on_boundary}
Let $\mathcal{M}$ be a compact connected $3$-manifold with boundary of the type $\mathcal{S}\times[-1,1]$ where $\mathcal{S}$ is a closed connected surface of genus at least $2$.
Let $h$ be a metric on $\partial\mathcal{M}$ of curvature $K\geq-1$ in Alexandrov sense. Then there exists a hyperbolic metric $g$ in $\mathcal{M}$ with a convex boundary $\partial\mathcal{M}$ such that the metric induced on $\partial\mathcal{M}$ is $h$.
\end{theorem}

In particular, the following result proved in \cite{chaaa_Slu2013} immediately follows from Theorem~\ref{chaaa_theorem_manifolds_with_convex_alexandrov_metric_on_boundary}.

\begin{theorem}\label{chaaa_theorem_manifolds_with_polyhedral_metric_on_boundary}
Let $\mathcal{M}$ be a compact connected $3$-manifold with boundary of the type $\mathcal{S}\times[-1,1]$ where $\mathcal{S}$ is a closed connected surface of genus at least $2$.
Let $h$ be a hyperbolic metric with cone singularities of angle less than $2\pi$ on $\partial\mathcal{M}$ such that every singular point of $h$ possesses a neighborhood in $\partial\mathcal{M}$ which does not contain other singular points of $h$. Then there exists a hyperbolic metric $g$ in $\mathcal{M}$ with a convex boundary $\partial\mathcal{M}$ such that the metric induced on $\partial\mathcal{M}$ is $h$.
\end{theorem}

Theorem~\ref{chaaa_theorem_manifolds_with_polyhedral_metric_on_boundary} can also be considered as an analogue of Theorem~\ref{chaaa_thm_alexandrov_polyhedra} for the convex hyperbolic manifolds with polyhedral boundary.

\textbf{Definition.} \cite{chaaa_CaEpGr2006}\emph{A pleated surface} in a hyperbolic $3$-manifold $\mathcal{M}$ is a complete hyperbolic surface $\mathcal{S}$ together with an isometric map $f:\mathcal{S}\rightarrow\mathcal{M}$ such that every $s\in\mathcal{S}$ is in the interior of some geodesic arc which is mapped by $f$ to a geodesic arc in $\mathcal{M}$.

A pleated surface resembles a polyhedron in the sense that it has flat faces that meet along edges. Unlike a polyhedron, a pleated surface has no corners, but it may have infinitely many edges that form a lamination.

\begin{remark}\label{chaaa_remark_surfaces_in_key_theorem_can_be_pleated}
The surfaces serving as the connected components of the boundary $\partial\mathcal{M}$ of the manifold $\mathcal{M}$ from the statement of Theorem~\ref{chaaa_theorem_manifolds_with_polyhedral_metric_on_boundary}, which are equipped by assumption with hyperbolic polyhedral metrics, do not necessarily have to be polyhedra embedded in $\mathcal{M}$: these surfaces can be partially pleated.
\end{remark}

\textbf{Definition.} \cite{chaaa_MS2009} Let $\mathcal{M}$ be the interior of
a compact manifold with boundary. A complete hyperbolic metric $g$
on $\mathcal{M}$ is convex co-compact if $\mathcal{M}$ contains a compact subset $\mathcal{K}$ which
is convex: any geodesic segment $c$ in $(\mathcal{M},g)$ with endpoints in $\mathcal{K}$ is
contained in $\mathcal{K}$.

In 2002 J.-M.~Schlenker \cite{chaaa_Sch2006} proved uniqueness of the metric $g$ in Theorem~\ref{chaaa_thm_labourie}. Thus, he obtained
\begin{theorem} \label{chaaa_thm_labourie-schlenker}
Let $M$ be a compact connected $3$-manifold with boundary (different from the solid torus) which admits a complete hyperbolic convex co-compact metric. Let $g$ be a hyperbolic metric on $M$ such that $\partial M$ is $C^{\infty}$-regular and strictly convex. Then the induced metric $I$ on $\partial M$ has curvature $K>-1$. Each $C^{\infty}$-regular metric on $\partial M$ with $K>-1$ is induced on $\partial M$ for a unique choice of $g$.
\end{theorem}

It would be natural to conjecture that the metric $g$ in the statements of Theorems~\ref{chaaa_theorem_manifolds_with_convex_alexandrov_metric_on_boundary} and~\ref{chaaa_theorem_manifolds_with_polyhedral_metric_on_boundary} is unique. The methods used in their demonstration do not presently allow to attack this problem.

At last, recalling that the convex quasi-Fuchsian manifolds are special cases of the convex co-compact manifolds, we can guess that Theorems~\ref{chaaa_theorem_manifolds_with_convex_alexandrov_metric_on_boundary} and~\ref{chaaa_theorem_manifolds_with_polyhedral_metric_on_boundary} remain valid in the case when $\mathcal{M}$ is a convex co-compact manifold. It would be interesting to verify this hypothesis in the future.

\subsection{Proof of Theorem~\ref{chaaa_theorem_manifolds_with_convex_alexandrov_metric_on_boundary}}
\label{chaaa_sec_proof_of_main_theorem}

A compact connected $3$-manifold $\mathcal{M}$ of the type $\mathcal{S}\times[-1,1]$ from the statement of Theorem~\ref{chaaa_theorem_manifolds_with_convex_alexandrov_metric_on_boundary}, where $\mathcal{S}$ is a closed connected surface of genus at least $2$,
can be regarded as a convex compact
$3$-dimensional domain of an unbounded quasi-Fuchsian manifold $\mathcal{M}^{\circ}=\mathbb{H}^{3}/\Gamma_{QF}$ where $\Gamma_{QF}$ stands for a quasi-Fuchsian group of isometries of hyperbolic space $\mathbb{H}^3$. Note that the boundary $\partial\mathcal{M}$ of such domain $\mathcal{M}$ consists of two distinct locally convex compact $2$-surfaces in $\mathcal{M}^{\circ}$. Thus, the metric $h$ from the statement of Theorem~\ref{chaaa_theorem_manifolds_with_convex_alexandrov_metric_on_boundary} is a pair of Alexandrov metrics of curvature $K\geq-1$ at every point defined on a couple of compact connected surfaces of the same genus as $\mathcal{M}$, and our aim is to find such quasi-Fuchsian subgroup $\Gamma_{QF}$ of isometries of hyperbolic space $\mathbb{H}^3$ and such convex compact domain $\mathcal{M}\subset\mathcal{M}^{\circ}$ that the induced metric of its boundary $\partial\mathcal{M}$ coincides with $h$.

The main idea of the proof of Theorem~\ref{chaaa_theorem_manifolds_with_convex_alexandrov_metric_on_boundary} is
\begin{itemize}
\item[$(1)$] to approximate the Alexandrov metric $h$ by a sequence $\{h_{n}\}_{n\in\mathbb{N}}$ of $C^{\infty}$-regular metrics for which the Labourie-Schlenker Theorem \ref{chaaa_thm_labourie-schlenker} is applicable, and therefore, there are such quasi-Fuchsian groups $\Gamma_{n}$ of isometries of $\mathbb{H}^3$ and such convex compact domains $\mathcal{M}_{n}$ in the quasi-Fuchsian manifolds $\mathcal{M}^{\circ}_{n}=\mathbb{H}^{3}/\Gamma_{n}$ that the induced metrics of the boundaries $\partial\mathcal{M}_{n}$ of the sets $\mathcal{M}_{n}$ are exactly $h_{n}$, $n\in\mathbb{N}$;
\item[$(2)$] to find a sequence of positive integers $n_{k}\xrightarrow[k\rightarrow\infty]{}\infty$ such that the subsequences of groups $\{\Gamma_{n_{k}}\}_{k\in\mathbb{N}}$ and of domains $\{\mathcal{M}_{n_{k}}\}_{k\in\mathbb{N}}$ converge (the types of convergence will be precised later);
\item[$(3)$] and to show that the induced metric on the boundary of the limit domain $\mathcal{M}$ coincides with $h$.
\end{itemize}

For convenience, let us introduce new notation of some entities that we considered before: we redefine the domain $\mathcal{M}$ and the quasi-Fuchsian manifold $\mathcal{M}^{\circ}$ by the symbols $\mathcal{M}_{\infty}$ and $\mathcal{M}^{\circ}_{\infty}$, correspondingly. Also, let us denote the connected components of the boundary $\partial\mathcal{M}_{\infty}$ of the limit domain $\mathcal{M}_{\infty}$ by $\mathcal{S}^{+}_{\infty}$ and $\mathcal{S}^{-}_{\infty}$, and the induced metrics on the surfaces $\mathcal{S}^{+}_{\infty}$ and $\mathcal{S}^{-}_{\infty}$ by $h^{+}_{\infty}$ and $h^{-}_{\infty}$, respectively. Therefore, to define the metric $h$ from the statement of Theorem~\ref{chaaa_theorem_manifolds_with_convex_alexandrov_metric_on_boundary} means to give a pair of Alexandrov metrics $h^{+}_{\infty}$ and $h^{-}_{\infty}$ of curvature $K\geq-1$ at every point.

\subsubsection{Construction of sequences of metrics converging to the prescribed metrics}
\label{chaaa_sec_metrics_converging_to_the_prescribed_ones}

\textbf{Definition.} We say that a sequence of metrics $\{h_{n}\}_{n\in\mathbb{N}}$ on a compact surface $\mathcal{S}$ converges to a metric $h$ if for any $\varepsilon>0$ there exists such $N(\varepsilon)\in\mathbb{N}$ that all integers $n\geq N(\varepsilon)$ and for any pair of points $x$ and $y$ on $\mathcal{S}$ the following inequality holds:
\begin{equation}\label{chaaa_frm_convergence_of_metrics_alexandrov_sense}
|\mathrm{d}_{h_{n}}(x,y)-\mathrm{d}_{h}(x,y)|<\varepsilon.
\end{equation}

First, we shall learn to approximate an Alexandrov metric of curvature $K\geq-1$ on a compact connected surface by a sequence of hyperbolic polyhedral metrics (i.e. of the sectional curvature $-1$ everywhere except at a discrete set of points with conic singularities of angles less than $2\pi$). Next, we shall learn to approximate any hyperbolic polyhedral metric by a sequence of $C^{\infty}$-regular metrics of curvature $K>-1$. Thus, we will be able to find a sequence of $C^{\infty}$-regular metrics of curvature $K>-1$ on a compact connected surface converging to the given metric of curvature $K\geq-1$ at every point in Alexandrov sense.

\subsubsection*{Construction of a sequence of hyperbolic polyhedral metrics converging to a metric in Alexandrov sense}
\label{chaaa_sec_polyhedral_metrics_converging_to_an_alexandrov_one}

A.~D.~Alexandrov in \cite{chaaa_ADA2006} developed a way to approximate an Alexandrov metric of curvature $K\geq0$ on a compact connected surface by a sequence of Euclidean polyhedral metrics. Recently T.~Richard \cite[Annex A]{chaaa_Ri2012} adapted the Alexandrov method to the case of Alexandrov metrics of curvature $K\geq-1$.

Here we give a more detailed description of what T.~Richard proved in the annex of his PhD thesis.

In terms of \cite[Annex A]{chaaa_Ri2012} let us recall the following definition due to A.~D.~Ale\-xandrov.

\textbf{Definition.} Let $(X,d)$ be an Alexandrov compact surface of curvature $K\geq-1$ everywhere. A \emph{triangulation} $\mathcal{T}$ of $(X,d)$ is a family of geodesic triangles $\{T_{i}\}_{i\in\emph{I}}$ with disjoint interiors each homeomorphic to an open disc and such that the family $\{T_{i}\}_{i\in\emph{I}}$ covers $X$. Note that in this definition two triangles can have edges intersecting in more than one point that do not coincide though.

T.~Richard verifies that the following proposition proved in \cite[Section~6, p.~88]{chaaa_ADA2006} is valid for an Alexandrov surface of curvature $K\geq-1$.

\begin{lemma}[Lemma~A.1.2 in~\cite{chaaa_Ri2012}]\label{chaaa_lemme_triangulation_of_alexandrov_space}
For every $\varepsilon>0$, $(X,d)$ admits a triangulation (in Alexandrov sense) by convex triangles which diameters are inferior than $\varepsilon$.
\end{lemma}

After T.~Richard let us fix $\varepsilon>0$, denote by $\mathcal{T}_{\varepsilon}$ a triangulation of $(X,d)$ provided by Lemma~\ref{chaaa_lemme_triangulation_of_alexandrov_space}, and construct a polyhedral surface with hyperbolic faces $(\overline{X}_{\varepsilon},\bar{d}_{\varepsilon})$ as it follows: for every triangle $T\in\mathcal{T}_{\varepsilon}$ we associate a comparison triangle $\overline{T}$ sketched on a hyperbolic plane $\mathbb{H}^2$ ($=M^2(-1)$) such that all corresponding edges of $T$ and $\overline{T}$ have equal lengths, then we glue together the collection of hyperbolic comparison triangles following the same combinatorics as one of $\mathcal{T}_{\varepsilon}$, and thus we obtain a polyhedral surface $\overline{X}_{\varepsilon}$.

We must note the following property of $\overline{X}_{\varepsilon}$:

\begin{lemma}[Lemma~A.2.1 in~\cite{chaaa_Ri2012}]\label{chaaa_lemme_curvature_of_a_comparison_polyhedron}
$(\overline{X}_{\varepsilon},\bar{d}_{\varepsilon})$ has curvature $K\geq-1$ everywhere in Alexandrov sense.
\end{lemma}

\begin{remark}\label{chaaa_remark_polyhedral_surface_of_curvature_geq_-1}
By construction, the curvature of $\overline{X}_{\varepsilon}$ is equal to $-1$ everywhere with the exception of vertices of the triangles forming $\overline{X}_{\varepsilon}$. Therefore, Lemma~\ref{chaaa_lemme_curvature_of_a_comparison_polyhedron} means that the above mentioned vertices are conic singularities of angles $\leq 2\pi$ of the hyperbolic polyhedral metric on $\overline{X}_{\varepsilon}$.
\end{remark}

At last, T.~Richard~\cite[pp.~87--91]{chaaa_Ri2012} proves that for any $\varepsilon>0$ there exists a real number $\varepsilon'>0$ (depending only on $(X,d)$ and verifying the property $\varepsilon'\rightarrow0$ as $\varepsilon\rightarrow0$) such that for any pair of points $v$ and $w$ in $\overline{X}$ and for a pair of corresponding points $\bar{v}$ and $\bar{w}$ in $\overline{X}_{\varepsilon}$ the following inequality holds:
\begin{equation}\label{chaaa_richard_convergence_of_metrics}
|\bar{d}_{\varepsilon}(\bar{v},\bar{w})-d(v,w)|<\varepsilon'.
\end{equation}
T.~Richard calls this way of convergence of hyperbolic polyhedral surfaces $(\overline{X}_{\varepsilon},\bar{d}_{\varepsilon})$ to the Alexandrov surface $(X,d)$ as $\varepsilon\rightarrow0$ a Gromov-Hausdorff convergence.

Let us rewrite the results of T.~Richard described above in the language developed in Section~\ref{chaaa_sec_proof_of_main_theorem}. We consider an Alexandrov compact surface $(X,d)$ as a topological surface $\mathcal{S}$ endowed with a metric $h$ of curvature $K\geq-1$ in Alexandrov sense and we note that the construction of a hyperbolic polyhedral surface $\overline{X}_{\varepsilon}$ based on a triangulation $\mathcal{T}_{\varepsilon}$ of $(X,d)$ ($=(\mathcal{S},h)$) is equivalent to a construction of a hyperbolic polyhedral metric $h_{\varepsilon}$ on $\mathcal{S}$ as follows: leaving the lengths of all edges of the triangulation $\mathcal{T}_{\varepsilon}$ unchanged, we replace the metric $h$ restricted on the interior of each triangle $T\in\mathcal{T}_{\varepsilon}$ by a hyperbolic metric (i.e. of curvature $-1$ everywhere) inside $T$. Thus, the inequality~(\ref{chaaa_richard_convergence_of_metrics}) becomes equivalent to the following one:
\begin{equation*}\label{chaaa_def_convergence_of_metrics_alexandrov_sense}
|\mathrm{d}_{h_{\varepsilon}}(v,w)-\mathrm{d}_{h}(v,w)|<\varepsilon'
\end{equation*}
for all pairs of points $v$ and $w$ in $\mathcal{S}$ (compare it with~(\ref{chaaa_frm_convergence_of_metrics_alexandrov_sense})).

Therefore, choosing a sequence of positive real numbers $\varepsilon_{n}\rightarrow0$ as $n\rightarrow\infty$ and then applying the argument of T.~Richard for each $\varepsilon_{n}$, we state

\begin{lemma}\label{chaaa_lemme_richard_approximation_by_polyhedral_metrics}
Let $\mathcal{S}$ be a closed compact surface endowed with a metric $h$ of curvature $K\geq-1$ in Alexandrov sense, there exists a sequence of hyperbolic polyhedral metrics $\{h_{n}\}_{n\in\mathbb{N}}$ converging to $h$ \emph{(}hereinafter we mean  on default the convergence of metrics in the sense of inequality~\emph{(\ref{chaaa_frm_convergence_of_metrics_alexandrov_sense}))}.
\end{lemma}

\subsubsection*{Construction of a sequence of $C^{\infty}$-regular metrics converging to a hyperbolic polyhedral metric}
\label{chaaa_sec_regular_metrics_converging_to_a_polyhedral_one}

In this Section, we prove the following

\begin{lemma}\label{chaaa_lemma_approximation_of_a_polyhedral_by_metrics_of_positive_curvature}
Let $\mathcal{S}$ be a surface with a hyperbolic polyhedral metric $h$.
Then there is a sequence of $C^{\infty}$-regular metrics $\{h_{n}\}_{n\in\mathbb{N}}$ with sectional curvatures strictly greater than $-1$ everywhere, converging to the metric $h$.
\end{lemma}

First, let us state two preliminary results.

\begin{lemma}\label{chaaa_lemma_approximation_of_a_polyhedral_by_regular_metrics}
Let $\mathcal{S}$ be a surface with a hyperbolic polyhedral metric $h$.
Then there is a sequence of $C^{\infty}$-regular metrics $\{h_{n}\}_{n\in\mathbb{N}}$ with sectional curvatures greater than or equal to $-1$ everywhere, converging to the metric $h$.
\end{lemma}

To prove Lemma~\ref{chaaa_lemma_approximation_of_a_polyhedral_by_regular_metrics}, we construct small conic surfaces in $\mathbb{H}^3$ whose induced metrics coincide with the restrictions of the metric $h$ on neighborhoods of the conic singularities of $h$, and then we convolute these conic surfaces with $C^{\infty}$-smooth functions as in~\cite{chaaa_Gho2002}. A full explanation of this idea is given in \cite[Lemma~3.10]{chaaa_Slu2013}.

Also, a direct calculation shows the validity of the following statement (see \cite[Lemma~3.11]{chaaa_Slu2013} for the detailed proof).

\begin{lemma}\label{chaaa_lemma_homothety_of_metrics}
Consider a regular metric surface $(\mathcal{S},h)$, where $\mathcal{S}$ stands for a $2$-di\-men\-si\-o\-nal surface, $h$ is a metric provided on $\mathcal{S}$, and $K_{h}(x)$ denotes the sectional curvature of $(\mathcal{S},h)$ at a point $x\in\mathcal{S}$. If we consider another metric surface $(\mathcal{S},g)$, where the metric $g=\lambda h$ is a multiple of $h$ and $\lambda>0$ is a positive constant, then the sectional curvature $K_{g}(x)$ of $(\mathcal{S},g)$ at a point $x\in\mathcal{S}$ is related to $K_{h}(x)$ as follows:
\begin{equation}\label{chaaa_frm_homothety_of_metrics_and_sectional_curvatures}
K_{g}(x)=\frac{1}{\lambda}K_{h}(x).
\end{equation}
\end{lemma}

We are now ready to give a demonstration of Lemma~\ref{chaaa_lemma_approximation_of_a_polyhedral_by_metrics_of_positive_curvature}.
\begin{proof}
Let $h$ be a hyperbolic polyhedral metric on a closed compact surface $\mathcal{S}$ of genus $g$.
By Lemma~\ref{chaaa_lemma_approximation_of_a_polyhedral_by_regular_metrics}, there is a sequences of $C^{\infty}$-smooth metrics $\{\hbar_{n}\}_{n\in\mathbb{N}}$ on $\mathcal{S}$, with sectional curvature $\geq-1$ everywhere, converging to $h$ as $n\rightarrow\infty$.

Next, let us choose a monotonically decreasing sequence of real numbers $\lambda_{n}\xrightarrow[n\rightarrow\infty]{}1$ and let us define the metrics $h_{n}\stackrel{\mathrm{def}}{=}\lambda_{n}\hbar_{n}$ on $\mathcal{S}$, $n\in\mathbb{N}$. Thus, by Lemma~\ref{chaaa_lemma_homothety_of_metrics}, the sectional curvatures of the metrics $h_{n}$, $n\in\mathbb{N}$ are strictly greater than $-1$ everywhere on $\mathcal{S}$, and, by construction, the sequence of $C^{\infty}$-smooth metrics $\{h_{n}\}_{n\in\mathbb{N}}$ converges to $h$ as $n\rightarrow\infty$.
\end{proof}

\subsubsection{Convergence of convex surfaces in a compact domain in $\mathbb{H}^{3}$} \label{chaaa_sec_Arzela-Ascoli_theorem_application}

\begin{figure}[!h]
\begin{center}
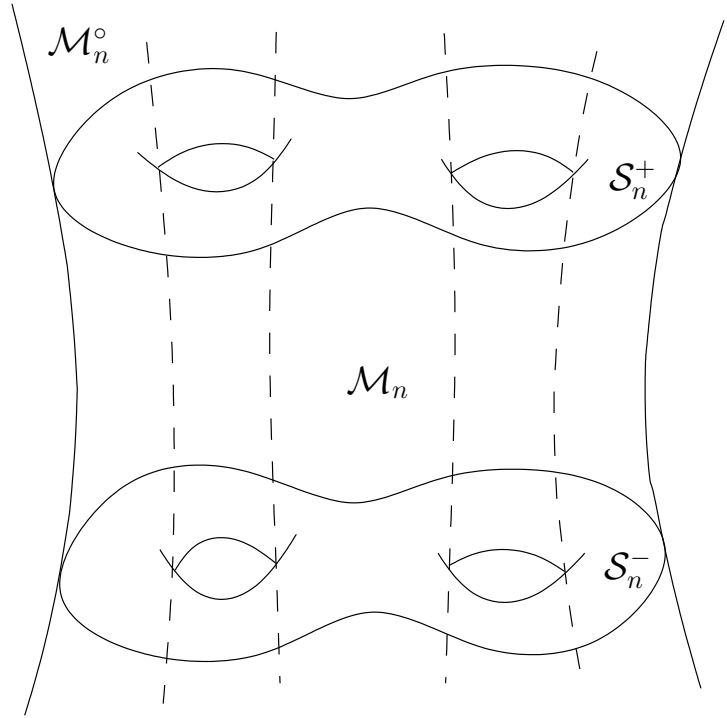
\caption{The surfaces $\mathcal{S}^{+}_{n}$ and $\mathcal{S}^{-}_{n}$ in the quasi-Fuchsian manifold $\mathcal{M}^{\circ}_{n}$.}\label{chaaa_aaa_hyperbolic_manifold}
\end{center}
\end{figure}

Let $h^{+}_{\infty}$ and $h^{-}_{\infty}$ be two metrics of curvature $K\geq-1$ in Alexandrov sense everywhere on a closed compact surface $\mathcal{S}$ of genus $g$. To be able to apply the Labourie-Schlenker Theorem~\ref{chaaa_thm_labourie-schlenker}, we shall construct two sequences of $C^{\infty}$-regular metrics of curvature strictly greater than $-1$, converging to $h^{+}_{\infty}$ and $h^{-}_{\infty}$.
By Lemma~\ref{chaaa_lemme_richard_approximation_by_polyhedral_metrics}, there are two sequences of hyperbolic polyhedral metrics $\{\hbar^{+}_{n}\}_{n\in\mathbb{N}}$ and
$\{\hbar^{-}_{n}\}_{n\in\mathbb{N}}$ on $\mathcal{S}$, converging to
$h^{+}_{\infty}$ and $h^{-}_{\infty}$ as $n\rightarrow\infty$.
Also, by Lemma~\ref{chaaa_lemma_approximation_of_a_polyhedral_by_metrics_of_positive_curvature}, for each $n\in\mathbb{N}$ there are sequences $\{\hbar^{+}_{n,k}\}_{k\in\mathbb{N}}$ and $\{\hbar^{-}_{n,k}\}_{k\in\mathbb{N}}$ of $C^{\infty}$-smooth metrics of curvature $K>-1$ everywhere on $\mathcal{S}$, converging to the hyperbolic polyhedral metrics $\hbar^{+}_{n}$ and $\hbar^{-}_{n}$, respectively, as $k\rightarrow\infty$.
Thus, we are now able to extract sequences of $C^{\infty}$-smooth metrics $\{h^{+}_{n}\}_{n\in\mathbb{N}}$ and $\{h^{-}_{n}\}_{n\in\mathbb{N}}$ of curvature $K>-1$, converging to the Alexandrov metrics $h^{+}_{\infty}$ and $h^{-}_{\infty}$, respectively (where $h^{+}_{n}\in\{\hbar^{+}_{n,k}\}_{k\in\mathbb{N}}$ and $h^{-}_{n}\in\{\hbar^{-}_{n,k}\}_{k\in\mathbb{N}}$, $n\in\mathbb{N}$).

By the Labourie-Schlenker Theorem~\ref{chaaa_thm_labourie-schlenker}, for each $n\in\mathbb{N}$ there is a unique compact convex domain $\mathcal{M}_{n}$ of a quasi-Fuchsian manifold $\mathcal{M}^{\circ}_{n}$ with hyperbolic metric $g_{n}$ such that the induced metrics of the components $\mathcal{S}^{+}_{n}$ and $\mathcal{S}^{-}_{n}$ of the boundary $\partial \mathcal{M}_{n}\stackrel{\mathrm{def}}{=}\mathcal{S}^{+}_{n}\cup \mathcal{S}^{-}_{n}$ are equal to $h^{+}_{n}$ and $h^{-}_{n}$ (see also Fig.~\ref{chaaa_aaa_hyperbolic_manifold}). It means that, for each $n\in\mathbb{N}$ there exist isometric embeddings $f_{\mathcal{S}^{+}_{n}}:(\mathcal{S},h^{+}_{n})\rightarrow\mathcal{M}^{\circ}_{n}$ and $f_{\mathcal{S}^{-}_{n}}:(\mathcal{S},h^{-}_{n})\rightarrow\mathcal{M}^{\circ}_{n}$ such that $f_{\mathcal{S}^{+}_{n}}(\mathcal{S})=\mathcal{S}^{+}_{n}\subset\mathcal{M}^{\circ}_{n}$ and $f_{\mathcal{S}^{-}_{n}}(\mathcal{S})=\mathcal{S}^{-}_{n}\subset\mathcal{M}^{\circ}_{n}$.

As $\mathcal{M}^{\circ}_{n}$ can be retracted by deformation
on $\mathcal{S}^{+}_{n}$ and $\mathcal{S}^{-}_{n}$,
we conclude that their fundamental groups are homomorphic:
\begin{equation*}\label{chaaa_frm_homo_of_fund_grps_S+n_S-n_Mn}
{\pi}_{1}(\mathcal{S}^{+}_{n})\simeq {\pi}_{1}(\mathcal{M}^{\circ}_{n})\simeq{\pi}_{1}(\mathcal{S}^{-}_{n}).
\end{equation*}
Also, by construction,
\begin{equation*}\label{chaaa_frm_homo_of_fund_grps_S+n_S-n_S}
{\pi}_{1}(\mathcal{S}^{+}_{n})\simeq {\pi}_{1}(\mathcal{S})\simeq{\pi}_{1}(\mathcal{S}^{-}_{n}).
\end{equation*}
Hence, for all $n\in\mathbb{N}$
\begin{equation}\label{chaaa_frm_homo_of_fund_grps_Mn_S}
{\pi}_{1}(\mathcal{M}^{\circ}_{n})\simeq {\pi}_{1}(\mathcal{S}).
\end{equation}

Since the manifolds $\mathcal{M}^{\circ}_{n}$, $n\in\mathbb{N}$, are hyperbolic, their universal coverings $\widetilde{\mathcal{M}}^{\circ}_{n}$ are actually copies of hyperbolic $3$-space $\mathbb{H}^3$. Moreover, as each $\mathcal{M}^{\circ}_{n}$ is quasi-Fuchsian, there exists a holonomy representation $\rho_{n}:{\pi}_{1}(\mathcal{M}^{\circ}_{n})\rightarrow\mathcal{I}(\widetilde{\mathcal{M}}^{\circ}_{n})(=\mathcal{I}(\mathbb{H}^3))$ of the fundamental group of $\mathcal{M}^{\circ}_{n}$ in the group of isometries of the universal covering $\widetilde{\mathcal{M}}^{\circ}_{n}(=\mathbb{H}^3)$ such that $\mathcal{M}^{\circ}_{n}=\widetilde{\mathcal{M}}^{\circ}_{n}/[\rho_{n}({\pi}_{1}(\mathcal{M}^{\circ}_{n}))]
=\mathbb{H}^3/[\rho_{n}({\pi}_{1}(\mathcal{M}^{\circ}_{n}))]$ and the limit set $\Lambda_{\rho_{n}}\subset\partial_{\infty}\mathbb{H}^3$ of $\rho_{n}({\pi}_{1}(\mathcal{M}^{\circ}_{n}))$ is homotopic to a circle. By (\ref{chaaa_frm_homo_of_fund_grps_Mn_S}), we can also speak about the holonomy representation $\rho^{\mathcal{S}}_{n}:{\pi}_{1}(\mathcal{S})\rightarrow\mathcal{I}(\widetilde{\mathcal{M}}^{\circ}_{n})(=\mathcal{I}(\mathbb{H}^3))$ of the fundamental group of $\mathcal{S}$ in the group of isometries of the universal covering $\widetilde{\mathcal{M}}^{\circ}_{n}(=\mathbb{H}^3)$ such that $\rho^{\mathcal{S}}_{n}({\pi}_{1}(\mathcal{S}))=\rho_{n}({\pi}_{1}(\mathcal{M}^{\circ}_{n}))$. Thus we have that $\mathcal{M}^{\circ}_{n}=\widetilde{\mathcal{M}}^{\circ}_{n}/[\rho^{\mathcal{S}}_{n}({\pi}_{1}(\mathcal{S}))]
=\mathbb{H}^3/[\rho^{\mathcal{S}}_{n}({\pi}_{1}(\mathcal{S}))]$ and the limit set $\Lambda_{\rho^{\mathcal{S}}_{n}}$ of $\rho^{\mathcal{S}}_{n}({\pi}_{1}(\mathcal{S}))$ is just $\Lambda_{\rho_{n}}$, $n\in\mathbb{N}$. We also suppose that ${\pi}_{1}(\mathcal{S})$ is generated by the elements $\{\gamma_{1},...,\gamma_{l}\}$.

\begin{figure}[!h]
\begin{center}
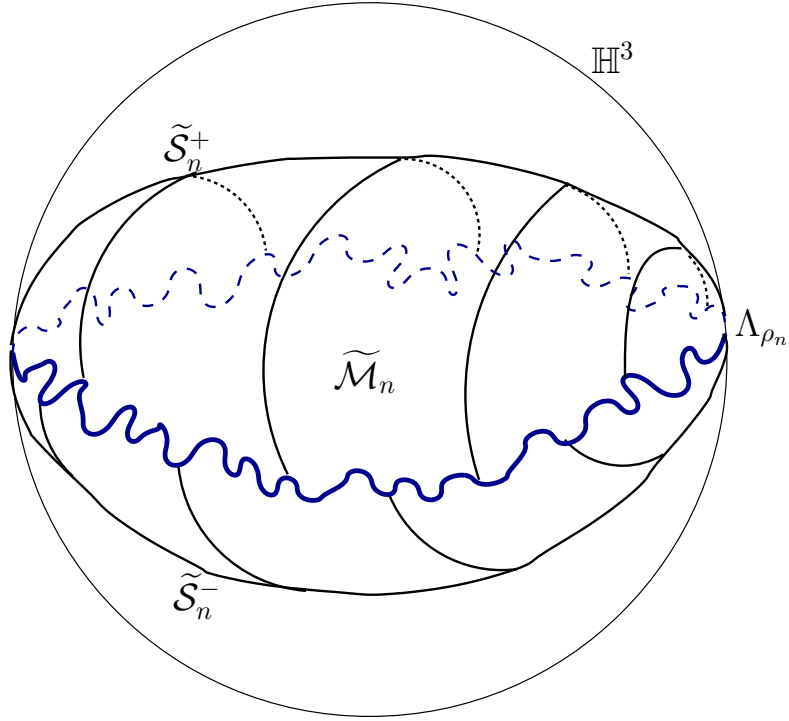
\caption{The universal coverings $\widetilde{\mathcal{S}}^{+}_{n}$ and $\widetilde{\mathcal{S}}^{-}_{n}$ in the Kleinian model $\mathbb{K}^3$ of hyperbolic space $\mathbb{H}^{3}$.}\label{chaaa_universal_cover}
\end{center}
\end{figure}

Inside $\widetilde{\mathcal{M}}^{\circ}_{n}(=\mathbb{H}^3)$, $n\in\mathbb{N}$, we can find a convex set $\widetilde{\mathcal{M}}_{n}$ serving as a universal covering of the domain $\mathcal{M}_{n}\subset\mathcal{M}^{\circ}_{n}$, i.e. such that $\mathcal{M}_{n}=\widetilde{\mathcal{M}}_{n}/[\rho^{\mathcal{S}}_{n}({\pi}_{1}(\mathcal{S}))]$, and a pair of convex surfaces $\widetilde{\mathcal{S}}^{+}_{n}$ and $\widetilde{\mathcal{S}}^{-}_{n}$ serving as universal coverings of the surfaces $\mathcal{S}^{+}_{n}\subset\mathcal{M}^{\circ}_{n}$ and $\mathcal{S}^{-}_{n}\subset\mathcal{M}^{\circ}_{n}$ (see Fig.~\ref{chaaa_universal_cover}), i.e. such that $\mathcal{S}^{+}_{n}=\widetilde{\mathcal{S}}^{+}_{n}/[\rho^{\mathcal{S}}_{n}({\pi}_{1}(\mathcal{S}))]$ and $\mathcal{S}^{-}_{n}=\widetilde{\mathcal{S}}^{-}_{n}/[\rho^{\mathcal{S}}_{n}({\pi}_{1}(\mathcal{S}))]$. By construction, $\partial\widetilde{\mathcal{M}}_{n}=\widetilde{\mathcal{S}}^{+}_{n}\cup\widetilde{\mathcal{S}}^{-}_{n}$ and the boundaries at infinity $\partial_{\infty}\widetilde{\mathcal{M}}_{n}=\partial_{\infty}\widetilde{\mathcal{S}}^{+}_{n}
=\partial_{\infty}\widetilde{\mathcal{S}}^{-}_{n}=\Lambda_{\rho^{\mathcal{S}}_{n}}$. Denote by $p_{n}:\widetilde{\mathcal{M}}_{n}\rightarrow\mathcal{M}_{n}$ the projection of $\widetilde{\mathcal{M}}_{n}$ on $\mathcal{M}_{n}$, $n\in\mathbb{N}$. By construction, $\mathcal{S}^{+}_{n}=p_{n}(\widetilde{\mathcal{S}}^{+}_{n})$ and $\mathcal{S}^{-}_{n}=p_{n}(\widetilde{\mathcal{S}}^{-}_{n})$, $n\in\mathbb{N}$.

For every $n\in\mathbb{N}$ we lift the metric $g_{n}$ of the manifold $\mathcal{M}_{n}$ to the metric $\tilde{g}_{n}$ of the universal covering $\widetilde{\mathcal{M}}_{n}$ in such a way that for any $\gamma\in\pi_{1}(\mathcal{S})$ and for $x\in \mathcal{M}_{n}$ and $\tilde{x}\in\widetilde{\mathcal{M}}_{n}$ satisfying the relation  $x=p_{n}(\tilde{x})$, we have $\tilde{g}_{n}(\tilde{x})={p_{n}}^{*}{g}_{n}(x)$, i.e. the metric $\tilde{g}_{n}(\tilde{x})\in T^{*}_{\tilde{x}}\widetilde{\mathcal{M}}_{n}$ is a pull-back of the metric ${g}_{n}(x)\in T^{*}_{x}\mathcal{M}_{n}$. We have already remarked that, since $g_{n}$ is hyperbolic, $\tilde{g}_{n}$ is hyperbolic too. Denote by $\tilde{h}^{+}_{n}$ the restriction of the metric $\tilde{g}_{n}$ on the surface $\widetilde{\mathcal{S}}^{+}_{n}$ and by $\tilde{h}^{-}_{n}$ the restriction of the metric $\tilde{g}_{n}$ on the surface $\widetilde{\mathcal{S}}^{-}_{n}$, $n\in\mathbb{N}$. By construction, the metric $\tilde{h}^{+}_{n}$ is the lift of $h^{+}_{n}$ from the surface $\mathcal{S}^{+}_{n}$ to its universal covering $\widetilde{\mathcal{S}}^{+}_{n}$ and the metric $\tilde{h}^{-}_{n}$ is the lift of $h^{-}_{n}$ from $\mathcal{S}^{-}_{n}$ to $\widetilde{\mathcal{S}}^{-}_{n}$, $n\in\mathbb{N}$.

\textbf{Definition.} The \emph{diameter} $\delta$ of a set $S$ with a metric $h$ is the following quantity: $\delta\stackrel{\mathrm{def}}{=}\sup\{\mathrm{d}_{h}(u,v)|u,v\in{S}\}$ where $\mathrm{d}_{h}(u,v)$ stands for the distance between points $u$ and $v$ in the metric $h$.

\begin{lemma}\label{chaaa_lemma_ubound_of_diameters_of_surfaces_Sn}
There exists a positive constant $\delta_{\mathcal{S}}<\infty$ which bounds from above the diameters $\delta^{+}_{n}$ and $\delta^{-}_{n}$ of the surfaces $(\mathcal{S},h^{+}_{n})$ and $(\mathcal{S},h^{-}_{n})$ for all $n\in\mathbb{N}$.
\end{lemma}
\begin{proof}
Recall the way of construction of the metric $h^{+}_{n}$ on $\mathcal{S}$, $n\in\mathbb{N}$.

First we applied Lemma~\ref{chaaa_lemme_richard_approximation_by_polyhedral_metrics}, and thus obtained the sequence of hyperbolic polyhedral metrics $\{\hbar^{+}_{n}\}_{n\in\mathbb{N}}$ converging to the Alexandrov metric
$h^{+}_{\infty}$. Every metric $\hbar^{+}_{n}$ is obtained from $h^{+}_{\infty}$ by choosing a geodesic triangulation on $(\mathcal{S},h^{+}_{\infty})$ and by replacing the metric $h^{+}_{\infty}$ of curvature $K\geq -1$ in the interior of each triangle by a hyperbolic plane metric (i.e., of curvature $K=-1$) while keeping the lengths of the edges of a considered triangulation unchanged. Therefore, by construction, the length of any curve on $\mathcal{S}$ measured in the metric $\hbar^{+}_{n}$ does not exceed the corresponding length measured in $h^{+}_{\infty}$.

Next, for each $n\in\mathbb{N}$ we constructed the sequence of $C^{\infty}$-regular metrics $\{\hbar^{+}_{n,k}\}_{k\in\mathbb{N}}$ of curvature $K>-1$ converging to the hyperbolic polyhedral metric $\hbar^{+}_{n}$ by applying Lemma~\ref{chaaa_lemma_approximation_of_a_polyhedral_by_metrics_of_positive_curvature}, and the metric $h^{+}_{n}$ belongs to the set $\{\hbar^{+}_{n,k}\}_{k\in\mathbb{N}}$. The application of Lemma~\ref{chaaa_lemma_approximation_of_a_polyhedral_by_metrics_of_positive_curvature} consists of two stages. The first step is the construction of a sequence of $C^{\infty}$-regular metrics $\{\overline{\hbar}^{+}_{n,k}\}_{k\in\mathbb{N}}$ of curvature $K\geq -1$ converging to $\hbar^{+}_{n}$ due to Lemma~\ref{chaaa_lemma_approximation_of_a_polyhedral_by_regular_metrics}, with the help of smoothing of the conic singularities of $\hbar^{+}_{n}$ by convolution. This procedure does not increase the distance between any two points on the surface $\mathcal{S}$. At the second stage, we considered a sequence of positive real numbers $\{\lambda_{k}\}_{k\in\mathbb{N}}$ decreasing to $1$ and then, by multiplying the metric $\overline{\hbar}^{+}_{n,k}$ by the constant $\lambda_{k}(>1)$, we obtained the metric $\hbar^{+}_{n,k}$ for each $k\in\mathbb{N}$ and for every $n\in\mathbb{N}$, and thus, we increased all distances on $\mathcal{S}$ by $\sqrt{\lambda_{k}}$.

Since $\lambda_{1}\geq\lambda_{k}$ for every $k\in\mathbb{N}$, the distances on $\mathcal{S}$ measured in the metric $h^{+}_{\lambda}\stackrel{\mathrm{def}}{=}\lambda_{1}h^{+}_{\infty}$ are not smaller than the corresponding distances measured in the metrics $h^{+}_{n}$ for all $n\in\mathbb{N}$. Similarly, the distances on $\mathcal{S}$ measured in the metric $h^{-}_{\lambda}\stackrel{\mathrm{def}}{=}\lambda_{1}h^{-}_{\infty}$ are not smaller than the corresponding distances measured in the metrics $h^{-}_{n}$ for all $n\in\mathbb{N}$.

The diameters $\delta^{+}_{\lambda}$ and $\delta^{-}_{\lambda}$ of the surfaces $(\mathcal{S},h^{+}_{\lambda})$ and $(\mathcal{S},h^{-}_{\lambda})$ are finite numbers because $\mathcal{S}$ is compact. We can pose $\delta_{\mathcal{S}}=\max(\delta^{+}_{\lambda},\delta^{-}_{\lambda})$.
\end{proof}

\begin{lemma}\label{chaaa_lemma_ubound_of_diameters_of_Mn}
There exists a positive constant $\delta_{\mathcal{M}}<\infty$ such that for each $n\in\mathbb{N}$ and for every pair of points $u\in\mathcal{S}^{+}_{n}\subset\mathcal{M}^{\circ}_{n}$ and $v\in\mathcal{S}^{-}_{n}\subset\mathcal{M}^{\circ}_{n}$ the distance $\mathrm{d}_{g_{n}}(u,v)$ between $u$ and $v$ in the manifold $\mathcal{M}^{\circ}_{n}$ is less than $\delta_{\mathcal{M}}$.
\end{lemma}
\begin{proof}
By Theorem~\ref{chbb_theorem_distance_between_pairs_of_boundaries_is_unibounded} in Section~\ref{ch_dist_bb}, the distances ${\sigma}^{\mathcal{S}}_{n}$ between the surfaces $\mathcal{S}^{+}_{n}$ and $\mathcal{S}^{-}_{n}$, $n\in\mathbb{N}$, are uniformly bounded by a constant ${\sigma}_{\mathcal{S}}$. Also, by Lemma~\ref{chaaa_lemma_ubound_of_diameters_of_surfaces_Sn}, the diameters of $\mathcal{S}^{+}_{n}$ and $\mathcal{S}^{-}_{n}$ are both bounded by a constant $\delta_{\mathcal{S}}$ which does not depend on $n$. Hence, our assertion is valid if we take $\delta_{\mathcal{M}}$ to be equal to ${\sigma}_{\mathcal{S}}+2\delta_{\mathcal{S}}$.
\end{proof}

Professor Gregory McShane remarked that the existence of a constant $\delta_{\mathcal{M}}>0$ which serves as an common upper bound for the distances between the boundary components $\mathcal{S}^{+}_{n}$ and $\mathcal{S}^{-}_{n}$ of the domains $\mathcal{M}_{n}$, $n\in\mathbb{N}$ does not guarantee that the diameters of $\mathcal{M}_{n}$ are uniformly bounded from above.

Indeed, Jeffrey Brock in his PhD thesis (see also \cite{chaaa_Brock2001}) studied the following example.

Given a pair of homeomorphic Riemann surfaces $X$ and $Y$ of finite type and a "partial pseudo Anosov" mapping class $\phi$, by the Ahlfors-Bers simultaneous uniformization theorem there is a sequence of quasi-Fuchsian manifolds $\{Q({\phi}^{n}X,Y)\}^{\infty}_{n=1}$. The diameters of each of the boundary components of the convex hull of $Q({\phi}^{n}X,Y)$ is uniformly bounded in $n$ and so is the distance between the two boundary components but the diameter of the convex hull of $Q({\phi}^{n}X,Y)$ goes to infinity because of a "cusp growing there" as $n\rightarrow\infty$.

However, the diameters of the domains $\mathcal{M}_{n}$, $n\in\mathbb{N}$ do not play role in the demonstration of Theorem~\ref{chaaa_theorem_manifolds_with_convex_alexandrov_metric_on_boundary}; only the distances between the surfaces $\mathcal{S}^{+}_{n}$ and $\mathcal{S}^{-}_{n}$, $n\in\mathbb{N}$, are of importance here.

Let us now return to the proof of Theorem~\ref{chaaa_theorem_manifolds_with_convex_alexandrov_metric_on_boundary}.

Let us fix an arbitrary point $x\in \mathcal{S}$, which is not, however, a point of singularity for the metrics $h^{+}_{\infty}$ and $h^{-}_{\infty}$ on $\mathcal{S}$, and let us denote $x^{+}_{n}\stackrel{\mathrm{def}}{=}f_{\mathcal{S}^{+}_{n}}(x)\in\mathcal{S}^{+}_{n}\subset\mathcal{M}^{\circ}_{n}$ and $x^{-}_{n}\stackrel{\mathrm{def}}{=}f_{\mathcal{S}^{-}_{n}}(x)\in\mathcal{S}^{-}_{n}\subset\mathcal{M}^{\circ}_{n}$,
$n\in\mathbb{N}$. Denote also the distance between the points $x^{+}_{n}$ and $x^{-}_{n}$ in $\mathcal{M}^{\circ}_{n}$ by ${\sigma}^{x}_{n}$, $n\in\mathbb{N}$. By Lemma~\ref{chaaa_lemma_ubound_of_diameters_of_Mn}, ${\sigma}^{x}_{n}<\delta_{\mathcal{M}}$ for all $n\in\mathbb{N}$.

Let us consider two copies $\widetilde{\mathcal{S}}^{+}$ and $\widetilde{\mathcal{S}}^{-}$ of the universal covering of the surface $\mathcal{S}$ with the projections $p^{+}:\widetilde{\mathcal{S}}^{+}\rightarrow\mathcal{S}$ and $p^{-}:\widetilde{\mathcal{S}}^{-}\rightarrow\mathcal{S}$ and let us fix some points $\tilde{x}^{+}\in\widetilde{\mathcal{S}}^{+}$ and $\tilde{x}^{-}\in\widetilde{\mathcal{S}}^{-}$ such that $p^{+}(\tilde{x}^{+})=x$ and $p^{-}(\tilde{x}^{-})=x$. Without loss of generality we may think that the fundamental group $\pi_{1}(\mathcal{S})$ acts on $\widetilde{\mathcal{S}}^{+}$ and $\widetilde{\mathcal{S}}^{-}$ in the sense that $\mathcal{S}\simeq\widetilde{\mathcal{S}}^{+}/\pi_{1}(\mathcal{S})$ and $\mathcal{S}\simeq\widetilde{\mathcal{S}}^{-}/\pi_{1}(\mathcal{S})$. For every $n\in\mathbb{N}$ we fix an arbitrary pair of points $\tilde{x}^{+}_{n}\in\widetilde{\mathcal{S}}^{+}_{n}\subset\widetilde{\mathcal{M}}^{\circ}_{n}(=\mathbb{H}^3)$ and
$\tilde{x}^{-}_{n}\in\widetilde{\mathcal{S}}^{-}_{n}\subset\widetilde{\mathcal{M}}^{\circ}_{n}$ verifying the conditions $p_{n}(\tilde{x}^{+}_{n})={x}^{+}_{n}$ and $p_{n}(\tilde{x}^{-}_{n})={x}^{-}_{n}$, and such that the distance in ${\mathcal{M}}^{\circ}_{n}$ between $\tilde{x}^{+}_{n}$ and $\tilde{x}^{-}_{n}$ is equal to ${\sigma}^{x}_{n}$. The functions $f_{\mathcal{S}^{+}_{n}}:\mathcal{S}\rightarrow\mathcal{S}^{+}_{n}$ and $f_{\mathcal{S}^{-}_{n}}:\mathcal{S}\rightarrow\mathcal{S}^{-}_{n}$ defined above induce the canonical bijective developing maps $\tilde{f}_{\widetilde{\mathcal{S}}^{+}_{n}}:\widetilde{\mathcal{S}}^{+}\rightarrow\widetilde{\mathcal{S}}^{+}_{n}$ and $\tilde{f}_{\widetilde{\mathcal{S}}^{-}_{n}}:\widetilde{\mathcal{S}}^{-}\rightarrow\widetilde{\mathcal{S}}^{-}_{n}$ with the properties $\tilde{f}_{\widetilde{\mathcal{S}}^{+}_{n}}(\tilde{x}^{+})=\tilde{x}^{+}_{n}$ and $\tilde{f}_{\widetilde{\mathcal{S}}^{-}_{n}}(\tilde{x}^{-})=\tilde{x}^{-}_{n}$ and such that for any  $\gamma\in\pi_{1}(\mathcal{S})$ it is true that $\tilde{f}_{\widetilde{\mathcal{S}}^{+}_{n}}(\gamma.\tilde{x}^{+})=\rho^{\mathcal{S}}_{n}(\gamma).\tilde{x}^{+}_{n}$ and $\tilde{f}_{\widetilde{\mathcal{S}}^{-}_{n}}(\gamma.\tilde{x}^{-})=\rho^{\mathcal{S}}_{n}(\gamma).\tilde{x}^{-}_{n}$, $n\in\mathbb{N}$.

\begin{remark}\label{chaaa_remark_periodicity_of_dev_applications}
The above-mentioned property of developing maps holds for any points $\tilde{y}^{+}\in\widetilde{\mathcal{S}}^{+}$, $\tilde{y}^{-}\in\widetilde{\mathcal{S}}^{-}$ and for every $\gamma\in\pi_{1}(\mathcal{S})$:
\begin{equation*}\label{chaaa_frm_periodicity_of_dev_applications}
\tilde{f}_{\widetilde{\mathcal{S}}^{+}_{n}}(\gamma.\tilde{y}^{+})=
\rho^{\mathcal{S}}_{n}(\gamma).\tilde{f}_{\widetilde{\mathcal{S}}^{+}_{n}}(\tilde{y}^{+})
\quad\mbox{and}\quad \tilde{f}_{\widetilde{\mathcal{S}}^{-}_{n}}(\gamma.\tilde{y}^{-})=
\rho^{\mathcal{S}}_{n}(\gamma).\tilde{f}_{\widetilde{\mathcal{S}}^{-}_{n}}(\tilde{y}^{-}),
\quad n\in\mathbb{N}.
\end{equation*}
\end{remark}

Let the metrics $\tilde{h}^{+}_{\lambda}$ and $\tilde{h}^{-}_{\lambda}$ on the universal coverings $\widetilde{\mathcal{S}}^{+}$ and $\widetilde{\mathcal{S}}^{-}$ of the surface $\mathcal{S}$ be the pull-backs of the metrics $h^{+}_{\lambda}$ and $h^{-}_{\lambda}$ on $\mathcal{S}$ defined in the proof of Lemma~\ref{chaaa_lemma_ubound_of_diameters_of_surfaces_Sn}. We are now able to construct the Dirichlet domains $\Delta^{+}\subset\widetilde{\mathcal{S}}^{+}$ and $\Delta^{-}\subset\widetilde{\mathcal{S}}^{-}$ of $\mathcal{S}$ with respect to the metrics $h^{+}_{\lambda}$ and $h^{-}_{\lambda}$ based in the points $\tilde{x}^{+}\in\widetilde{\mathcal{S}}^{+}$ and $\tilde{x}^{-}\in\widetilde{\mathcal{S}}^{-}$, respectively. In what follows we will work with the fundamental domains $\Delta^{+}\subset\widetilde{\mathcal{S}}^{+}$ and $\Delta^{-}\subset\widetilde{\mathcal{S}}^{-}$ of $\mathcal{S}$.

\begin{lemma}\label{chaaa_lemma_ubound_of_diameters_of_fundomains}
For each $n\in\mathbb{N}$ the domains $\Delta^{+}_{n}\stackrel{\mathrm{def}}{=}\tilde{f}_{\widetilde{\mathcal{S}}^{+}_{n}}(\Delta^{+})
\subset\widetilde{\mathcal{S}}^{+}_{n}\subset\mathbb{H}^3$ and $\Delta^{-}_{n}\stackrel{\mathrm{def}}{=}\tilde{f}_{\widetilde{\mathcal{S}}^{-}_{n}}(\Delta^{-})
\subset\widetilde{\mathcal{S}}^{-}_{n}\subset\mathbb{H}^3$ are included in the hyperbolic balls $B(\tilde{x}^{+}_{n}, \delta_{\mathcal{S}})$ and $B(\tilde{x}^{-}_{n}, \delta_{\mathcal{S}})$ of radius $\delta_{\mathcal{S}}$ centered at the points $\tilde{x}^{+}_{n}$ and $\tilde{x}^{-}_{n}$ respectively.
\end{lemma}
\begin{proof}
It suffices to prove this statement for the domain $\Delta^{+}_{n}$.

Assume that the surface $\widetilde{\mathcal{S}}^{+}$ is equipped with the metric $\tilde{h}^{+}_{\lambda}$. It follows from the definition of the Dirichlet domain that the distance from any point $x\in\Delta^{+}\subset\widetilde{\mathcal{S}}^{+}$ to the center $\tilde{x}^{+}$ of $\Delta^{+}$ is not greater than the diameter of the surface $(\mathcal{S},h^{+}_{\lambda})$ which is less than or equal to $\delta_{\mathcal{S}}$ (see the proof of Lemma~\ref{chaaa_lemma_ubound_of_diameters_of_surfaces_Sn}). Recall that the developing map $\tilde{f}_{\widetilde{\mathcal{S}}^{+}_{n}}:\widetilde{\mathcal{S}}^{+}\rightarrow\widetilde{\mathcal{S}}^{+}_{n}$ can be viewed as the identical application from one copy of the surface $\widetilde{\mathcal{S}}^{+}$ equipped with the metric $\tilde{h}^{+}_{\lambda}$ to another copy of $\widetilde{\mathcal{S}}^{+}$ equipped with the metric $\tilde{h}^{+}_{n}$. Also, by the construction made in the proof of Lemma~\ref{chaaa_lemma_ubound_of_diameters_of_surfaces_Sn}, all distances on the surface $\mathcal{S}$ measured in the metric ${h}^{+}_{n}$ do not exceed the corresponding distances on $\mathcal{S}$ in the metric ${h}^{+}_{\lambda}$. Hence, this property is valid for the pull-backs $\tilde{h}^{+}_{n}$ and $\tilde{h}^{+}_{\lambda}$ on $\widetilde{\mathcal{S}}^{+}$ of the metrics $\tilde{h}^{+}_{n}$ and ${h}^{+}_{\lambda}$ on $\mathcal{S}$. Therefore, the distance from any point $v\in\Delta^{+}_{n}=\tilde{f}_{\widetilde{\mathcal{S}}^{+}_{n}}(\Delta^{+})\subset\widetilde{\mathcal{S}}^{+}_{n}$ to the center $\tilde{x}^{+}_{n}=\tilde{f}_{\widetilde{\mathcal{S}}^{+}_{n}}(\tilde{x}^{+})$ of $\Delta^{+}_{n}$ is not greater than $\delta_{\mathcal{S}}$.

To complete the proof we remark that for any couple of points $v_{1},v_{2}\in\widetilde{\mathcal{S}}^{+}_{n}$ the distance between them in the hyperbolic metric of $3$-space $\mathbb{H}^3$ does not exceed the distance between $v_{1}$ and $v_{2}$ in the induced metric $\tilde{h}^{+}_{n}$ on the $2$-surface $\widetilde{\mathcal{S}}^{+}_{n}$: $\mathrm{d}_{\mathbb{H}^3}(v_{1},v_{2})\leq\mathrm{d}_{\tilde{h}^{+}_{n}}(v_{1},v_{2})$.
\end{proof}

Denote by $\widehat{\Delta}^{+}\subset\widetilde{\mathcal{S}}^{+}$ the union of $\Delta^{+}$ with all "neighbor" fundamental domains of $\mathcal{S}$ of the form $\gamma.\Delta^{+}$ for all $\gamma\in\pi_{1}(\mathcal{S})$ such that $\cl\Delta^{+}\cap\cl\gamma.\Delta^{+}\neq\emptyset$. Similarly we define the set $\widehat{\Delta}^{-}\subset\widetilde{\mathcal{S}}^{-}$.

\begin{lemma}\label{chaaa_lemma_ubound_of_diameters_of_fundomain_nhoods}
For each $n\in\mathbb{N}$ the domains $\widehat{\Delta}^{+}_{n}\stackrel{\mathrm{def}}{=}\tilde{f}_{\widetilde{\mathcal{S}}^{+}_{n}}(\widehat{\Delta}^{+})
\subset\widetilde{\mathcal{S}}^{+}_{n}\subset\mathbb{H}^3$ and $\widehat{\Delta}^{-}_{n}\stackrel{\mathrm{def}}{=}\tilde{f}_{\widetilde{\mathcal{S}}^{-}_{n}}(\widehat{\Delta}^{-})
\subset\widetilde{\mathcal{S}}^{-}_{n}\subset\mathbb{H}^3$ are included in the hyperbolic balls $B(\tilde{x}^{+}_{n}, 3\delta_{\mathcal{S}})$ and $B(\tilde{x}^{-}_{n}, 3\delta_{\mathcal{S}})$ of radius $3\delta_{\mathcal{S}}$ centered at the points $\tilde{x}^{+}_{n}$ and $\tilde{x}^{-}_{n}$ correspondingly.
\end{lemma}
\begin{proof}
It suffices to prove this statement for the domain $\widehat{\Delta}^{+}_{n}$.

First, by Lemma~\ref{chaaa_lemma_ubound_of_diameters_of_fundomains}, the domain $\Delta^{+}_{n}$ is inscribed in the ball $B(\tilde{x}^{+}_{n}, \delta_{\mathcal{S}})$. Similarly, for each $\gamma\in\pi_{1}(\mathcal{S})$ the domain $\rho^{\mathcal{S}}_{n}(\gamma).\Delta^{+}_{n}$ (isometric to $\Delta^{+}_{n}$) is inscribed in the ball $B(\rho^{\mathcal{S}}_{n}(\gamma).\tilde{x}^{+}_{n}, \delta_{\mathcal{S}})$. Note that $\widehat{\Delta}^{+}_{n}$ is the union of $\Delta^{+}_{n}$ with the domains of the form $\rho^{\mathcal{S}}_{n}(\gamma).\Delta^{+}_{n}$ such that $\cl \Delta^{+}_{n}\cap\cl \rho^{\mathcal{S}}_{n}(\gamma).\Delta^{+}_{n}\neq\emptyset$, where $\gamma\in\pi_{1}(\mathcal{S})$. Thus, the set $\widehat{\Delta}^{+}_{n}$ is contained in the union $\mathcal{U}_{B}$ of the ball $B(\tilde{x}^{+}_{n}, \delta_{\mathcal{S}})$ and all balls of the type $B(\rho^{\mathcal{S}}_{n}(\gamma).\tilde{x}^{+}_{n}, \delta_{\mathcal{S}})$ such that $B(\rho^{\mathcal{S}}_{n}(\gamma).\tilde{x}^{+}_{n}, \delta_{\mathcal{S}})\cap B(\tilde{x}^{+}_{n}, \delta_{\mathcal{S}})\neq\emptyset$. Clearly, $\mathcal{U}_{B}$ lies entirely inside the ball $B(\tilde{x}^{-}_{n}, 3\delta_{\mathcal{S}})$.
\end{proof}

The following statement is an immediate corollary of Lemmas~\ref{chaaa_lemma_ubound_of_diameters_of_Mn} and~\ref{chaaa_lemma_ubound_of_diameters_of_fundomain_nhoods}.

\begin{lemma}\label{chaaa_lemma_ubound_of_diameters_of_fundomain_nhoods_pair}
For each $n\in\mathbb{N}$ the domains $\widehat{\Delta}^{+}_{n}\stackrel{\mathrm{def}}{=}\tilde{f}_{\widetilde{\mathcal{S}}^{+}_{n}}(\widehat{\Delta}^{+})
\subset\widetilde{\mathcal{S}}^{+}_{n}\subset\mathbb{H}^3$ and $\widehat{\Delta}^{-}_{n}\stackrel{\mathrm{def}}{=}\tilde{f}_{\widetilde{\mathcal{S}}^{-}_{n}}(\widehat{\Delta}^{-})
\subset\widetilde{\mathcal{S}}^{-}_{n}\subset\mathbb{H}^3$ are both included in the hyperbolic balls $B(\tilde{x}^{+}_{n}, 3\delta_{\mathcal{S}}+\delta_{\mathcal{M}})$ and $B(\tilde{x}^{-}_{n}, 3\delta_{\mathcal{S}}+\delta_{\mathcal{M}})$ of radius $3\delta_{\mathcal{S}}+\delta_{\mathcal{M}}$ centered at the points $\tilde{x}^{+}_{n}$ and $\tilde{x}^{-}_{n}$.
\end{lemma}

It is high time to identify the universal coverings $\widetilde{\mathcal{M}}^{\circ}_{n}$ (which are copies of $\mathbb{H}^3$) by supposing that the points $\tilde{x}^{+}_{n}$ coincide for all $n\in\mathbb{N}$. Let us temporarily forget the $3$-dimensional domains $\widetilde{\mathcal{M}}_{n}$ of hyperbolic space $\mathbb{H}^3$ in order to concentrate our attention on the study of properties of the sequences of surfaces $\{\widetilde{\mathcal{S}}^{+}_{n}\}_{n\in\mathbb{N}}$ and $\{\widetilde{\mathcal{S}}^{-}_{n}\}_{n\in\mathbb{N}}$.

Recall the statement of the classical Arzel\`a-Ascoli Theorem.
\begin{theorem}[Theorem~7.5.7 in~\cite{chaaa_Dieudonne1960}, p.~137]\label{chaaa_theorem_arzela-ascoli}
Suppose $F$ is a Banach space and $E$ a compact metric space. In order that a subset $H$ of the Banach space $\mathcal{C}_{F}(E)$ of continuous functions from $E$ to $F$ be relatively compact, necessary and sufficient conditions are that $H$ be equicontinuous and that, for each $x\in E$ the set $H_x$ of all $f(x)$ such that $f\in H$ be relatively compact in $F$.
\end{theorem}
We will apply it in the following
\begin{lemma}\label{chaaa_lemma_convergence_of_domains}
There exist subsequences of functions
$\{\tilde{f}_{\widetilde{\mathcal{S}}^{+}_{n_{k}}}:
\widehat{\Delta}^{+}\rightarrow\mathbb{H}^3\}_{k\in\mathbb{N}}$ and
$\{\tilde{f}_{\widetilde{\mathcal{S}}^{-}_{n_{k}}}:
\widehat{\Delta}^{-}\rightarrow\mathbb{H}^3\}_{k\in\mathbb{N}}$
that converge to continuous functions
$\tilde{f}_{\widetilde{\mathcal{S}}^{+}_{\infty}}:
\widehat{\Delta}^{+}\rightarrow\mathbb{H}^3$
and $\tilde{f}_{\widetilde{\mathcal{S}}^{-}_{\infty}}:
\widehat{\Delta}^{-}\rightarrow\mathbb{H}^3$ correspondingly.
\end{lemma}
\begin{proof}

It suffices to find a converging subsequence of the sequence of functions $\{\tilde{f}_{\widetilde{\mathcal{S}}^{+}_{n}}:
\widehat{\Delta}^{+}\rightarrow\mathbb{H}^3\}_{n\in\mathbb{N}}$. To this purpose we will apply the Arzel\`a-Ascoli Theorem~\ref{chaaa_theorem_arzela-ascoli}.

Let us equip the domain $\widehat{\Delta}^{+}\subset\widetilde{\mathcal{S}}^{+}$ with the restriction $\tilde{h}^{+}_{\lambda}\mid_{\widehat{\Delta}^{+}}$ of the metric $\tilde{h}^{+}_{\lambda}$. Consider the domain $(\widehat{\Delta}^{+},\tilde{h}^{+}_{\lambda}\mid_{\widehat{\Delta}^{+}})$ as a compact metric space $E$ from the statement of Theorem~\ref{chaaa_theorem_arzela-ascoli}; hyperbolic space $\mathbb{H}^3$ as a Banach space $F$; the sequence of functions $\{\tilde{f}_{\widetilde{\mathcal{S}}^{+}_{n}}:
\widehat{\Delta}^{+}\rightarrow\mathbb{H}^3\}_{n\in\mathbb{N}}$ in the space of continuous functions from $(\widehat{\Delta}^{+},\tilde{h}^{+}_{\lambda}\mid_{\widehat{\Delta}^{+}})$ to $\mathbb{H}^3$ as the set $H\subset\mathcal{C}_{F}(E)$.

By Lemma~\ref{chaaa_lemma_ubound_of_diameters_of_fundomain_nhoods_pair}, the images $\widehat{\Delta}^{+}_{n}=\tilde{f}_{\widetilde{\mathcal{S}}^{+}_{n}}(\widehat{\Delta}^{+})
\subset\widetilde{\mathcal{S}}^{+}_{n}\subset\mathbb{H}^3$ of the maps $\tilde{f}_{\widetilde{\mathcal{S}}^{+}_{n}}$, $n\in\mathbb{N}$, are all included in the ball $B(\tilde{x}^{+}_{n}, 3\delta_{\mathcal{S}}+\delta_{\mathcal{M}})$ (recall that we identified all points $\tilde{x}^{+}_{n}\in\mathbb{H}^3$, $n\in\mathbb{N}$). Thus, for each $x\in E$ the set $H_x$ is relatively compact in $F$.

As it was already done in the proof of Lemma~\ref{chaaa_lemma_ubound_of_diameters_of_fundomains}, we consider every developing map $\tilde{f}_{\widetilde{\mathcal{S}}^{+}_{n}}:\widehat{\Delta}^{+}\rightarrow\widetilde{\mathcal{S}}^{+}_{n}$ as the inclusion of the domain $\widehat{\Delta}^{+}$ equipped with the metric $\tilde{h}^{+}_{\lambda}\mid_{\widehat{\Delta}^{+}}$ to the surface $\widetilde{\mathcal{S}}^{+}$ with the metric $\tilde{h}^{+}_{n}$, $n\in\mathbb{N}$. So, for any $\varepsilon>0$ if we pose $\delta:=\varepsilon$ then for every pair of points $x,y\in \widehat{\Delta}^{+}$ such that ${\mathrm{d}}_{\tilde{h}^{+}_{\lambda}}(x,y)<\delta$ it is true that ${\mathrm{d}}_{\mathbb{H}^3}
(\tilde{f}_{\widetilde{\mathcal{S}}^{+}_{n}}(x),\tilde{f}_{\widetilde{\mathcal{S}}^{+}_{n}}(y))\leq
{\mathrm{d}}_{\tilde{h}^{+}_{n}}
(\tilde{f}_{\widetilde{\mathcal{S}}^{+}_{n}}(x),\tilde{f}_{\widetilde{\mathcal{S}}^{+}_{n}}(y))<\varepsilon$ (recall that, by construction, distances measured in the metric $\tilde{h}^{+}_{\lambda}$ are not smaller than the corresponding distances measured in the metric $\tilde{h}^{+}_{n}$), $n\in\mathbb{N}$. Thus, the functions $\{\tilde{f}_{\widetilde{\mathcal{S}}^{+}_{n}}:
\widehat{\Delta}^{+}\rightarrow\mathbb{H}^3\}_{n\in\mathbb{N}}$ are equicontinuous.

Therefore, by the Arzel\`a-Ascoli Theorem~\ref{chaaa_theorem_arzela-ascoli}, there exists a subsequence of functions
$\{\tilde{f}_{\widetilde{\mathcal{S}}^{+}_{n_{k}}}:
\widehat{\Delta}^{+}\rightarrow\mathbb{H}^3\}_{k\in\mathbb{N}}$ that converges to some continuous function
$\tilde{f}_{\widetilde{\mathcal{S}}^{+}_{\infty}}:
\widehat{\Delta}^{+}\rightarrow\mathbb{H}^3$. Similarly we obtain that there exists a subsequence of functions
$\{\tilde{f}_{\widetilde{\mathcal{S}}^{-}_{n_{k}}}:
\widehat{\Delta}^{-}\rightarrow\mathbb{H}^3\}_{k\in\mathbb{N}}$ that converges to some continuous function
$\tilde{f}_{\widetilde{\mathcal{S}}^{-}_{\infty}}:
\widehat{\Delta}^{-}\rightarrow\mathbb{H}^3$.
\end{proof}

\begin{assumption}\label{chaaa_aspt_convergence_of_dev_applications_in_delta}
Further we assume that the sequences of functions $\{\tilde{f}_{\widetilde{\mathcal{S}}^{+}_{n}}:
\widehat{\Delta}^{+}\rightarrow\mathbb{H}^3\}_{n\in\mathbb{N}}$ and
$\{\tilde{f}_{\widetilde{\mathcal{S}}^{-}_{n}}:\widehat{\Delta}^{-}\rightarrow\mathbb{H}^3\}_{n\in\mathbb{N}}$
converge to continuous functions
$\tilde{f}_{\widetilde{\mathcal{S}}^{+}_{\infty}}:\widehat{\Delta}^{+}\rightarrow\mathbb{H}^3$
and $\tilde{f}_{\widetilde{\mathcal{S}}^{-}_{\infty}}:\widehat{\Delta}^{-}\rightarrow\mathbb{H}^3$.
\end{assumption}

\subsubsection{Convergence of the holonomy representations $\{\rho^{\mathcal{S}}_{n}\}_{n\in\mathbb{N}}$ and of the developing maps $\{\tilde{f}_{\widetilde{\mathcal{S}}^{+}_{n}}:
\widetilde{\mathcal{S}}^{+}\rightarrow\mathbb{H}^3\}_{n\in\mathbb{N}}$ and $\{\tilde{f}_{\widetilde{\mathcal{S}}^{-}_{n}}:
\widetilde{\mathcal{S}}^{-}\rightarrow\mathbb{H}^3\}_{n\in\mathbb{N}}$} \label{chaaa_sec_convergence_of_isometries_of_H3}

Now we need to derive several properties of the holonomy representations $\rho^{\mathcal{S}}_{n}({\pi}_{1}(\mathcal{S}))$, $n\in\mathbb{N}$.

\begin{lemma}\label{chaaa_lemma_unique_determination_of_an_isometry}
Given two points $y^{1},y^{2}\in\mathbb{H}^3$ together with orthogonal bases $\{e^{1},e^{2},e^{3}\}$
and $\{\hat{e}^{1},\hat{e}^{2},\hat{e}^{3}\}$ of the tangent spaces
$T_{y^{1}}\mathbb{H}^3$ and $T_{y^{2}}\mathbb{H}^3$, there is a unique isometry
$\vartheta\in\mathcal{I}(\mathbb{H}^3)$ such that $y^{2}=\vartheta.y^{1}$ and $\hat{e}^{i}=d_{y^{1}}\vartheta(e^{i})$, $i=1,...,3$.
\end{lemma}
\begin{proof} Following Chapter~1, \S~1.5 in \cite[p.~13]{chaaa_Vinberg1988} let us recall the construction of the hyperboloid model $\mathbb{I}^3$ of hyperbolic space $\mathbb{H}^3$. Denoting the coordinates in space $\mathbb{R}^{4}$ by $x_{0}, x_{1}, x_{2}, x_{3}$, we introduce the Minkowski scalar product in $\mathbb{R}^{4}$ by the formula
\begin{equation}\label{chaaa_frm_minkowski_scalar_product}
(x,y)_{M}=-x_{0}y_{0}+x_{1}y_{1}+x_{2}y_{2}+x_{3}y_{3},
\end{equation}
which turns $\mathbb{R}^{4}$ into a pseudo-Euclidean vector space, denoted by $\mathbb{R}^{3,1}$.

A basis $\{u^{0},u^{1},u^{2},u^{3}\}\subset\mathbb{R}^{3,1}$ is said to be \emph{orthonormal} if $(u^{0},u^{0})_{M}=-1$, $(u^{i},u^{i})_{M}=1$ for $i\neq0$, and $(u^{i},u^{j})_{M}=0$ for $i\neq j$. For example, the standard basis
\begin{equation}\label{chaaa_frm_minkowski_standard_basis}
\{\epsilon^{0},\epsilon^{1},\epsilon^{2},\epsilon^{3}\}=
\Bigg{\{} \begin{pmatrix} 1\\0\\0\\0 \end{pmatrix},
\begin{pmatrix} 0\\1\\0\\0 \end{pmatrix},
\begin{pmatrix} 0\\0\\1\\0 \end{pmatrix},
\begin{pmatrix} 0\\0\\0\\1 \end{pmatrix}\Bigg{\}}\subset\mathbb{R}^{3,1}
\end{equation}
is orthonormal.

Each pseudo-orthogonal (i.e. preserving the above scalar product) transformation of $\mathbb{R}^{3,1}$ takes an open cone of time-like vectors
\begin{equation*}\label{chaaa_frm_cone_of_timelike_vectors}
\mathfrak{C}=\{x\in\mathbb{R}^{3,1}: (x,x)_{M}<0\}
\end{equation*}
consisting of two connected components
\begin{equation*}\label{chaaa_frm_timelike_cone_components}
\mathfrak{C}^{+}=\{x\in\mathfrak{C}: x_{0}>0\},\quad\mathfrak{C}^{-}=\{x\in\mathfrak{C}: x_{0}<0\}
\end{equation*}
onto itself. Denote by $O(3,1)$ the group of all pseudo-orthogonal transformations of
space $\mathbb{R}^{3,1}$, and by $O'(3,1)$ its subgroup of index $2$ consisting of those pseudo
orthogonal transformations which map each connected component of the cone $\mathfrak{C}$ onto itself.

Using notation developed in \S~A.1 \cite[p.~1]{chaaa_BP2003} we remind that the manifold
\begin{equation*}\label{chaaa_frm_upper_hyperboloid}
\mathbb{I}^3=\{x\in\mathbb{R}^{3,1}: (x,x)_{M}=-1, x_{0}>0\}
\end{equation*}
with the metric induced by the pseudo-Euclidean metric~(\ref{chaaa_frm_minkowski_scalar_product}) is called the hyperboloid model $\mathbb{I}^3$ of hyperbolic space $\mathbb{H}^3$, and the restrictions of the elements of $O'(3,1)$ on $\mathbb{I}^3$ form the group $\mathcal{I}(\mathbb{H}^3)$ of all isometries of $\mathbb{H}^3$.

Again, by Chapter~1, \S~1.5 in \cite[p.~13]{chaaa_Vinberg1988}, for any $x\in\mathbb{I}^3$ we can naturally identify the tangent space $T_{x}\mathbb{I}^3$ with the orthogonal complement of the vector $x$ in space $\mathbb{R}^{3,1}$, which is a $3$-dimensional Euclidean space (with respect to the same scalar product). If $\{u^{1},u^{2},u^{3}\}$ is an orthonormal basis in it, then $\{x, u^{1},u^{2},u^{3}\}$ is an orthonormal basis in the space $\mathbb{R}^{3,1}$.

Obviously, the vector $\epsilon^{0}$ of the standard basis~(\ref{chaaa_frm_minkowski_standard_basis}) $\mathbb{R}^{3,1}$ lies in $\mathbb{I}^3$ and the vectors $\{\epsilon^{1},\epsilon^{2},\epsilon^{3}\}$ defined in~(\ref{chaaa_frm_minkowski_standard_basis}) form an orthonormal basis of the tangent space $T_{\epsilon^{0}}\mathbb{I}^3$. Also, according to a fact mentioned in the previous paragraph, the sets of four vectors
$\{y^{1},e^{1},e^{2},e^{3}\}\subset\mathbb{R}^{3,1}$ and $\{y^{2},\hat{e}^{1},\hat{e}^{2},\hat{e}^{3}\}\subset\mathbb{R}^{3,1}$ from the statement of Lemma~\ref{chaaa_lemma_unique_determination_of_an_isometry} are orthonormal bases of $\mathbb{R}^{3,1}$. Define the linear transformations $\vartheta_{1}$ and $\vartheta_{2}$ of $\mathbb{R}^{3,1}$ determined by their $4\times4$-real matrices $M^{\vartheta}_{1}\stackrel{\mathrm{def}}{=}(y^{1},e^{1},e^{2},e^{3})$ and $M^{\vartheta}_{2}\stackrel{\mathrm{def}}{=}(y^{2},\hat{e}^{1},\hat{e}^{2},\hat{e}^{3})$ with the columns consisting of the coordinates of the corresponding vectors in the standard basis of $\mathbb{R}^{3,1}$. A direct calculation shows the transformations $\vartheta_{1}$ and $\vartheta_{2}$ send the standard base to the orthonormal bases $\{y^{1},e^{1},e^{2},e^{3}\}$ and $\{y^{2},\hat{e}^{1},\hat{e}^{2},\hat{e}^{3}\}$ of $\mathbb{R}^{3,1}$, respectively. Moreover, we know that the vectors $\epsilon^{0}$, $y^{1}$, and $y^{2}$ belong to the upper cone $\mathfrak{C}^{+}$. Hence, $\vartheta_{1}$ and $\vartheta_{2}$ are elements of the group $O'(3,1)$, and we can take the transformation $\vartheta$ from the statement of Lemma~\ref{chaaa_lemma_unique_determination_of_an_isometry} to be equal to $\vartheta_{2}[\vartheta_{1}]^{-1}$.
\end{proof}

\textbf{Definition.} Given a sequence of hyperbolic isometries $\{\vartheta_{n}\in\mathcal{I}(\mathbb{H}^3)\}_{n\in\mathbb{N}}$ determined by points $y^{1}_{n},y^{2}_{n}\in\mathbb{H}^3$ and orthogonal bases $\{e^{1}_{n},e^{2}_{n},e^{3}_{n}\}$, $\{\hat{e}^{1}_{n},\hat{e}^{2}_{n},\hat{e}^{3}_{n}\}$ of the tangent spaces
$T_{y^{1}_{n}}\mathbb{H}^3$ and $T_{y^{2}_{n}}\mathbb{H}^3$, we say that the isometries $\{\vartheta_{n}\}_{n\in\mathbb{N}}$ \emph{converge} to an isometry $\vartheta_{\infty}\in\mathcal{I}(\mathbb{H}^3)$ \emph{in the sense of Lemma}~\ref{chaaa_lemma_unique_determination_of_an_isometry} if the sequences of base points $\{y^{1}_{n}\}_{n\in\mathbb{N}}$, $\{y^{2}_{n}\}_{n\in\mathbb{N}}$ converge to points $y^{1}_{\infty},y^{2}_{\infty}\in\mathbb{H}^3$ and the sequences of orthogonal bases $\{e^{1}_{n},e^{2}_{n},e^{3}_{n}\}_{n\in\mathbb{N}}$, $\{\hat{e}^{1}_{n},\hat{e}^{2}_{n},\hat{e}^{3}_{n}\}_{n\in\mathbb{N}}$ converge to orthogonal bases $\{e^{1}_{\infty},e^{2}_{\infty},e^{3}_{\infty}\}$, $\{\hat{e}^{1}_{\infty},\hat{e}^{2}_{\infty},\hat{e}^{3}_{\infty}\}$ of the tangent spaces
$T_{y^{1}_{\infty}}\mathbb{H}^3$ and $T_{y^{2}_{\infty}}\mathbb{H}^3$, and the above-mentioned limits define uniquely the isometry $\vartheta_{\infty}$. Denote a convergence of isometries in the sense of Lemma~\ref{chaaa_lemma_unique_determination_of_an_isometry} by $\vartheta_{n}\Rightarrow\vartheta_{\infty}$ as $n\rightarrow\infty$.

\textbf{Definition.} We say that hyperbolic isometries $\{\vartheta_{n}\in\mathcal{I}(\mathbb{H}^3)\}_{n\in\mathbb{N}}$ \emph{converge }to an isometry $\vartheta_{\infty}\in\mathcal{I}(\mathbb{H}^3)$ \emph{in a "weak" sense} if for any point $y\in\mathbb{H}^3$ the sequence $\{\vartheta_{n}.y\}_{n\in\mathbb{N}}$ converges to the point $\vartheta_{\infty}.y\in\mathbb{H}^3$ as $n\rightarrow\infty$. Denote a "weak" convergence of isometries by $\vartheta_{n}\xrightarrow[n\rightarrow\infty]{}\vartheta_{\infty}$.

\begin{lemma}\label{chaaa_lemma_types_of_convergence_of_isometries}
Given a collection of hyperbolic isometries $\{\vartheta_{n}\in\mathcal{I}(\mathbb{H}^3)\}_{n=1}^{\infty}$, $\vartheta_{n}\Rightarrow\vartheta_{\infty}$ as $n\rightarrow\infty$ if and only if $\vartheta_{n}\xrightarrow[n\rightarrow\infty]{}\vartheta_{\infty}$.
\end{lemma}

\begin{proof} A hyperbolic isometry $\vartheta:\mathbb{H}^3\rightarrow\mathbb{H}^3$ which sends any $y\in\mathbb{H}^3$ to the point $\vartheta.y\in\mathbb{H}^3$ can be interpreted as a linear transformation of Minkowski space $\mathbb{R}^{3,1}$ as it was mentioned in the proof of Lemma~\ref{chaaa_lemma_unique_determination_of_an_isometry}. Therefore, $\vartheta(y)$ depends continuously on $y\in\mathbb{H}^3$.

Suppose that $\vartheta_{n}\Rightarrow\vartheta_{\infty}$ as $n\rightarrow\infty$. By construction, a transformation $\vartheta\in\mathcal{I}(\mathbb{H}^3)$ from Lemma~\ref{chaaa_lemma_unique_determination_of_an_isometry} depends continuously on the parameters $y^{1},y^{2}\in\mathbb{H}^3$, $\{e^{1},e^{2},e^{3}\}\subset T_{y^{1}}\mathbb{H}^3$,
and $\{\hat{e}^{1},\hat{e}^{2},\hat{e}^{3}\}\subset T_{y^{2}}\mathbb{H}^3$. Hence, for any point $y\in\mathbb{H}^3$ the sequence $\{\vartheta_{n}.y\}_{n\in\mathbb{N}}$ converges to the point $\vartheta_{\infty}.y\in\mathbb{H}^3$ as $n\rightarrow\infty$, which means that the convergence of the isometries $\{\vartheta_{n}\}_{n\in\mathbb{N}}$ in the sense of Lemma~\ref{chaaa_lemma_unique_determination_of_an_isometry} implies also the "weak" convergence of these isometries to $\vartheta_{\infty}$.

Suppose now that $\vartheta_{n}\xrightarrow[n\rightarrow\infty]{}\vartheta_{\infty}$. Being a linear transformation of Minkowski space $\mathbb{R}^{3,1}$, the hyperbolic isometries $\{\vartheta_{n}\in\mathcal{I}(\mathbb{H}^3)\}_{n=1}^{\infty}$ are represented in the standard basis of $\mathbb{R}^{3,1}$ by the $4\times4$-real matrices $M^{\vartheta_{n}}\stackrel{\mathrm{def}}{=}({\vartheta}^{0}_{n},{\vartheta}^{1}_{n},{\vartheta}^{2}_{n},{\vartheta}^{3}_{n})$,
where ${\vartheta}^{k}_{n}$, $k=0,1,2,3$, are the columns of $M^{\vartheta}_{n}$.

Let $P_{0}\stackrel{\mathrm{def}}{=}(1,0,0,0)^{T}\in\mathbb{I}^{3}\subset\mathbb{R}^{3,1}$. The "weak" convergence of the isometries $\{\vartheta_{n}\}_{n\in\mathbb{N}}$ at the point $P_{0}$ means that $M^{\vartheta_{n}}.P_{0}\xrightarrow[n\rightarrow\infty]{}M^{\vartheta_{\infty}}.P_{0}$, i.e.
\begin{equation}\label{chaaa_frm_first_column_convergence}
{\vartheta}^{0}_{n}\xrightarrow[n\rightarrow\infty]{}{\vartheta}^{0}_{\infty}.
\end{equation}
Let $P_{1}\stackrel{\mathrm{def}}{=}(\sqrt{2},1,0,0)^{T}\in\mathbb{I}^{3}\subset\mathbb{R}^{3,1}$. The "weak" convergence of the isometries $\{\vartheta_{n}\}_{n\in\mathbb{N}}$ at the point $P_{1}$ means that $M^{\vartheta_{n}}.P_{1}\xrightarrow[n\rightarrow\infty]{}M^{\vartheta_{\infty}}.P_{1}$, i.e. $\sqrt{2}{\vartheta}^{0}_{n}+{\vartheta}^{1}_{n}\xrightarrow[n\rightarrow\infty]{}
\sqrt{2}{\vartheta}^{0}_{\infty}+{\vartheta}^{0}_{\infty}$. Taking into account~(\ref{chaaa_frm_first_column_convergence}), we obtain that ${\vartheta}^{1}_{n}\xrightarrow[n\rightarrow\infty]{}{\vartheta}^{1}_{\infty}$. Similarly we get that ${\vartheta}^{2}_{n}\xrightarrow[n\rightarrow\infty]{}{\vartheta}^{2}_{\infty}$ and ${\vartheta}^{3}_{n}\xrightarrow[n\rightarrow\infty]{}{\vartheta}^{3}_{\infty}$. Thus, the "weak" convergence of the isometries $\{\vartheta_{n}\}_{n\in\mathbb{N}}$ to $\vartheta_{\infty}$ as $n\rightarrow\infty$ implies also their convergence in the sense of Lemma~\ref{chaaa_lemma_unique_determination_of_an_isometry}.
\end{proof}

\begin{lemma}\label{chaaa_lemma_converging_subsequence_of_isometries}
For each $n\in\mathbb{N}$ let a pair of surfaces $\widetilde{\mathcal{S}}^{+}_{n}$ and $\widetilde{\mathcal{S}}^{-}_{n}\subset\mathbb{H}^3$ (which are the images of developing maps $\tilde{f}_{\widetilde{\mathcal{S}}^{+}_{n}}:\widetilde{\mathcal{S}}^{+}\rightarrow\widetilde{\mathcal{S}}^{+}_{n}$ and $\tilde{f}_{\widetilde{\mathcal{S}}^{-}_{n}}:\widetilde{\mathcal{S}}^{-}\rightarrow\widetilde{\mathcal{S}}^{-}_{n}$) be invariant under the actions of a quasi-Fuchsian group $\rho^{\mathcal{S}}_{n}({\pi}_{1}(\mathcal{S}))$ of isometries of $\mathbb{H}^3$. Suppose in addition that the restrictions of the developing maps $\{\tilde{f}_{\widetilde{\mathcal{S}}^{+}_{n}}:
\widehat{\Delta}^{+}\rightarrow\mathbb{H}^3\}_{n\in\mathbb{N}}$ and
$\{\tilde{f}_{\widetilde{\mathcal{S}}^{-}_{n}}:\widehat{\Delta}^{-}\rightarrow\mathbb{H}^3\}_{n\in\mathbb{N}}$
on the domains $\widehat{\Delta}^{+}\subset\widetilde{\mathcal{S}}^{+}$ and $\widehat{\Delta}^{-}\subset\widetilde{\mathcal{S}}^{-}$ defined in Section~\ref{chaaa_sec_Arzela-Ascoli_theorem_application} converge to continuous functions
$\tilde{f}_{\widetilde{\mathcal{S}}^{+}_{\infty}}:\widehat{\Delta}^{+}\rightarrow\mathbb{H}^3$
and $\tilde{f}_{\widetilde{\mathcal{S}}^{-}_{\infty}}:\widehat{\Delta}^{-}\rightarrow\mathbb{H}^3$. Then there is a sequence of positive integers $n_{k}\xrightarrow[k\rightarrow\infty]{}\infty$ such that the morphisms $\{\rho^{\mathcal{S}}_{n_{k}}:{\pi}_{1}(\mathcal{S})\rightarrow\mathcal{I}(\mathbb{H}^3)\}_{k\in\mathbb{N}}$ converge to a morphism $\rho^{\mathcal{S}}_{\infty}:{\pi}_{1}(\mathcal{S})\rightarrow\mathcal{I}(\mathbb{H}^3)$ in the sense of Lemma~\ref{chaaa_lemma_unique_determination_of_an_isometry}, i.e. for every $\gamma\in{\pi}_{1}(\mathcal{S})$ there exists a hyperbolic isometry which we denote by $\rho^{\mathcal{S}}_{\infty}(\gamma)$ such that $\rho^{\mathcal{S}}_{n_{k}}(\gamma)\Rightarrow\rho^{\mathcal{S}}_{\infty}(\gamma)$ as $k\rightarrow\infty$.
\end{lemma}
\begin{proof}
First, we prove that there is a sequence of positive integers $n_{k}\xrightarrow[k\rightarrow\infty]{}\infty$ such that for any generator $\gamma_{i}$ of the group ${\pi}_{1}(\mathcal{S})$ together with its inverse element $\gamma^{-1}_{i}\in{\pi}_{1}(\mathcal{S})$, $i=1,...,l$, the subsequences of isometries $\rho^{\mathcal{S}}_{n_{k}}(\gamma_{i})\Rightarrow\rho^{\mathcal{S}}_{\infty}(\gamma_{i})$ and $\rho^{\mathcal{S}}_{n_{k}}(\gamma^{-1}_{i})\Rightarrow\rho^{\mathcal{S}}_{\infty}(\gamma^{-1}_{i})$ converge as $k\rightarrow\infty$.

Indeed, since for any $i=1,...,l$ points $\tilde{x}^{+}$, $\gamma_{i}.\tilde{x}^{+}$, and $\gamma^{-1}_{i}.\tilde{x}^{+}$ lie inside $\widehat{\Delta}^{+}\subset\widetilde{\mathcal{S}}^{+}$ by construction, and because of convergence of the developing maps $\{\tilde{f}_{\widetilde{\mathcal{S}}^{+}_{n}}:
\widehat{\Delta}^{+}\rightarrow\mathbb{H}^3\}_{n\in\mathbb{N}}$ to a continuous function $\tilde{f}_{\widetilde{\mathcal{S}}^{+}_{\infty}}:\widehat{\Delta}^{+}\rightarrow\mathbb{H}^3$, we know that the sequences of points $\tilde{x}^{+}_{n}(=\tilde{f}_{\widetilde{\mathcal{S}}^{+}_{n}}(\tilde{x}^{+}))
\xrightarrow[n\rightarrow\infty]{}
\tilde{x}^{+}_{\infty}(=\tilde{f}_{\widetilde{\mathcal{S}}^{+}_{\infty}}(\tilde{x}^{+}))$, $\rho^{\mathcal{S}}_{n}(\gamma_{i}).\tilde{x}^{+}_{n}
(=\rho^{\mathcal{S}}_{n}(\gamma_{i}).\tilde{f}_{\widetilde{\mathcal{S}}^{+}_{n}}(\tilde{x}^{+})
=\tilde{f}_{\widetilde{\mathcal{S}}^{+}_{n}}(\gamma_{i}.\tilde{x}^{+}))\xrightarrow[n\rightarrow\infty]{}
\rho^{\mathcal{S}}_{\infty}(\gamma_{i}).\tilde{x}^{+}_{\infty}
(=\rho^{\mathcal{S}}_{\infty}(\gamma_{i}).\tilde{f}_{\widetilde{\mathcal{S}}^{+}_{\infty}}(\tilde{x}^{+})
=\tilde{f}_{\widetilde{\mathcal{S}}^{+}_{\infty}}(\gamma_{i}.\tilde{x}^{+}))$, and
$[\rho^{\mathcal{S}}_{n}(\gamma_{i})]^{-1}.\tilde{x}^{+}_{n}
(=\rho^{\mathcal{S}}_{n}(\gamma^{-1}_{i}).\tilde{f}_{\widetilde{\mathcal{S}}^{+}_{n}}(\tilde{x}^{+})
=\tilde{f}_{\widetilde{\mathcal{S}}^{+}_{n}}(\gamma^{-1}_{i}.\tilde{x}^{+}))\xrightarrow[n\rightarrow\infty]{}
[\rho^{\mathcal{S}}_{\infty}(\gamma_{i})]^{-1}.\tilde{x}^{+}_{\infty}
(=\rho^{\mathcal{S}}_{\infty}(\gamma^{-1}_{i}).\tilde{f}_{\widetilde{\mathcal{S}}^{+}_{\infty}}(\tilde{x}^{+})
=\tilde{f}_{\widetilde{\mathcal{S}}^{+}_{\infty}}(\gamma^{-1}_{i}.\tilde{x}^{+}))$ converge in $\mathbb{H}^3$.

Also we know that for each $n\in\mathbb{N}$ and for every $i=1,...,l$, the differential $d_{\tilde{x}^{+}_{n}}\rho^{\mathcal{S}}_{n}(\gamma_{i})$
sends an orthonormal base $\{e^{n,i}_{1},e^{n,i}_{2},e^{n,i}_{3}\}$ of the tangent space $T_{\tilde{x}^{+}_{n}}\mathbb{H}^3$ to an orthonormal base $\{\hat{e}^{n,i}_{1},\hat{e}^{n,i}_{2},\hat{e}^{n,i}_{3}\}$
of $T_{\rho^{\mathcal{S}}_{n}(\gamma_{i}).\tilde{x}^{+}_{n}}\mathbb{H}^3$ (recall that, by constructions all the points $\tilde{x}^{+}_{n}$, $n\in\mathbb{N}$ coincide). Since the subsequences $\{{e}^{n,i}_{j}\}_{n\in\mathbb{N}}$,
$\{\hat{e}^{n,i}_{j}\}_{n\in\mathbb{N}}$, $j=1,2,3$, $i=1,...,l$, of unitary vectors are bounded, there exists
a sequence of positive integers $n_{k}\xrightarrow[k\rightarrow\infty]{}\infty$ such that the pairs of subsequences of orthonormal bases $\{e^{n_{k},i}_{1},e^{n_{k},i}_{2},e^{n_{k},i}_{3}\}_{k\in\mathbb{N}}$ and
$\{\hat{e}^{n_{k},i}_{1},\hat{e}^{n_{k},i}_{2},\hat{e}^{n_{k},i}_{3}\}_{k\in\mathbb{N}}$
converge all together ($i=1,...,l$) ensemble to orthonormal bases
$\{e^{\infty,i}_{1},e^{\infty,i}_{2},e^{\infty,i}_{3}\}$ and
$\{\hat{e}^{\infty,i}_{1},\hat{e}^{\infty,i}_{2},\hat{e}^{\infty,i}_{3}\}$.
Hence, by Lemma~\ref{chaaa_lemma_unique_determination_of_an_isometry}, there exists a hyperbolic isometry that we denote by $\rho^{\mathcal{S}}_{\infty}(\gamma_{i})$ which sends the point $\tilde{x}^{+}_{\infty}$ to the point  $\rho^{\mathcal{S}}_{\infty}(\gamma_{i}).\tilde{x}^{+}_{\infty}$ defined above, and which differential $d_{\tilde{x}^{+}_{\infty}}\rho^{\mathcal{S}}_{\infty}(\gamma_{i})$
sends an orthonormal base $\{e^{\infty,i}_{1},e^{\infty,i}_{2},e^{\infty,i}_{3}\}$ of the tangent space $T_{\tilde{x}^{+}_{\infty}}\mathbb{H}^3$ to an orthonormal base $\{\hat{e}^{\infty,i}_{1},\hat{e}^{\infty,i}_{2},\hat{e}^{\infty,i}_{3}\}$
of $T_{\rho^{\mathcal{S}}_{\infty}(\gamma_{i}).\tilde{x}^{+}_{\infty}}\mathbb{H}^3$ such that $\rho^{\mathcal{S}}_{n_{k}}(\gamma_{i})\Rightarrow\rho^{\mathcal{S}}_{\infty}(\gamma_{i})$ as $k\rightarrow\infty$.

Secondly, we derive that for any element $\gamma\in{\pi}_{1}(\mathcal{S})$ the subsequences of isometries $\rho^{\mathcal{S}}_{n_{k}}(\gamma)\Rightarrow\rho^{\mathcal{S}}_{\infty}(\gamma)$ converges as $k\rightarrow\infty$. Indeed, every $\gamma\in{\pi}_{1}(\mathcal{S})$ can be decomposed in a product of generators of ${\pi}_{1}(\mathcal{S})$ together with their inverse elements, for which the demanded convergence has already been shown.
\end{proof}

\begin{assumption}\label{chaaa_aspt_convergence_of_isometry_groups}
Further we assume that the sequence of holonomy representations $\{\rho^{\mathcal{S}}_{n}:{\pi}_{1}(\mathcal{S})\rightarrow\mathcal{I}(\mathbb{H}^3)\}_{n\in\mathbb{N}}$ (where the groups $\rho^{\mathcal{S}}_{n}({\pi}_{1}(\mathcal{S}))$ of isometries of $\mathbb{H}^3$ are quasi-Fuchsian) converges to a holonomy representation $\rho^{\mathcal{S}}_{\infty}:{\pi}_{1}(\mathcal{S})\rightarrow\mathcal{I}(\mathbb{H}^3)$ (where $\rho^{\mathcal{S}}_{\infty}({\pi}_{1}(\mathcal{S}))$ is a discrete group of isometries of $\mathbb{H}^3$) in the sense of Lemma~\ref{chaaa_lemma_unique_determination_of_an_isometry} as $n\rightarrow\infty$.
\end{assumption}

Let us now prove the following property of the functions
$\tilde{f}_{\widetilde{\mathcal{S}}^{+}_{\infty}}:\widehat{\Delta}^{+}\rightarrow\mathbb{H}^3$
and $\tilde{f}_{\widetilde{\mathcal{S}}^{-}_{\infty}}:\widehat{\Delta}^{-}\rightarrow\mathbb{H}^3$ with respect to the group of isometries $\rho^{\mathcal{S}}_{\infty}(\pi_{1}(\mathcal{S}))$ of space $\mathbb{H}^3$.

\begin{remark}\label{chaaa_remark_periodicity_of_limit_dev_applications}
If for a pair of points $\tilde{y}^{+}_{1},\tilde{y}^{+}_{2}\in\widehat{\Delta}^{+}$ there exists a transformation $\gamma^{+}\in\pi_{1}(\mathcal{S})$ such that $\tilde{y}^{+}_{2}=\gamma^{+}.\tilde{y}^{+}_{1}$, then the following equality holds:
\begin{equation}\label{chaaa_frm_periodicity_of_dev+_applications}
\tilde{f}_{\widetilde{\mathcal{S}}^{+}_{\infty}}(\tilde{y}^{+}_{2})=
\rho^{\mathcal{S}}_{\infty}(\gamma^{+}).\tilde{f}_{\widetilde{\mathcal{S}}^{+}_{\infty}}(\tilde{y}^{+}_{1}).
\end{equation}
Similarly, if for a pair of points $\tilde{y}^{-}_{1},\tilde{y}^{-}_{2}\in\widehat{\Delta}^{-}$ there exists a transformation $\gamma^{-}\in\pi_{1}(\mathcal{S})$ such that $\tilde{y}^{-}_{2}=\gamma^{-}.\tilde{y}^{-}_{1}$, then
\begin{equation*}\label{chaaa_frm_periodicity_of_dev-_applications}
\tilde{f}_{\widetilde{\mathcal{S}}^{-}_{\infty}}(\tilde{y}^{-}_{2})=
\rho^{\mathcal{S}}_{\infty}(\gamma^{-}).\tilde{f}_{\widetilde{\mathcal{S}}^{-}_{\infty}}(\tilde{y}^{-}_{1}).
\end{equation*}
\end{remark}
\begin{proof}
It suffices to prove the formula~(\ref{chaaa_frm_periodicity_of_dev+_applications}).

By Remark~\ref{chaaa_remark_periodicity_of_dev_applications}, the relation
\begin{equation}\label{chaaa_frm_periodicity_of_dev_applications_in_delta}
\tilde{f}_{\widetilde{\mathcal{S}}^{+}_{n}}(\tilde{y}^{+}_{2})=
\rho^{\mathcal{S}}_{n}(\gamma^{+}).\tilde{f}_{\widetilde{\mathcal{S}}^{+}_{n}}(\tilde{y}^{+}_{1})
\end{equation}
is valid for all $n\in\mathbb{N}$.

By Assumption~\ref{chaaa_aspt_convergence_of_dev_applications_in_delta},
the sequence $\{\tilde{f}_{\widetilde{\mathcal{S}}^{+}_{n}}(\tilde{y}^{+}_{2})\}_{n\in\mathbb{N}}\subset\mathbb{H}^3$ converges to the point $\tilde{f}_{\widetilde{\mathcal{S}}^{+}_{\infty}}(\tilde{y}^{+}_{2})\in\mathbb{H}^3$. Hence, taking into account the formula~(\ref{chaaa_frm_periodicity_of_dev_applications_in_delta}) we see that in order to prove the equality~(\ref{chaaa_frm_periodicity_of_dev+_applications}) we need to demonstrate the convergence of the sequence $\{\rho^{\mathcal{S}}_{n}(\gamma^{+}).\tilde{f}_{\widetilde{\mathcal{S}}^{+}_{n}}(\tilde{y}^{+}_{1})\}_{n\in\mathbb{N}}
\subset\mathbb{H}^3$ to the point $\rho^{\mathcal{S}}_{\infty}(\gamma^{+}).\tilde{f}_{\widetilde{\mathcal{S}}^{+}_{\infty}}(\tilde{y}^{+}_{1})$, i.e.,
fixing $\varepsilon>0$, we ought to find such $n_{0}\in\mathbb{N}$ that
\begin{equation}\label{chaaa_frm_eps-n_relation_to_prove}
\forall n>n_{0}\quad\mbox{the inequality}\quad\mathrm{d}_{\mathbb{H}^3}
(\rho^{\mathcal{S}}_{n}(\gamma^{+}).\tilde{f}_{\widetilde{\mathcal{S}}^{+}_{n}}(\tilde{y}^{+}_{1}),
\rho^{\mathcal{S}}_{\infty}(\gamma^{+}).\tilde{f}_{\widetilde{\mathcal{S}}^{+}_{\infty}}(\tilde{y}^{+}_{1}))
<\varepsilon\quad\mbox{holds}.
\end{equation}

First, by the above-mentioned Assumption~\ref{chaaa_aspt_convergence_of_dev_applications_in_delta},
the sequence $\{\tilde{f}_{\widetilde{\mathcal{S}}^{+}_{n}}(\tilde{y}^{+}_{1})\}_{n\in\mathbb{N}}\subset\mathbb{H}^3$ converges to the point $\tilde{f}_{\widetilde{\mathcal{S}}^{+}_{\infty}}(\tilde{y}^{+}_{1})\in\mathbb{H}^3$. Therefore,
\begin{equation}\label{chaaa_frm_eps-n_relation_for_y1+}
\exists n_{1}\in\mathbb{N}:\forall n>n_{1}\quad\mbox{the inequality}\quad\mathrm{d}_{\mathbb{H}^3}
(\tilde{f}_{\widetilde{\mathcal{S}}^{+}_{n}}(\tilde{y}^{+}_{1}),
\tilde{f}_{\widetilde{\mathcal{S}}^{+}_{\infty}}(\tilde{y}^{+}_{1}))<\frac{\varepsilon}{2}\quad\mbox{is valid}.
\end{equation}

Also, by Assumption~\ref{chaaa_aspt_convergence_of_isometry_groups}, $\rho^{\mathcal{S}}_{n}(\gamma^{+})\Rightarrow\rho^{\mathcal{S}}_{\infty}(\gamma^{+})$ as $n\rightarrow\infty$. Hence, by Lemma~\ref{chaaa_lemma_types_of_convergence_of_isometries}, the sequence of points $\{\rho^{\mathcal{S}}_{n}(\gamma^{+}).\tilde{f}_{\widetilde{\mathcal{S}}^{+}_{\infty}}(\tilde{y}^{+}_{1})\}_{n\in\mathbb{N}}\subset\mathbb{H}^3$ converges to the point $\rho^{\mathcal{S}}_{\infty}(\gamma^{+}).\tilde{f}_{\widetilde{\mathcal{S}}^{+}_{\infty}}(\tilde{y}^{+}_{1})\in\mathbb{H}^3$, i.e.
\begin{equation}\label{chaaa_frm_eps-n_relation_for_isometry_convergence}
\exists n_{2}\in\mathbb{N}:\forall n>n_{2}\quad\mbox{the inequality}\quad\mathrm{d}_{\mathbb{H}^3}
(\rho^{\mathcal{S}}_{n}(\gamma^{+}).\tilde{f}_{\widetilde{\mathcal{S}}^{+}_{\infty}}(\tilde{y}^{+}_{1}),
\rho^{\mathcal{S}}_{\infty}(\gamma^{+}).\tilde{f}_{\widetilde{\mathcal{S}}^{+}_{\infty}}(\tilde{y}^{+}_{1}))
<\frac{\varepsilon}{2}\quad\mbox{is true}.
\end{equation}

Applying the triangle inequality, we get:
\begin{equation*}\label{chaaa_frm_triangle_inequality_for_dev_apps_periodicity_1}
\mathrm{d}_{\mathbb{H}^3}
(\rho^{\mathcal{S}}_{n}(\gamma^{+}).\tilde{f}_{\widetilde{\mathcal{S}}^{+}_{n}}(\tilde{y}^{+}_{1}),
\rho^{\mathcal{S}}_{\infty}(\gamma^{+}).\tilde{f}_{\widetilde{\mathcal{S}}^{+}_{\infty}}(\tilde{y}^{+}_{1}))\leq
\end{equation*}
\begin{equation}\label{chaaa_frm_triangle_inequality_for_dev_apps_periodicity_2}
\mathrm{d}_{\mathbb{H}^3}
(\rho^{\mathcal{S}}_{n}(\gamma^{+}).\tilde{f}_{\widetilde{\mathcal{S}}^{+}_{n}}(\tilde{y}^{+}_{1}),
\rho^{\mathcal{S}}_{n}(\gamma^{+}).\tilde{f}_{\widetilde{\mathcal{S}}^{+}_{\infty}}(\tilde{y}^{+}_{1}))+
\mathrm{d}_{\mathbb{H}^3}
(\rho^{\mathcal{S}}_{n}(\gamma^{+}).\tilde{f}_{\widetilde{\mathcal{S}}^{+}_{\infty}}(\tilde{y}^{+}_{1}),
\rho^{\mathcal{S}}_{\infty}(\gamma^{+}).\tilde{f}_{\widetilde{\mathcal{S}}^{+}_{\infty}}(\tilde{y}^{+}_{1})).
\end{equation}
The fact that $\rho^{\mathcal{S}}_{n}(\gamma^{+})$ is an isometry of $\mathbb{H}^3$ implies the equality:
\begin{equation}\label{chaaa_frm_isometry property_for_dev_apps_periodicity}
\mathrm{d}_{\mathbb{H}^3}(\rho^{\mathcal{S}}_{n}(\gamma^{+}).\tilde{f}_{\widetilde{\mathcal{S}}^{+}_{n}}(\tilde{y}^{+}_{1}),
\rho^{\mathcal{S}}_{n}(\gamma^{+}).\tilde{f}_{\widetilde{\mathcal{S}}^{+}_{\infty}}(\tilde{y}^{+}_{1}))=
\mathrm{d}_{\mathbb{H}^3}(\tilde{f}_{\widetilde{\mathcal{S}}^{+}_{n}}(\tilde{y}^{+}_{1}),
\tilde{f}_{\widetilde{\mathcal{S}}^{+}_{\infty}}(\tilde{y}^{+}_{1})).
\end{equation}
Therefore, substituting~(\ref{chaaa_frm_isometry property_for_dev_apps_periodicity}) in~(\ref{chaaa_frm_triangle_inequality_for_dev_apps_periodicity_2}), we obtain:
\begin{equation*}\label{chaaa_frm_triangle+isometry_for_dev_apps_periodicity_1}
\mathrm{d}_{\mathbb{H}^3}
(\rho^{\mathcal{S}}_{n}(\gamma^{+}).\tilde{f}_{\widetilde{\mathcal{S}}^{+}_{n}}(\tilde{y}^{+}_{1}),
\rho^{\mathcal{S}}_{\infty}(\gamma^{+}).\tilde{f}_{\widetilde{\mathcal{S}}^{+}_{\infty}}(\tilde{y}^{+}_{1}))\leq
\end{equation*}
\begin{equation}\label{chaaa_frm_triangle+isometry_for_dev_apps_periodicity_2}
\mathrm{d}_{\mathbb{H}^3}
(\tilde{f}_{\widetilde{\mathcal{S}}^{+}_{n}}(\tilde{y}^{+}_{1}),
\tilde{f}_{\widetilde{\mathcal{S}}^{+}_{\infty}}(\tilde{y}^{+}_{1}))+
\mathrm{d}_{\mathbb{H}^3}
(\rho^{\mathcal{S}}_{n}(\gamma^{+}).\tilde{f}_{\widetilde{\mathcal{S}}^{+}_{\infty}}(\tilde{y}^{+}_{1}),
\rho^{\mathcal{S}}_{\infty}(\gamma^{+}).\tilde{f}_{\widetilde{\mathcal{S}}^{+}_{\infty}}(\tilde{y}^{+}_{1})).
\end{equation}
Hence, by~(\ref{chaaa_frm_triangle+isometry_for_dev_apps_periodicity_2}), (\ref{chaaa_frm_eps-n_relation_for_y1+}), and~(\ref{chaaa_frm_eps-n_relation_for_isometry_convergence}), we conclude that it is sufficient to pose $n_{0}=\max(n_{1},n_{2})$ to satisfy the condition~(\ref{chaaa_frm_eps-n_relation_to_prove}).
\end{proof}

Now we are able to extend the functions $\tilde{f}_{\widetilde{\mathcal{S}}^{+}_{\infty}}:
\widehat{\Delta}^{+}\rightarrow\mathbb{H}^3$ and $\tilde{f}_{\widetilde{\mathcal{S}}^{-}_{\infty}}:
\widehat{\Delta}^{-}\rightarrow\mathbb{H}^3$ to the whole domains $\widetilde{\mathcal{S}}^{+}$ and $\widetilde{\mathcal{S}}^{-}$. Let us do it as follows: for arbitrary points $\tilde{y}^{+}\in\widetilde{\mathcal{S}}^{+}$ and $\tilde{y}^{-}\in\widetilde{\mathcal{S}}^{-}$ we find such points $\tilde{y}^{+}_{\Delta}$ and $\tilde{y}^{-}_{\Delta}$ in the fundamental domains $\Delta^{+}\subset\widehat{\Delta}^{+}\subset\widetilde{\mathcal{S}}^{+}$ and $\Delta^{-}\subset\widehat{\Delta}^{-}\subset\widetilde{\mathcal{S}}^{-}$ of the surface $\mathcal{S}$ and such elements $\gamma^{+},\gamma^{-}\in{\pi}_{1}(\mathcal{S})$ that $\tilde{y}^{+}=\gamma^{+}.\tilde{y}^{+}_{\Delta}$ and $\tilde{y}^{-}=\gamma^{-}.\tilde{y}^{-}_{\Delta}$, then we define $\tilde{f}_{\widetilde{\mathcal{S}}^{+}_{\infty}}(\tilde{y}^{+})\stackrel{\mathrm{def}}{=}
\rho^{\mathcal{S}}_{\infty}(\gamma^{+}).\tilde{f}_{\widetilde{\mathcal{S}}^{+}_{\infty}}(\tilde{y}^{+}_{\Delta})$
and $\tilde{f}_{\widetilde{\mathcal{S}}^{-}_{\infty}}(\tilde{y}^{-})\stackrel{\mathrm{def}}{=}
\rho^{\mathcal{S}}_{\infty}(\gamma^{-}).\tilde{f}_{\widetilde{\mathcal{S}}^{-}_{\infty}}(\tilde{y}^{-}_{\Delta})$. By construction, the surfaces
$\widetilde{\mathcal{S}}^{+}_{\infty}\stackrel{\mathrm{def}}{=}
\tilde{f}_{\widetilde{\mathcal{S}}^{+}_{\infty}}(\widetilde{\mathcal{S}}^{+})$ and
$\widetilde{\mathcal{S}}^{-}_{\infty}\stackrel{\mathrm{def}}{=}
\tilde{f}_{\widetilde{\mathcal{S}}^{-}_{\infty}}(\widetilde{\mathcal{S}}^{-})$ are
invariant under the actions of the group $\rho^{\mathcal{S}}_{\infty}(\pi_{1}(\mathcal{S}))$ of isometries of $\mathbb{H}^3$.

Repeating almost literally the demonstration of Remark~\ref{chaaa_remark_periodicity_of_limit_dev_applications}, we can prove

\begin{lemma}\label{chaaa_lemma_convergence_of_infinite_surfaces}
The sequences of developing maps
$\{\tilde{f}_{\widetilde{\mathcal{S}}^{+}_{n}}:\widetilde{\mathcal{S}}^{+}\rightarrow\mathbb{H}^3\}_{n\in\mathbb{N}}$
and $\{\tilde{f}_{\widetilde{\mathcal{S}}^{-}_{n}}:\widetilde{\mathcal{S}}^{-}\rightarrow\mathbb{H}^3\}_{n\in\mathbb{N}}$ converge to continuous functions
$\tilde{f}_{\widetilde{\mathcal{S}}^{+}_{\infty}}:\widetilde{\mathcal{S}}^{+}\rightarrow\mathbb{H}^3$ and $\tilde{f}_{\widetilde{\mathcal{S}}^{-}_{\infty}}:\widetilde{\mathcal{S}}^{-}\rightarrow\mathbb{H}^3$.
\end{lemma}

Finally, we show

\begin{remark}\label{chaaa_remark_group_rho_infty_is_quasi-fuchsian}
The boundaries at infinity $\partial_{\infty}\widetilde{\mathcal{S}}^{+}_{\infty}\subset\partial_{\infty}\mathbb{H}^3$ and $\partial_{\infty}\widetilde{\mathcal{S}}^{-}_{\infty}\subset\partial_{\infty}\mathbb{H}^3$ of the surfaces $\widetilde{\mathcal{S}}^{+}_{\infty}$ and $\widetilde{\mathcal{S}}^{-}_{\infty}$ coincide with the limit set $\Lambda_{\rho^{\mathcal{S}}_{\infty}}$ of the group $\rho^{\mathcal{S}}_{\infty}({\pi}_{1}(\mathcal{S}))$. Moreover, the group $\rho^{\mathcal{S}}_{\infty}({\pi}_{1}(\mathcal{S}))$ of isometries of $\mathbb{H}^3$ from Lemma~\ref{chaaa_lemma_converging_subsequence_of_isometries} is quasi-Fuchsian.
\end{remark}
\begin{proof}
By Lemma~\ref{chaaa_lemma_convergence_of_infinite_surfaces}, the sequences of surfaces $\{\widetilde{\mathcal{S}}^{+}_{n}\}_{n\in\mathbb{N}}$ and $\{\widetilde{\mathcal{S}}^{-}_{n}\}_{n\in\mathbb{N}}$ bounding the convex connected hyperbolic domains $\{\widetilde{\mathcal{M}}_{n}\}_{n\in\mathbb{N}}$ converge to the surfaces $\widetilde{\mathcal{S}}^{+}_{\infty}$ and $\widetilde{\mathcal{S}}^{-}_{\infty}$ in $\mathbb{H}^3$. Hence, the sets $\{\widetilde{\mathcal{M}}_{n}\}_{n\in\mathbb{N}}$ converge to a convex connected hyperbolic domain $\widetilde{\mathcal{M}}_{\infty}$.
Moreover, the boundaries at infinity $\{\partial_{\infty}\widetilde{\mathcal{S}}^{+}_{n}\}_{n\in\mathbb{N}}$ and $\{\partial_{\infty}\widetilde{\mathcal{S}}^{-}_{n}\}_{n\in\mathbb{N}}$ converge to the curves $\partial_{\infty}\widetilde{\mathcal{S}}^{+}_{\infty}\subset\partial_{\infty}\mathbb{H}^3$ and $\partial_{\infty}\widetilde{\mathcal{S}}^{-}_{\infty}\subset\partial_{\infty}\mathbb{H}^3$. Indeed, our surfaces in the Poincar\'e disc model of $\mathbb{H}^3$ considered as Euclidean surfaces inside a unitary ball converge together with their boundaries.

Recall that, by the Labourie-Schlenker Theorem \ref{chaaa_thm_labourie-schlenker}, for each $n\in\mathbb{N}$ the curves $\partial_{\infty}\widetilde{\mathcal{S}}^{+}_{n}$ and $\partial_{\infty}\widetilde{\mathcal{S}}^{-}_{n}$ coincide with the limit set $\Lambda_{\rho^{\mathcal{S}}_{n}}$ of the quasi-Fuchsian holonomy representations
$\rho^{\mathcal{S}}_{n}({\pi}_{1}(\mathcal{S}))$ which is homotopic to a circle in $\partial_{\infty}\mathbb{H}^3$. On the other hand, by Assumption~\ref{chaaa_aspt_convergence_of_isometry_groups}, $\rho^{\mathcal{S}}_{n}({\pi}_{1}(\mathcal{S}))\Rightarrow\rho^{\mathcal{S}}_{\infty}({\pi}_{1}(\mathcal{S}))$ as $n\rightarrow\infty$, which implies that the sequence of the limit sets $\{\Lambda_{\rho^{\mathcal{S}}_{n}}\}_{n\in\mathbb{N}}$ converges to the limit set $\Lambda_{\rho^{\mathcal{S}}_{\infty}}$ (see, for instance, \cite[p.~323]{chaaa_Mat2004}).

Thus, the boundaries at infinity $\partial_{\infty}\widetilde{\mathcal{S}}^{+}_{\infty}$ and $\partial_{\infty}\widetilde{\mathcal{S}}^{-}_{\infty}$ of the surfaces $\widetilde{\mathcal{S}}^{+}_{\infty}$ and $\widetilde{\mathcal{S}}^{-}_{\infty}$ coincide with the limit set $\Lambda_{\rho^{\mathcal{S}}_{\infty}}$ of the group $\rho^{\mathcal{S}}_{\infty}({\pi}_{1}(\mathcal{S}))$. Furthermore, we conclude that the boundary $\partial\widetilde{\mathcal{M}}_{\infty}$ of the domain $\widetilde{\mathcal{M}}_{\infty}$ consists of the surfaces $\widetilde{\mathcal{S}}^{+}_{\infty}$ and $\widetilde{\mathcal{S}}^{-}_{\infty}$, and the boundary at infinity $\partial_{\infty}\widetilde{\mathcal{M}}_{\infty}$ of $\widetilde{\mathcal{M}}_{\infty}$ also coincides with $\Lambda_{\rho^{\mathcal{S}}_{\infty}}$.

Since the surfaces $\widetilde{\mathcal{S}}^{+}_{\infty}$ and $\widetilde{\mathcal{S}}^{-}_{\infty}$ are topological discs embedded in $\mathbb{H}^3$, their common boundary at infinity is homotopic to a circle. Therefore, by definition, the group $\rho^{\mathcal{S}}_{\infty}({\pi}_{1}(\mathcal{S}))$ is quasi-Fuchsian.
\end{proof}

Note that the domain $\widetilde{\mathcal{M}}_{\infty}$ which appeared during the demonstration of Remark~\ref{chaaa_remark_group_rho_infty_is_quasi-fuchsian}, is invariant under the actions of the quasi-Fuchsian group $\rho^{\mathcal{S}}_{\infty}({\pi}_{1}(\mathcal{S}))$ of isometries of $\mathbb{H}^3$.

\subsubsection{Adaptation of a classical theorem of A.~D.~Alexandrov to the hyperbolic case} \label{chaaa_sec_adaptation_of_Alexandrov_theorem}

Recall a classical result due to A.~D.~Alexandrov:
\begin{theorem}[Theorem~1 in Sec.~1 of Chapter III \cite{chaaa_ADA2006}, p.~91]\label{chaaa_theorem_alexandrov_convergence_of_convex_surfaces}
If a sequence of closed convex surfaces $\mathcal{F}_{n}$ converges to a closed convex surface $\mathcal{F}$ and if two sequences of points $X_n$ and $Y_n$ on $\mathcal{F}_{n}$ converge to two points $X$ and $Y$ of $\mathcal{F}$, respectively, then the distances between the points $X_n$ and $Y_n$ measured on the surfaces $\mathcal{F}_{n}$ converge to the distance between the points $X$ and $Y$ measured on $\mathcal{F}$, i.e., $\mathrm{d}_{\mathcal{F}}(X,Y)={\lim}_{n\rightarrow\infty}\mathrm{d}_{\mathcal{F}_{n}}(X_{n},Y_{n})$.
\end{theorem}

A.~D.~Alexandrov demonstrated this theorem in Euclidean $3$-space. Slightly modifying his proof, here we show the validity of Theorem~\ref{chaaa_theorem_alexandrov_convergence_of_convex_surfaces} in hyperbolic space $\mathbb{H}^3$. We will largely use this result in Section~\ref{chaaa_sec_induced_metrics_of_limit_surfaces}.

First we remark that the proof of Theorem~\ref{chaaa_theorem_alexandrov_convergence_of_convex_surfaces} in the Euclidean case is based on the two following lemmas which hold true in all Hadamard spaces (i.e. in the hyperbolic space as well), and it uses the mentioned below properties of the arc length in any complete metric space:

\begin{lemma}[Lemma~2 in Sec.~1 of Chapter III \cite{chaaa_ADA2006}, p.~93]\label{chaaa_lemma_no-1_used_in_alexandrov_theorem} If a curve $L$ lies outside a closed convex surface $\mathcal{F}$, then the length of this curve is not less than the distance on $\mathcal{F}$ between the projections of its endpoints to the surface $\mathcal{F}$. In particular, if the ends $A$ and $B$ of the curve $L$ lie on $\mathcal{F}$, then the length of the curve $L$ is not less than the length of the shortest arc $AB$ on the surface $\mathcal{F}$.
\end{lemma}

\begin{lemma}[Lemma~3 in Sec.~1 of Chapter III \cite{chaaa_ADA2006}, p.~93]\label{chaaa_lemma_no-2_used_in_alexandrov_theorem} If a sequence of closed convex surfaces $\mathcal{F}_{n}$ converges to a nondegenerate surface $\mathcal{F}$ and if points $X_n$ and $Y_n$ converge to the same point $X$ on $\mathcal{F}$, then the distance between $X_n$ and $Y_n$ on $\mathcal{F}_{n}$ converges to zero: ${\lim}_{n\rightarrow\infty}\mathrm{d}_{\mathcal{F}_{n}}(X_{n},Y_{n})=0$.
\end{lemma}

\begin{property}[Theorem~3 in Sec.~2 of Chapter II \cite{chaaa_ADA2006}, p.~66]\label{chaaa_property_no-1_used_in_alexandrov_theorem} There is a shortest arc of every two points on a manifold with complete intrinsic metric.
\end{property}

\begin{property}[Theorem~4 in Sec.~1 of Chapter II \cite{chaaa_ADA2006}, p.~59]\label{chaaa_property_no-2_used_in_alexandrov_theorem} We can choose a convergent subsequence from each infinite set of curves in a compact domain of length not exceeding a given one.
\end{property}

\begin{property}[Theorem~5 in Sec.~1 of Chapter II \cite{chaaa_ADA2006}, p.~59]\label{chaaa_property_no-2_used_in_alexandrov_theorem} If curves $L_n$ converge to a curve $L$, then the length of $L$ is not greater than the lower limit of the lengths of $L_n$.
\end{property}

However, there is a place in the proof of Theorem~\ref{chaaa_theorem_alexandrov_convergence_of_convex_surfaces} which uses some particular properties of Euclidean space, specifically, of the Euclidean homothety. In the following statement we formulate what is shown there:

\begin{lemma}\label{chaaa_lemma_inequality_in_alexandrov_theorem}
If a sequence of closed convex surfaces $\mathcal{F}_{n}$ converges to a nondegenerate closed convex surface $\mathcal{F}$ and if two sequences of points $X_n$ and $Y_n$ on $\mathcal{F}_{n}$ converge to two points $X$ and $Y$ of $\mathcal{F}$, respectively, then
\begin{equation}\label{chaaa_frm_inequality_in_alexandrov_theorem}
{\limsup}_{n\rightarrow\infty}\mathrm{d}_{\mathcal{F}_{n}}(X_{n},Y_{n})\leq\mathrm{d}_{\mathcal{F}}(X,Y).
\end{equation}
\end{lemma}

\emph{Proof of Lemma~\ref{chaaa_lemma_inequality_in_alexandrov_theorem} in the Euclidean case} \cite[pp.~95--96]{chaaa_ADA2006}. Take a point $O$ inside the surface $\mathcal{F}$ and perform the homothety transform with the center at $O$ of the surfaces $\mathcal{F}_{n}$ so that all these surfaces turn out to be inside $\mathcal{F}$. Note that if the initial surface $\mathcal{F}_{n}$ lies inside $\mathcal{F}$ then we do not need to apply the homothety, so we pose the coefficient of homothety $\lambda_{n}=1$; otherwise we perform the scaling back homothety transform with $\lambda_{n}<1$. Since the surfaces $\mathcal{F}_{n}$ converge to $\mathcal{F}$, the coefficients $\lambda_{n}$ can be taken closer and closer to $1$ as $n$ increases and $\lambda_{n}\rightarrow 1$ as $n\rightarrow\infty$. The surfaces and points, which are obtained from the surfaces $\mathcal{F}_{n}$ and the points $X_n$ and $Y_n$ as a result of this transformation, will be denoted by $\lambda_{n}\mathcal{F}_{n}$, $\lambda_{n}{X}_{n}$, and $\lambda_{n}{Y}_{n}$. Since $\lambda_{n}\rightarrow 1$ and the points $X_n$ and $Y_n$ tend to $X$ and $Y$, the points $\lambda_{n}{X}_{n}$ and $\lambda_{n}{Y}_{n}$ also converge to $X$ and $Y$, respectively.

Let $X_{n}'$ and $Y_{n}'$ be the projections of the points $X$ and $Y$ to the surfaces $\lambda_{n}\mathcal{F}_{n}$. By Lemma~\ref{chaaa_lemma_no-1_used_in_alexandrov_theorem},
\begin{equation}\label{chaaa_frm_dist_inequality_for_projections}
\mathrm{d}_{\lambda_{n}\mathcal{F}_{n}}(X_{n}',Y_{n}')\leq\mathrm{d}_{\mathcal{F}}(X,Y).
\end{equation}

Obviously, the points $X_{n}'$ converge to $X$ as $n\rightarrow\infty$, and at the same time, the points $\lambda_{n}{X}_{n}$ also converge to $X$. Therefore, by Lemma~\ref{chaaa_lemma_no-2_used_in_alexandrov_theorem},
\begin{equation}\label{chaaa_frm_dist_between_homothety_and_projection_of_X_on_lambdaF}
\mathrm{d}_{\lambda_{n}\mathcal{F}_{n}}(\lambda_{n}{X}_{n},X_{n}')\rightarrow 0,
\end{equation}
and, by the same arguments,
\begin{equation}\label{chaaa_frm_dist_between_homothety_and_projection_of_Y_on_lambdaF}
\mathrm{d}_{\lambda_{n}\mathcal{F}_{n}}(Y_{n}',\lambda_{n}{Y}_{n})\rightarrow 0.
\end{equation}
By the "triangle inequality",
\begin{equation}\label{chaaa_frm_triangle_inequality_on_lambdaF}
\mathrm{d}_{\lambda_{n}\mathcal{F}_{n}}(\lambda_{n}{X}_{n},\lambda_{n}{Y}_{n})\leq
\mathrm{d}_{\lambda_{n}\mathcal{F}_{n}}(\lambda_{n}{X}_{n},X_{n}')+
\mathrm{d}_{\lambda_{n}\mathcal{F}_{n}}(X_{n}',Y_{n}')+
\mathrm{d}_{\lambda_{n}\mathcal{F}_{n}}(Y_{n}',\lambda_{n}{Y}_{n}).
\end{equation}

Using the inequality~(\ref{chaaa_frm_dist_inequality_for_projections}) and the relations~(\ref{chaaa_frm_dist_between_homothety_and_projection_of_X_on_lambdaF}) and~(\ref{chaaa_frm_dist_between_homothety_and_projection_of_Y_on_lambdaF}) and passing to the limit in~(\ref{chaaa_frm_triangle_inequality_on_lambdaF}) as $n\rightarrow\infty$, we obtain
\begin{equation}\label{chaaa_frm_inequality_for_homotheties_in_alexandrov_theorem}
{\limsup}_{n\rightarrow\infty}\mathrm{d}_{\lambda_{n}\mathcal{F}_{n}}(\lambda_{n}{X}_{n},\lambda_{n}{Y}_{n})
\leq\mathrm{d}_{\mathcal{F}}(X,Y).
\end{equation}
But under the homothety with coefficient $\lambda_{n}$, all distances change by $\lambda_{n}$ times, and, therefore,
\begin{equation}\label{chaaa_frm_euclidean_homothety_property}
\mathrm{d}_{\lambda_{n}\mathcal{F}_{n}}(\lambda_{n}{X}_{n},\lambda_{n}{Y}_{n})=
\lambda_{n}\mathrm{d}_{\mathcal{F}_{n}}({X}_{n},{Y}_{n});
\end{equation}
since $\lambda_{n}\rightarrow 1$, the formula~(\ref{chaaa_frm_inequality_for_homotheties_in_alexandrov_theorem}) implies~(\ref{chaaa_frm_inequality_in_alexandrov_theorem}).
$\square$

Let us adapt the proof of Lemma~\ref{chaaa_lemma_inequality_in_alexandrov_theorem} for hyperbolic $3$-space.

\emph{Modification of the proof of Lemma~\ref{chaaa_lemma_inequality_in_alexandrov_theorem} for the hyperbolic case}.
Further we will use the notation developed in the proof of the Euclidean version of Lemma~\ref{chaaa_lemma_inequality_in_alexandrov_theorem}. Considering the surfaces $\mathcal{F}\subset\mathbb{H}^3$ and $\mathcal{F}_{n}\subset\mathbb{H}^3$ ($n\in\mathbb{N}$) in the projective model $\mathbb{K}^3$ of hyperbolic space $\mathbb{H}^3$ as surfaces of Euclidean space $\mathbb{R}^3$ and supposing in addition that the center $O_{\mathbb{K}}$ of the Kleinian model $\mathbb{K}^3$ lies inside the surface $\mathcal{F}$, as previously, let us perform the Euclidean homothety transforms with the center at $O_{\mathbb{K}}$ of the surfaces $\mathcal{F}_{n}$ so that all resulting surfaces $\lambda_{n}\mathcal{F}_{n}$ turn out to be inside $\mathcal{F}$ (here $\lambda_{n}$ are the Euclidean homothety coefficients, $n\in\mathbb{N}$).
Below we will call \emph{Euclidean homothety transform} any transformation of hyperbolic space $\mathbb{H}^3$ which corresponds to a homothety transformation of Euclidean space $\mathbb{R}^3$ when we identify $\mathbb{R}^3$ with the projective model $\mathbb{K}^3$ of $\mathbb{H}^3$.
We already know that in the Euclidean case the distances between corresponding pairs of points ${X}_{n},{Y}_{n}\in\mathcal{F}_{n}$ and $\lambda_{n}{X}_{n},\lambda_{n}{Y}_{n}\in\lambda_{n}\mathcal{F}_{n}$ in the induced metrics of the surfaces $\mathcal{F}_{n}$ and $\lambda_{n}\mathcal{F}_{n}$ satisfy the relation~(\ref{chaaa_frm_euclidean_homothety_property}). Let us now find a similar condition in the case when $\mathcal{F}_{n}$ and $\lambda_{n}\mathcal{F}_{n}$ are regarded as surfaces of hyperbolic space $\mathbb{H}^3$.

All closed convex surfaces $\mathcal{F}_{n}$ together with their limit surface $\mathcal{F}$ can be included into a sufficiently large ball $\mathcal{B}\subset\mathbb{H}^3$ centered at $O_{\mathbb{K}}$. Let us put $\mathcal{B}$ into the Kleinian model $\mathbb{K}^3$ of $\mathbb{H}^3$ and let $\rho_{\mathcal{B}}<1$ stands for the Euclidean radius of $\mathcal{B}$ in $\mathbb{K}^3$.

An Euclidean homothety transform $\tau$ centered at $O_{\mathbb{K}}\in\mathbb{K}^3$ with a coefficient $\lambda\leq1$ sends any point $Z$ inside $\mathcal{B}$ to the point $\lambda Z$. Denote by $\rho(<\rho_{\mathcal{B}})$ the length of the Euclidean radius-vector connecting the points $O_{\mathbb{K}}$ and $Z$ in the projective model $\mathbb{K}^3$ of $\mathbb{H}^3$. The differential $d\tau$ of the hyperbolic transformation $\tau$ sends any vector $v_{Z}\in T_{Z}\mathbb{H}^3$ codirectional with the geodesic $L_{Z}$ which contains the points $O_{\mathbb{K}}$, $Z$, and $\lambda Z$, to the vector $v_{\lambda Z}\in T_{\lambda Z}\mathbb{H}^3$ also codirectional with $L_{Z}$. A direct calculation shows that the norms of the vectors $v_{Z}$ and $v_{\lambda Z}$ are related as follows:
\begin{equation}\label{chaaa_frm_radial_vectors_norms_after_homothety}
\|v_{\lambda Z}\|=\frac{\lambda(1-{\rho}^2)}{1-{\lambda}^2{\rho}^2}\|v_{Z}\|.
\end{equation}
It is easy to verify that for $\lambda\leq 1$ the function $f_{\lambda}(\rho)\stackrel{\mathrm{def}}{=}\frac{\lambda(1-{\rho}^2)}{1-{\lambda}^2{\rho}^2}$ in $\rho$ is monotonically decreasing in the segment $[0,\rho_{\mathcal{B}}]$. Together with~(\ref{chaaa_frm_radial_vectors_norms_after_homothety}), this fact implies:
\begin{equation}\label{chaaa_frm_ineq_radial_vectors_norms_after_homothety}
\|v_{\lambda Z}\|\geq\frac{\lambda(1-{\rho_{\mathcal{B}}}^2)}{1-{\lambda}^2{\rho_{\mathcal{B}}}^2}\|v_{Z}\|.
\end{equation}
Similarly, the differential $d\tau$ sends any vector $v^{\perp}_{Z}\in T_{Z}\mathbb{H}^3$ perpendicular to the geodesic $L_{Z}$, to the vector $v^{\perp}_{\lambda Z}\in T_{\lambda Z}\mathbb{H}^3$ also perpendicular to $L_{Z}$. A direct calculation shows that the norms of the vectors $v^{\perp}_{Z}$ and $v^{\perp}_{\lambda Z}$ are related as follows:
\begin{equation}\label{chaaa_frm_perpend_vectors_norms_after_homothety}
\|v^{\perp}_{\lambda Z}\|=\frac{\lambda\sqrt{1-{\rho}^2}}{\sqrt{1-{\lambda}^2{\rho}^2}}\|v^{\perp}_{Z}\|.
\end{equation}
It is easy to verify that for $\lambda\leq 1$ the function $g_{\lambda}(\rho)\stackrel{\mathrm{def}}{=}\frac{\lambda\sqrt{1-{\rho}^2}}{\sqrt{1-{\lambda}^2{\rho}^2}}$ in $\rho$ is monotonically decreasing in the segment $[0,\rho_{\mathcal{B}}]$. Together with~(\ref{chaaa_frm_perpend_vectors_norms_after_homothety}), it implies:
\begin{equation}\label{chaaa_frm_ineq_perpend_vectors_norms_after_homothety}
\|v^{\perp}_{\lambda Z}\|\geq
\frac{\lambda\sqrt{1-{\rho_{\mathcal{B}}}^2}}{\sqrt{1-{\lambda}^2{\rho_{\mathcal{B}}}^2}}\|v^{\perp}_{Z}\|.
\end{equation}

Any vector $u\in T_{Z}\mathbb{H}^3$ can be decomposed as the sum of two vectors $u=v+v^{\perp}$, $v,v^{\perp}\in T_{Z}\mathbb{H}^3$, such that the vector $v$ is codirectional with the geodesic $L_{Z}$, and the vector $v^{\perp}$ is perpendicular to $L_{Z}$. Hence, (\ref{chaaa_chaaa_frm_ineq_radial_vectors_norms_after_homothety}) and (\ref{chaaa_frm_ineq_perpend_vectors_norms_after_homothety}) imply that the norms of the vectors $u\in T_{Z}\mathbb{H}^3$ and $u_{\lambda}\stackrel{\mathrm{def}}{=}d\tau(Z).u\in T_{\lambda Z}\mathbb{H}^3$ satisfy the following inequality:
\begin{equation}\label{chaaa_frm_ineq_arbirtary_vectors_norms_after_homothety}
\|u_{\lambda}\|\geq\min\bigg{\{}\frac{\lambda(1-{\rho_{\mathcal{B}}}^2)}{1-{\lambda}^2{\rho_{\mathcal{B}}}^2},
\frac{\lambda\sqrt{1-{\rho_{\mathcal{B}}}^2}}{\sqrt{1-{\lambda}^2{\rho_{\mathcal{B}}}^2}}\bigg{\}}\|u\|
=\frac{\lambda(1-{\rho_{\mathcal{B}}}^2)}{1-{\lambda}^2{\rho_{\mathcal{B}}}^2}\|u\|
\end{equation}
as $0<\lambda\leq1$.

Recall that the length of a curve $c:[0,1]\rightarrow\mathbb{H}^3$ which is $C^{1}$-smooth almost everywhere is given by the formula $l(c)\stackrel{\mathrm{def}}{=}\int^{1}_{0}\|c'(t)\|dt$ where $c'(t)\in T_{c(t)}\mathbb{H}^3$ for almost all $t\in[0,1]$. Suppose in addition that the curve $c$ lies in the interior of the ball $\mathcal{B}$, apply the Euclidean homothety transform $\tau$ to $c$, and denote the resulting curve by $c_{\lambda}$. Hence, taking into account the inequality~(\ref{chaaa_frm_ineq_arbirtary_vectors_norms_after_homothety}), we see that the lengths of the curves $c$ and $c_{\lambda}$ are related as follows:
\begin{equation*}\label{chaaa_frm_ineq_arbirtary_vectors_norms_after_homothety}
l(c_{\lambda})\geq\frac{\lambda(1-{\rho_{\mathcal{B}}}^2)}{1-{\lambda}^2{\rho_{\mathcal{B}}}^2}l(c).
\end{equation*}

Thus, returning to the consideration of the distances between the pairs of points ${X}_{n},{Y}_{n}\in\mathcal{F}_{n}$ and $\lambda_{n}{X}_{n},\lambda_{n}{Y}_{n}\in\lambda_{n}\mathcal{F}_{n}$ in the induced metrics of the surfaces $\mathcal{F}_{n}$ and $\lambda_{n}\mathcal{F}_{n}$, we conclude that in the hyperbolic case the inequality
\begin{equation}\label{chaaa_frm_kleinian_homothety_property}
\mathrm{d}_{\lambda_{n}\mathcal{F}_{n}}(\lambda_{n}{X}_{n},\lambda_{n}{Y}_{n})\geq
\frac{\lambda_{n}(1-{\rho_{\mathcal{B}}}^2)}{1-\lambda_{n}^2{\rho_{\mathcal{B}}}^2}
\mathrm{d}_{\mathcal{F}_{n}}({X}_{n},{Y}_{n})
\end{equation}
holds.
Substituting~(\ref{chaaa_frm_kleinian_homothety_property}) in the formula~(\ref{chaaa_frm_inequality_for_homotheties_in_alexandrov_theorem}) which is valid in both Euclidean and hyperbolic situations, we get:
\begin{equation}\label{chaaa_frm_last_hyperbolic_step_in_alexandrov_theorem}
{\limsup}_{n\rightarrow\infty}\frac{\lambda_{n}(1-{\rho_{\mathcal{B}}}^2)}{1-\lambda_{n}^2{\rho_{\mathcal{B}}}^2}
\mathrm{d}_{\mathcal{F}_{n}}({X}_{n},{Y}_{n})\leq\mathrm{d}_{\mathcal{F}}(X,Y).
\end{equation}
Since the expression $\frac{\lambda_{n}(1-{\rho_{\mathcal{B}}}^2)}{1-\lambda_{n}^2{\rho_{\mathcal{B}}}^2}$ tends to $1$ as the numbers $\lambda_{n}$ approach to $1$, the formula~(\ref{chaaa_frm_last_hyperbolic_step_in_alexandrov_theorem}) implies~(\ref{chaaa_frm_inequality_in_alexandrov_theorem}).
$\square$

We have just adapted to the hyperbolic situation the only place in the proof of Theorem~\ref{chaaa_theorem_alexandrov_convergence_of_convex_surfaces} largely depending on properties of Euclidean space. Therefore, Theorem~\ref{chaaa_theorem_alexandrov_convergence_of_convex_surfaces} remains valid in hyperbolic $3$-space.

When the present work was already written, the author found that A.~D.~Alexandrov proved the hyperbolic version of Theorem~\ref{chaaa_theorem_alexandrov_convergence_of_convex_surfaces} using different methods long ago in 1945 (see his paper~\cite[Theorem 3]{chaaa_ADA1945} in Russian).

\subsubsection{Induced metrics of the surfaces $\widetilde{\mathcal{S}}^{+}_{\infty}$ and $\widetilde{\mathcal{S}}^{-}_{\infty}$} \label{chaaa_sec_induced_metrics_of_limit_surfaces}

Return to consideration of the family of convex domains $\{\widetilde{\mathcal{M}}_{n}\}_{n=1}^{\infty}$ with the boundaries $\partial\widetilde{\mathcal{M}}_{n}=\widetilde{\mathcal{S}}^{+}_{n}\cup\widetilde{\mathcal{S}}^{-}_{n}$ (see Sections~\ref{chaaa_sec_Arzela-Ascoli_theorem_application} and~\ref{chaaa_sec_convergence_of_isometries_of_H3}) in
hyperbolic space $\mathbb{H}^3$. Assume in addition that the marked points $\tilde{x}^{+}_{n}\in\widetilde{\mathcal{S}}^{+}_{n}$, $n=1,...,\infty$, are all identified with an arbitrary point $O_{\mathbb{H}}\in\mathbb{K}^3$.

Consider a ball $\mathcal{\hat{B}}\subset\mathbb{H}^3$ centered at $O_{\mathbb{H}}$ of a sufficiently big hyperbolic radius $\hat{\rho}$ (it will be enough to put $\hat{\rho}=9\delta_{\mathcal{S}}+\delta_{\mathcal{M}}$, where the constants $\delta_{\mathcal{S}}$ and $\delta_{\mathcal{M}}$ are defined in Lemmas~\ref{chaaa_lemma_ubound_of_diameters_of_surfaces_Sn} and~\ref{chaaa_lemma_ubound_of_diameters_of_Mn}). Define the convex compact hyperbolic sets $\mathcal{M}^{\mathcal{B}}_{n}\stackrel{\mathrm{def}}{=}\widetilde{\mathcal{M}}_{n}\cap\mathcal{\hat{B}}$, and denote by $\mathcal{\hat{S}}^{+}_{n}\stackrel{\mathrm{def}}{=}\partial\mathcal{M}^{\mathcal{B}}_{n}
\cap\widetilde{\mathcal{S}}^{+}_{n}$ and $\mathcal{\hat{S}}^{-}_{n}\stackrel{\mathrm{def}}{=}\partial\mathcal{M}^{\mathcal{B}}_{n}
\cap\widetilde{\mathcal{S}}^{-}_{n}$ the intersections of the boundary $\partial\mathcal{M}^{\mathcal{B}}_{n}$ of the domain $\mathcal{M}^{\mathcal{B}}_{n}$ with the surfaces $\widetilde{\mathcal{S}}^{+}_{n}$ and $\widetilde{\mathcal{S}}^{-}_{n}$, $n=1,...,\infty$. By construction, the sets $\widehat{\Delta}^{+}_{n}$ and $\widehat{\Delta}^{-}_{n}$ defined in Lemma~\ref{chaaa_lemma_ubound_of_diameters_of_fundomain_nhoods} are subsets of $\mathcal{\hat{S}}^{+}_{n}$ and $\mathcal{\hat{S}}^{-}_{n}$ correspondingly, $n=1,...,\infty$.

\begin{remark}\label{chaaa_remark_radius_of_restricting_ball}
The ball $\mathcal{\hat{B}}$ is taken big enough in order to provide the following property: for an arbitrary pair of points $A^{+},B^{+}\in\widehat{\Delta}^{+}_{n}$ there exists a path $\zeta^{+}\subset\widehat{\Delta}^{+}_{n}$ connecting $A^{+}$ and $B^{+}$ which is shorter than any path $\xi^{+}\subset\partial\mathcal{M}^{\mathcal{B}}_{n}$ connecting $A^{+}$ and $B^{+}$ and such that $\xi^{+}\cap(\partial\mathcal{M}^{\mathcal{B}}_{n}\setminus\mathcal{\hat{S}}^{+}_{n})\neq\emptyset$. Similarly, for points $A^{-},B^{-}\in\widehat{\Delta}^{-}_{n}$ there exists a path $\zeta^{-}\subset\widehat{\Delta}^{-}_{n}$ connecting $A^{-}$ and $B^{-}$ which is shorter than any path $\xi^{-}\subset\partial\mathcal{M}^{\mathcal{B}}_{n}$ connecting $A^{-}$ and $B^{-}$ and such that $\xi^{-}\cap(\partial\mathcal{M}^{\mathcal{B}}_{n}\setminus\mathcal{\hat{S}}^{-}_{n})\neq\emptyset$. For this purpose, radius $\hat{\rho}=9\delta_{\mathcal{S}}+\delta_{\mathcal{M}}$ of the ball $\mathcal{\hat{B}}$ is sufficient although not optimal.
\end{remark}

Recall that, by Lemma~\ref{chaaa_lemma_convergence_of_infinite_surfaces},
the sequences of developing maps
$\{\tilde{f}_{\widetilde{\mathcal{S}}^{+}_{n}}:\widetilde{\mathcal{S}}^{+}\rightarrow\mathbb{H}^3\}_{n\in\mathbb{N}}$
and $\{\tilde{f}_{\widetilde{\mathcal{S}}^{-}_{n}}:\widetilde{\mathcal{S}}^{-}\rightarrow\mathbb{H}^3\}_{n\in\mathbb{N}}$ converge to continuous functions
$\tilde{f}_{\widetilde{\mathcal{S}}^{+}_{\infty}}:\widetilde{\mathcal{S}}^{+}\rightarrow\mathbb{H}^3$ and $\tilde{f}_{\widetilde{\mathcal{S}}^{-}_{\infty}}:\widetilde{\mathcal{S}}^{-}\rightarrow\mathbb{H}^3$, and the images of the maps $\tilde{f}_{\widetilde{\mathcal{S}}^{+}_{n}}$ and $\tilde{f}_{\widetilde{\mathcal{S}}^{-}_{n}}$ are convex surfaces $\widetilde{\mathcal{S}}^{+}_{n}$ and $\widetilde{\mathcal{S}}^{+}_{n}$ respectively, $n=1,...,\infty$.
Therefore, by construction, the surfaces $\{\widehat{\Delta}^{+}_{n}\}_{n\in\mathbb{N}}$ and $\{\widehat{\Delta}^{-}_{n}\}_{n\in\mathbb{N}}$ converge to $\widehat{\Delta}^{+}_{\infty}$ and $\widehat{\Delta}^{-}_{\infty}$, and moreover, the sequence of closed convex nondegenerate surfaces $\{\partial\mathcal{M}^{\mathcal{B}}_{n}\}_{n\in\mathbb{N}}$ converges to the closed convex nondegenerate surface $\partial\mathcal{M}^{\mathcal{B}}_{\infty}$ in $\mathbb{H}^3$. Applying the hyperbolic version of Theorem~\ref{chaaa_theorem_alexandrov_convergence_of_convex_surfaces} to the family of surfaces $\{\partial\mathcal{M}^{\mathcal{B}}_{n}\}_{n\in\mathbb{N}}$ which converges to $\partial\mathcal{M}^{\mathcal{B}}_{\infty}$ we conclude that the sequence of induced metrics on $\partial\mathcal{M}^{\mathcal{B}}_{n}$ tends to the induced metric on $\partial\mathcal{M}^{\mathcal{B}}_{\infty}$ as $n\rightarrow\infty$. In particular, given any two sequences of points $A^{+}_{n}$ and $B^{+}_{n}$ in $\widehat{\Delta}^{+}_{n}\subset\partial\mathcal{M}^{\mathcal{B}}_{n}$ converging to two points $A^{+}_{\infty}$ and $B^{+}_{\infty}$ in $\widehat{\Delta}^{+}_{\infty}\subset\partial\mathcal{M}^{\mathcal{B}}_{n}$, respectively, the distances between the points $A^{+}_{n}$ and $B^{+}_{n}$ measured on the surfaces $\partial\mathcal{M}^{\mathcal{B}}_{n}$ converge to the distance between the points $A^{+}_{\infty}$ and $B^{+}_{\infty}$ measured on $\partial\mathcal{M}^{\mathcal{B}}_{\infty}$, i.e.
\begin{equation}\label{chaaa_frm_convergence_of_boundary_metrics_MnB}
\mathrm{d}_{\partial\mathcal{M}^{\mathcal{B}}_{\infty}}(A^{+}_{\infty},B^{+}_{\infty})=
{\lim}_{n\rightarrow\infty}\mathrm{d}_{\partial\mathcal{M}^{\mathcal{B}}_{n}}(A^{+}_{n},B^{+}_{n}).
\end{equation}
By Remark~\ref{chaaa_remark_radius_of_restricting_ball}, the distance between the points $A^{+}_{n}$ and $B^{+}_{n}$ measured on $\partial\mathcal{M}^{\mathcal{B}}_{n}$ is equal to the distance between these points measured on $\mathcal{\hat{S}}^{+}_{n}$; also, by construction, $\mathcal{\hat{S}}^{+}_{n}$ is a convex subset of the surface $\widetilde{\mathcal{S}}^{+}_{n}$ with the induced metric $\tilde{h}^{+}_{n}$, therefore
\begin{equation}\label{chaaa_frm_coincidence_of_boundary_metrics_MnB_and_h+n}
\mathrm{d}_{\partial\mathcal{M}^{\mathcal{B}}_{n}}(A^{+}_{n},B^{+}_{n})=
\mathrm{d}_{\tilde{h}^{+}_{n}}(A^{+}_{n},B^{+}_{n}),
\end{equation}
$n=1,...,\infty$.
Substituting~(\ref{chaaa_frm_coincidence_of_boundary_metrics_MnB_and_h+n}) in~(\ref{chaaa_frm_convergence_of_boundary_metrics_MnB}), we get:
\begin{equation*}\label{chaaa_frm_convergence_of_restricted_h+n}
\mathrm{d}_{\tilde{h}^{+}_{\infty}}(A^{+}_{\infty},B^{+}_{\infty})=
{\lim}_{n\rightarrow\infty}\mathrm{d}_{\tilde{h}^{+}_{n}}(A^{+}_{n},B^{+}_{n}).
\end{equation*}
Hence, the sequence of the induced metrics $\tilde{h}^{+}_{n}$ of the surfaces $\widetilde{\mathcal{S}}^{+}_{n}$ restricted on the sets $\widehat{\Delta}^{+}_{n}$ converges to the induced metric $\tilde{h}^{+}_{\infty}$ of the surface $\widetilde{\mathcal{S}}^{+}_{\infty}$ restricted on $\widehat{\Delta}^{+}_{\infty}$ as $n\rightarrow\infty$. By analogy, the sequence of the induced metrics $\{\tilde{h}^{-}_{n}|_{\widehat{\Delta}^{-}_{n}}\}_{n\in\mathbb{N}}$  converges to the induced metric $\tilde{h}^{-}_{\infty}|_{\widehat{\Delta}^{-}_{\infty}}$.

In Sections~\ref{chaaa_sec_Arzela-Ascoli_theorem_application} and~\ref{chaaa_sec_convergence_of_isometries_of_H3} we constructed the surfaces $\widetilde{\mathcal{S}}^{+}_{n}$ and $\widetilde{\mathcal{S}}^{-}_{n}$ to be invariant under the actions of the discrete group $\rho^{\mathcal{S}}_{n}(\pi_{1}(\mathcal{S}))$ of isometries of $\mathbb{H}^3$ for each $n=1,...,\infty$.
Hence, the induced metrics $\tilde{h}^{+}_{n}$ and $\tilde{h}^{-}_{n}$ on the surfaces $\widetilde{\mathcal{S}}^{+}_{n}$ and $\widetilde{\mathcal{S}}^{-}_{n}$, respectively, are periodic with respect to the group $\rho^{\mathcal{S}}_{n}(\pi_{1}(\mathcal{S}))$, $n=1,...,\infty$. We have just proved that the metrics $\tilde{h}^{+}_{n}$ and $\tilde{h}^{-}_{n}$ converge to $\tilde{h}^{+}_{\infty}$ and $\tilde{h}^{-}_{\infty}$, correspondingly, in the neighborhoods $\widehat{\Delta}^{+}_{n}\subset\widetilde{\mathcal{S}}^{+}_{n}$ and $\widehat{\Delta}^{-}_{n}\subset\widetilde{\mathcal{S}}^{-}_{n}$ of the fundamental domains ${\Delta}^{+}_{n}\subset\widetilde{\mathcal{S}}^{+}_{n}$ and ${\Delta}^{-}_{n}\subset\widetilde{\mathcal{S}}^{-}_{n}$ of the surfaces ${\mathcal{S}}^{+}_{n}$ and ${\mathcal{S}}^{-}_{n}$. Since, by Assumption~\ref{chaaa_aspt_convergence_of_isometry_groups} and Remark~\ref{chaaa_remark_group_rho_infty_is_quasi-fuchsian},
the sequence of quasi-Fuchsian groups $\{\rho^{\mathcal{S}}_{n}({\pi}_{1}(\mathcal{S}))\}_{n\in\mathbb{N}}$ converges to a quasi-Fuchsian group $\rho^{\mathcal{S}}_{\infty}({\pi}_{1}(\mathcal{S}))$ of isometries of $\mathbb{H}^3$, we now conclude that the metrics $\tilde{h}^{+}_{n}$ and $\tilde{h}^{-}_{n}$ converge to $\tilde{h}^{+}_{\infty}$ and $\tilde{h}^{-}_{\infty}$ everywhere on $\widetilde{\mathcal{S}}^{+}_{n}$ and $\widetilde{\mathcal{S}}^{-}_{n}$ as $n\rightarrow\infty$.

To complete the proof of Theorem~\ref{chaaa_theorem_manifolds_with_convex_alexandrov_metric_on_boundary} let us consider the convex compact hyperbolic domain $\mathcal{M}_{\infty}\stackrel{\mathrm{def}}{=}
\widetilde{\mathcal{M}}_{\infty}/[\rho^{\mathcal{S}}_{\infty}({\pi}_{1}(\mathcal{S}))]$ with the boundary
\begin{equation*}\label{chaaa_frm_boundary_of_M_infty}
\partial\mathcal{M}_{\infty}\stackrel{\mathrm{def}}{=}
\mathcal{S}^{+}_{\infty}\cup\mathcal{S}^{-}_{\infty}\stackrel{\mathrm{def}}{=} \big{(}\widetilde{\mathcal{S}}^{+}_{\infty}/[\rho^{\mathcal{S}}_{\infty}({\pi}_{1}(\mathcal{S}))]\big{)}\bigcup
\big{(}\widetilde{\mathcal{S}}^{-}_{\infty}/[\rho^{\mathcal{S}}_{\infty}({\pi}_{1}(\mathcal{S}))]\big{)}
\end{equation*}
in the unbounded hyperbolic manifold $\mathcal{M}^{\circ}_{\infty}\stackrel{\mathrm{def}}{=}
\mathbb{H}^3/[\rho^{\mathcal{S}}_{\infty}({\pi}_{1}(\mathcal{S}))]$. The metric $\tilde{h}^{+}_{\infty}$ on the universal covering $\widetilde{\mathcal{S}}^{+}_{\infty}$ of the boundary component $\mathcal{S}^{+}_{\infty}$ of the domain $\mathcal{M}_{\infty}$ induces the metric $\breve{h}^{+}_{\infty}$ on the compact surface $\mathcal{S}^{+}_{\infty}$. We have recently showed that the pull-backs $\tilde{h}^{+}_{n}$ of the metrics ${h}^{+}_{n}$ (see Section~\ref{chaaa_sec_Arzela-Ascoli_theorem_application}) converge to the pull-back $\tilde{h}^{+}_{\infty}$ of the metric $\breve{h}^{+}_{\infty}$. Hence, the sequence of metrics $\{h^{+}_{n}\}_{n\in\mathbb{N}}$ tends to the metric $\breve{h}^{+}_{\infty}$ as $n\rightarrow\infty$. But in the very beginning of Section~\ref{chaaa_sec_Arzela-Ascoli_theorem_application} the $C^{\infty}$-smooth metrics $\{h^{+}_{n}\}_{n\in\mathbb{N}}$ were constructed in order to approximate the Alexandrov metric $h^{+}_{\infty}$. Therefore, the induced metric $\breve{h}^{+}_{\infty}$ on $\mathcal{S}^{+}_{\infty}$ coincides with the prescribed metric $h^{+}_{\infty}$. Similarly we obtain that the metric on the surface $\mathcal{S}^{-}_{\infty}$ is exactly $h^{-}_{\infty}$.

We sum up that the convex hyperbolic bounded domain $\mathcal{M}_{\infty}$ with the boundary $\partial\mathcal{M}_{\infty}=\mathcal{S}^{+}_{\infty}\cup\mathcal{S}^{-}_{\infty}$ in the quasi-Fuchsian manifold $\mathcal{M}^{\circ}_{\infty}$ was constructed in such a way that the induced metrics of the boundary components $\mathcal{S}^{+}_{\infty}$ and $\mathcal{S}^{-}_{\infty}$ coincide with the prescribed Alexandrov metrics $h^{+}_{\infty}$ and $h^{-}_{\infty}$. Theorem~\ref{chaaa_theorem_manifolds_with_convex_alexandrov_metric_on_boundary} is proved. $\square$

\section{Distance between boundary components of a convex compact domain in a quasi-Fuchsian manifold.}
\label{ch_dist_bb}

Consider a sequence of convex bounded domains $\mathcal{M}_n$ with the upper boundaries $\mathcal{S}^{+}_{n}$ and the lower boundaries $\mathcal{S}^{-}_{n}$ in quasi-Fuchsian manifolds $\mathcal{M}^{\circ}_n$, such that for all $n$ the convex regular metric surfaces $\mathcal{S}^{+}_{n}$ and $\mathcal{S}^{-}_{n}$ with the induced metrics $h^{+}_{n}$ and $h^{-}_{n}$, respectively, are topologically the same surface $\mathcal{S}$.

\textbf{Definition.} The \emph{distance} $d(\mathcal{K},\mathcal{L})$ \emph{between subsets} $\mathcal{K}$ \emph{and} $\mathcal{L}$ \emph{of a set} $\mathcal{N}$ is defined as follows: $d(\mathcal{K},\mathcal{L})\stackrel{\mathrm{def}}{=}\inf\{\mathrm{d}_{\mathcal{N}}(u,v)|u\in\mathcal{K},v\in\mathcal{L}\}$, where $\mathrm{d}_{\mathcal{N}}(u,v)$ stands for the distance between points $u$ and $v$ in $\mathcal{N}$.

In this section, we prove the following result which is essentially used in the demonstration of Theorem~\ref{chaaa_theorem_manifolds_with_polyhedral_metric_on_boundary} from the first part of this paper:
\begin{theorem}\label{chbb_theorem_distance_between_pairs_of_boundaries_is_unibounded}
Let the metrics $h^{+}_{n}$ tend to some metric $h^{+}_{\infty}$ (correspondingly, $h^{-}_{n}$ tend to $h^{-}_{\infty}$) as $n$ goes to $\infty$.
Then there is a common upper bound for the distances between $\mathcal{S}^{+}_{n}$ and $\mathcal{S}^{-}_{n}$ in $\mathcal{M}^{\circ}_n$ which does not depend on $n$.
\end{theorem}

The proof of Theorem~\ref{chbb_theorem_distance_between_pairs_of_boundaries_is_unibounded} is essentially based on
\begin{theorem}\label{chbb_theorem_distance_between_boundaries}
Given a convex bounded domain $\mathcal{M}$ with the upper boundary $\mathcal{S}^{+}$ and the lower boundary $\mathcal{S}^{-}$ in a quasi-Fuchsian manifold $\mathcal{M}^{\circ}$. If the metric surface $\mathcal{S}^{+}$ possesses two homotopically different nontrivial closed simple intersecting curves $c^{+}_{1}$ and $c^{+}_{2}$ of the lengths $l^{+}_{1}$ and $l^{+}_{2}$, and $\mathcal{S}^{-}$ possesses two homotopically different nontrivial closed simple intersecting curves $c^{-}_{1}$ and $c^{-}_{2}$ of the lengths $l^{-}_{1}$ and $l^{-}_{2}$ such that $c^{+}_{1}$ and $c^{-}_{1}$, as well as $c^{+}_{2}$ and $c^{-}_{2}$, are homotopically equivalent pairs of curves in $\mathcal{M}$, then the distance $d(\mathcal{S}^{+},\mathcal{S}^{-})$ between $\mathcal{S}^{+}$ and $\mathcal{S}^{-}$ is bounded from above by the constant
\begin{equation*}\label{chbb_frm1_thm_dist_bb_estimates}
d(\mathcal{S}^{+},\mathcal{S}^{-})<\max\bigg{\{}\bigg{(}l^{+}_{1}+l^{-}_{1}+\ln\frac{2l^{+}_{1}}{l^{-}_{1}}\bigg{)}, \bigg{(}l^{+}_{1}+l^{-}_{1}+\ln\frac{2l^{-}_{1}}{l^{+}_{1}}\bigg{)},\bigg{(}l^{+}_{2}+l^{-}_{2}+
\ln\frac{2l^{+}_{2}}{l^{-}_{2}}\bigg{)}, \bigg{(}l^{+}_{2}+l^{-}_{2}+\ln\frac{2l^{-}_{2}}{l^{+}_{2}}\bigg{)},
\end{equation*}
\begin{equation*}\label{chbb_frm2_thm_dist_bb_estimates}
2\arch\bigg{[}\cosinh{l}^{+}_{1}\cosinh \bigg{(}l^{+}_{1}+\arch\frac{e^{l^{+}_{1}}(l^{+}_{1})^{2}}{\varepsilon_{3}^{2}}\bigg{)}\bigg{]},
2\arch\bigg{[}\cosinh{l}^{-}_{1}\cosinh \bigg{(}l^{-}_{1}+\arch\frac{e^{l^{-}_{1}}(l^{-}_{1})^{2}}{\varepsilon_{3}^{2}}\bigg{)}\bigg{]},
\end{equation*}
\begin{equation*}\label{chbb_frm3_thm_dist_bb_estimates}
2\arch\bigg{[}\cosinh{l}^{+}_{2}\cosinh \bigg{(}l^{+}_{2}+\arch\frac{e^{l^{+}_{2}}(l^{+}_{2})^{2}}{\varepsilon_{3}^{2}}\bigg{)}\bigg{]},
2\arch\bigg{[}\cosinh{l}^{-}_{2}\cosinh \bigg{(}l^{-}_{2}+\arch\frac{e^{l^{-}_{2}}(l^{-}_{2})^{2}}{\varepsilon_{3}^{2}}\bigg{)}\bigg{]}\bigg{\}},
\end{equation*}
where the symbol $\varepsilon_{3}$ stands for the Margulis constant of hyperbolic space $\mathbb{H}^3$ (this constant will be defined shortly).
\end{theorem}

This result is of independent interest as well. Note that we do not require the regularity of surface metrics in Theorems~\ref{chbb_theorem_distance_between_pairs_of_boundaries_is_unibounded} and~\ref{chbb_theorem_distance_between_boundaries}.

Let us show how Theorem~\ref{chbb_theorem_distance_between_boundaries} implies Theorem~\ref{chbb_theorem_distance_between_pairs_of_boundaries_is_unibounded}.

\emph{Proof of Theorem~\ref{chbb_theorem_distance_between_pairs_of_boundaries_is_unibounded}.}

Consider two homotopically different nontrivial closed curves $c_{1}$ and $c_{2}$ on the surface $\mathcal{S}$ such that they intersect each other but do not intersect with the singular points of the metrics $h^{+}_{\infty}$ and $h^{-}_{\infty}$ on $\mathcal{S}$. Since the sequence of metrics $\{h^{+}_{n}\}_{n\in\mathbb{N}}$ converges to the metric $h^{+}_{\infty}$, the lengths $l^{+,n}_{1}$ of the curve $c_{1}\in\mathcal{S}$ measured in the metrics $h^{+}_{n}$, $n\in\mathbb{N}$, tend to the length $l^{+,\infty}_{1}>0$ of $c_{1}$ measured in the metric $h^{+}_{\infty}$ as $n\rightarrow\infty$. The converging sequence of the positive real numbers $\{l^{+,n}_{1}\}_{n\in\mathbb{N}}$ is bounded from below by a real number $\omega^{+}_{1}>0$ and from above by a real number $\Omega^{+}_{1}>0$. Similarly, the lengths $l^{-,n}_{1}$ of the curve $c_{1}\in\mathcal{S}$ measured in the metrics $h^{-}_{n}$, $n\in\mathbb{N}$, are bounded from below by some $\omega^{-}_{1}>0$ and from above by some $\Omega^{-}_{1}>0$; the lengths $l^{+,n}_{2}$ of the curve $c_{2}\in\mathcal{S}$ measured in the metrics $h^{+}_{n}$, $n\in\mathbb{N}$, are bounded from below by some $\omega^{+}_{2}>0$ and from above by some $\Omega^{+}_{2}>0$; and the lengths $l^{-,n}_{2}$ of the curve $c_{2}\in\mathcal{S}$ measured in the metrics $h^{-}_{n}$, $n\in\mathbb{N}$, are bounded from below by some $\omega^{-}_{2}>0$ and from above by some $\Omega^{-}_{2}>0$.

By Theorem~\ref{chbb_theorem_distance_between_boundaries}, the distance $d(\mathcal{S}^{+}_{n},\mathcal{S}^{-}_{n})$ between the surfaces $\mathcal{S}^{+}_{n}$ and $\mathcal{S}^{-}_{n}$ in the quasi-Fuchsian manifold $\mathcal{M}^{\circ}_{n}$ is uniformly bounded from above for any $n\in\mathbb{N}$:
\begin{equation*}\label{chbb_frm1_thm_uniform_dist_bb_estimates}
d(\mathcal{S}^{+}_{n},\mathcal{S}^{-}_{n})<
\max\bigg{\{}\bigg{(}\Omega^{+}_{1}+\Omega^{-}_{1}+\ln\frac{2\Omega^{+}_{1}}{\omega^{-}_{1}}\bigg{)}, \bigg{(}\Omega^{+}_{1}+\Omega^{-}_{1}+\ln\frac{2\Omega^{-}_{1}}{\omega^{+}_{1}}\bigg{)},
\bigg{(}\Omega^{+}_{2}+\Omega^{-}_{2}+
\ln\frac{2\Omega^{+}_{2}}{\omega^{-}_{2}}\bigg{)},
\end{equation*}
\begin{equation*}\label{chbb_frm2_thm_uniform_dist_bb_estimates} \bigg{(}\Omega^{+}_{2}+\Omega^{-}_{2}+\ln\frac{2\Omega^{-}_{2}}{\omega^{+}_{2}}\bigg{)},
2\arch\bigg{[}\cosinh{\Omega}^{+}_{1}\cosinh \bigg{(}\Omega^{+}_{1}+\arch\frac{e^{\Omega^{+}_{1}}(\Omega^{+}_{1})^{2}}{\varepsilon_{3}^{2}}\bigg{)}\bigg{]},
\end{equation*}
\begin{equation*}\label{chbb_frm3_thm_uniform_dist_bb_estimates}
2\arch\bigg{[}\cosinh{\Omega}^{-}_{1}\cosinh \bigg{(}\Omega^{-}_{1}+\arch\frac{e^{\Omega^{-}_{1}}(\Omega^{-}_{1})^{2}}{\varepsilon_{3}^{2}}\bigg{)}\bigg{]},
\end{equation*}
\begin{equation*}\label{chbb_frm4_thm_uniform_dist_bb_estimates}
2\arch\bigg{[}\cosinh{\Omega}^{+}_{2}\cosinh \bigg{(}\Omega^{+}_{2}+\arch\frac{e^{\Omega^{+}_{2}}(\Omega^{+}_{2})^{2}}{\varepsilon_{3}^{2}}\bigg{)}\bigg{]},
\end{equation*}
\begin{equation*}\label{chbb_frm5_thm_uniform_dist_bb_estimates}
2\arch\bigg{[}\cosinh{\Omega}^{-}_{2}\cosinh \bigg{(}\Omega^{-}_{2}+\arch\frac{e^{\Omega^{-}_{2}}(\Omega^{-}_{2})^{2}}{\varepsilon_{3}^{2}}\bigg{)}\bigg{]}\bigg{\}}.
\end{equation*}
$\square$

Our aim now is to demonstrate Theorem~\ref{chbb_theorem_distance_between_boundaries}. We will widely use the Margulis lemma to prove this fact. In the most general case the Margulis lemma reads as follows \cite[Theorem D.1.1, p.~134]{chaaa_BP2003}:

\textbf{General Margulis Lemma.} \emph{For every} $m\in\mathbb{N}$ \emph{there exists a constant} $\varepsilon_m\geq0$ \emph{such that for any properly discontinuous subgroup} $\Gamma$ \emph{of the group} $\mathcal{I}(\mathbb{H}^m)$ \emph{of isometries of} $\mathbb{H}^m$ \emph{and for any} $x\in\mathbb{H}^m$, \emph{the group} $\Gamma_{\varepsilon_m}(x)$ \emph{generated by the set} $F_{\varepsilon_m}(x)=\{\gamma\in\Gamma : \mathrm{d}_{\mathbb{H}^m}(x,\gamma(x))\leq\varepsilon_m\}$ \emph{is almost-nilpotent}, \emph{where} $\mathrm{d}_{\mathbb{H}^m}(\cdot,\cdot)$ \emph{stands for the distance in hyperbolic space} $\mathbb{H}^m$.

If we restrict the General Margulis Lemma to the case of the quasifuchsian isometries of hyperbolic $3$-space $\mathbb{H}^3$ which is interesting to us, then the lemma can be rewritten in this way \cite[Theorem B, p.~100]{chbb_Otal2001}:

\textbf{Margulis Lemma.} \emph{There is a universal constant} $\varepsilon_{3}>0$ \emph{such that for any properly discontinuous subgroup} $\Gamma$ \emph{of the group} $\mathcal{I}(\mathbb{H}^3)$ \emph{of isometries of} $\mathbb{H}^3$ \emph{if two closed simple intersecting curves} $\tilde{\gamma}_{1}$ \emph{and} $\tilde{\gamma}_{2}$ \emph{of the manifold} $\mathbb{H}^3/\Gamma$ \emph{have lengths less than} $\varepsilon_{3}$, \emph{then} $\tilde{\gamma}_{1}$ \emph{and} $\tilde{\gamma}_{2}$ \emph{are homotopically equivalent in} $\mathbb{H}^3/\Gamma$.

Hence, the main idea of the proof of Theorem~\ref{chbb_theorem_distance_between_boundaries} is to find a pair of closed simple intersecting curves inside $\mathcal{M}$ of lengths less than the Margulis constant $\varepsilon_{3}$ and such that they are not homotopically equivalent once the distance between $\mathcal{S}^{+}$ and $\mathcal{S}^{-}$ is big enough. Then, by the Margulis lemma, the curves under consideration ought to be homotopically equivalent, which leads us to a contradiction.
\begin{figure}[!h]
\begin{center}
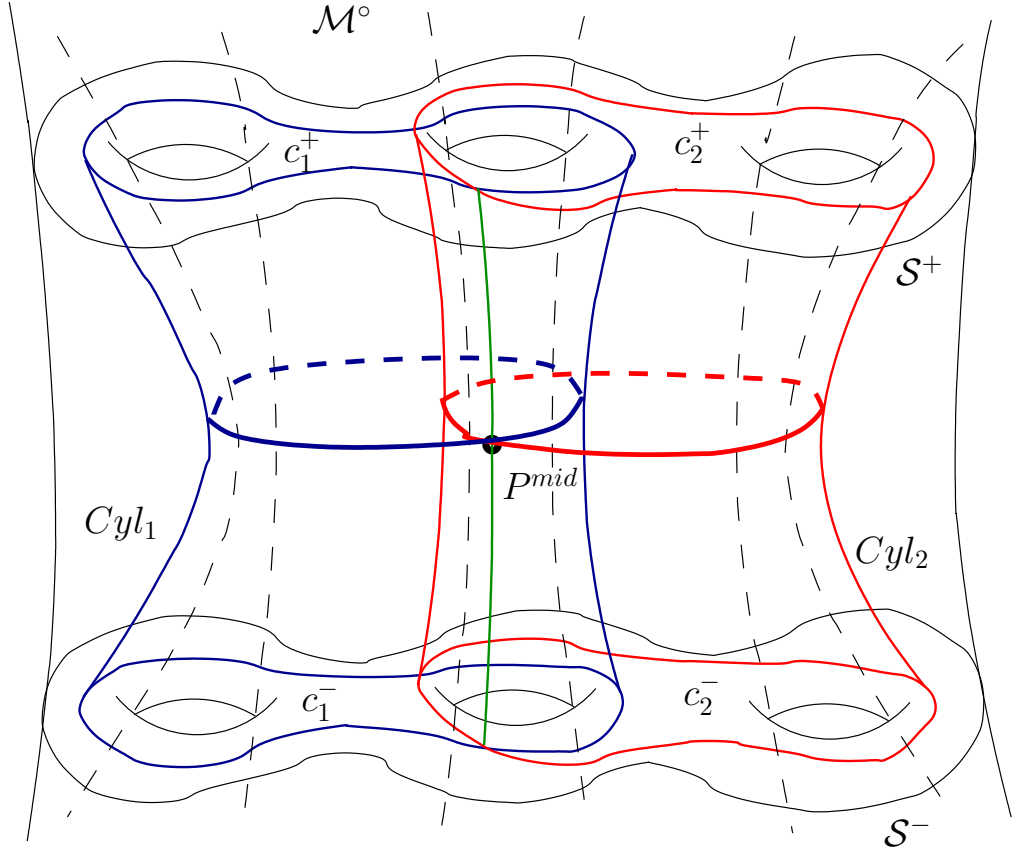
\caption{The cylinders $Cyl_1$ and $Cyl_2$ in the manifold $\mathcal{M}^{\circ}$.}\label{chbb_dist_bb_cylinders}
\end{center}
\end{figure}
Let us now give a more detailed plan of the proof of Theorem~\ref{chbb_theorem_distance_between_boundaries}:
\begin{itemize}
\item[$\bullet$] Suppose that the curves $c^{+}_{1}$ and $c^{+}_{2}$ intersect at a point $P^{+}$ (this point is not necessarily unique), and the curves $c^{-}_{1}$ and $c^{-}_{2}$ intersect at a point $P^{-}$. We will construct cylinders $Cyl_1$ and $Cyl_2$ in $\mathcal{M}$ that realize homotopies between $c^{+}_{1}$ and $c^{-}_{1}$ and between $c^{+}_{2}$ and $c^{-}_{2}$ correspondingly. Then the intersection of $Cyl_1$ and $Cyl_2$ contains a (curved) line with ends $P^{+}$ and $P^{-}$. Denote the midpoint of this line by $P^{mid}$.

\item[$\bullet$] We will find a constant based on $l^{+}_{1}$, $l^{-}_{1}$, $l^{+}_{2}$, $l^{-}_{2}$, and $\varepsilon_{3}$, and we will construct curves on $Cyl_1$ and $Cyl_2$ (see Fig.~\ref{chbb_dist_bb_cylinders}) passing through $P^{mid}$ such that if the distance between $\mathcal{S}^{+}$ and $\mathcal{S}^{-}$ is greater than the constant mentioned above then both constructed curves are shorter than $\varepsilon_{3}$.
\end{itemize}

\subsection{Construction of the cylinders $Cyl_1$ and $Cyl_2$}
\label{chbb_sec_construction_of_the_cylinders_Cyl_1_and_Cyl_2}

We consider a quasifuchsian manifold $\mathcal{M}^{\circ}$. By definition, it means that $\mathcal{M}^{\circ}$ is a quotient $\mathbb{H}^3/\Gamma^{\circ}$ where $\Gamma^{\circ}$ is a quasifuchsian subgroup of the group  $\mathcal{I}(\mathbb{H}^3)$ of isometries of hyperbolic $3$-space. Note that $\Gamma^{\circ}$ is homomorphic to the fundamental group $\pi_1(\mathcal{M}^{\circ})$.

Denote by $\gamma_1$ the closed geodesic of $\mathcal{M}^{\circ}$ homotopically equivalent to $c^{+}_{1}$ and $c^{-}_{1}$. Similarly, denote by $\gamma_2$ the closed geodesic of $\mathcal{M}^{\circ}$ homotopically equivalent to $c^{+}_{2}$ and $c^{-}_{2}$. By abuse of notation, we denote by $\gamma_1$ and $\gamma_2$ the elements of $\pi_1(\mathcal{M}^{\circ})$ corresponding to the closed geodesics under consideration. The universal covering of the domain $\mathcal{M}\subset\mathcal{M}^{\circ}$ is a convex simply connected subset $\widetilde{\mathcal{M}}$ of $\mathbb{H}^3$. Denote by $\tilde{\gamma}_1$ and $\tilde{\gamma}_2$ the isometries of $\mathbb{H}^3$ corresponding to the elements $\gamma_1$ and $\gamma_2$ of $\pi_1(\mathcal{M}^{\circ})$.

\begin{figure}[!h]
\begin{center}
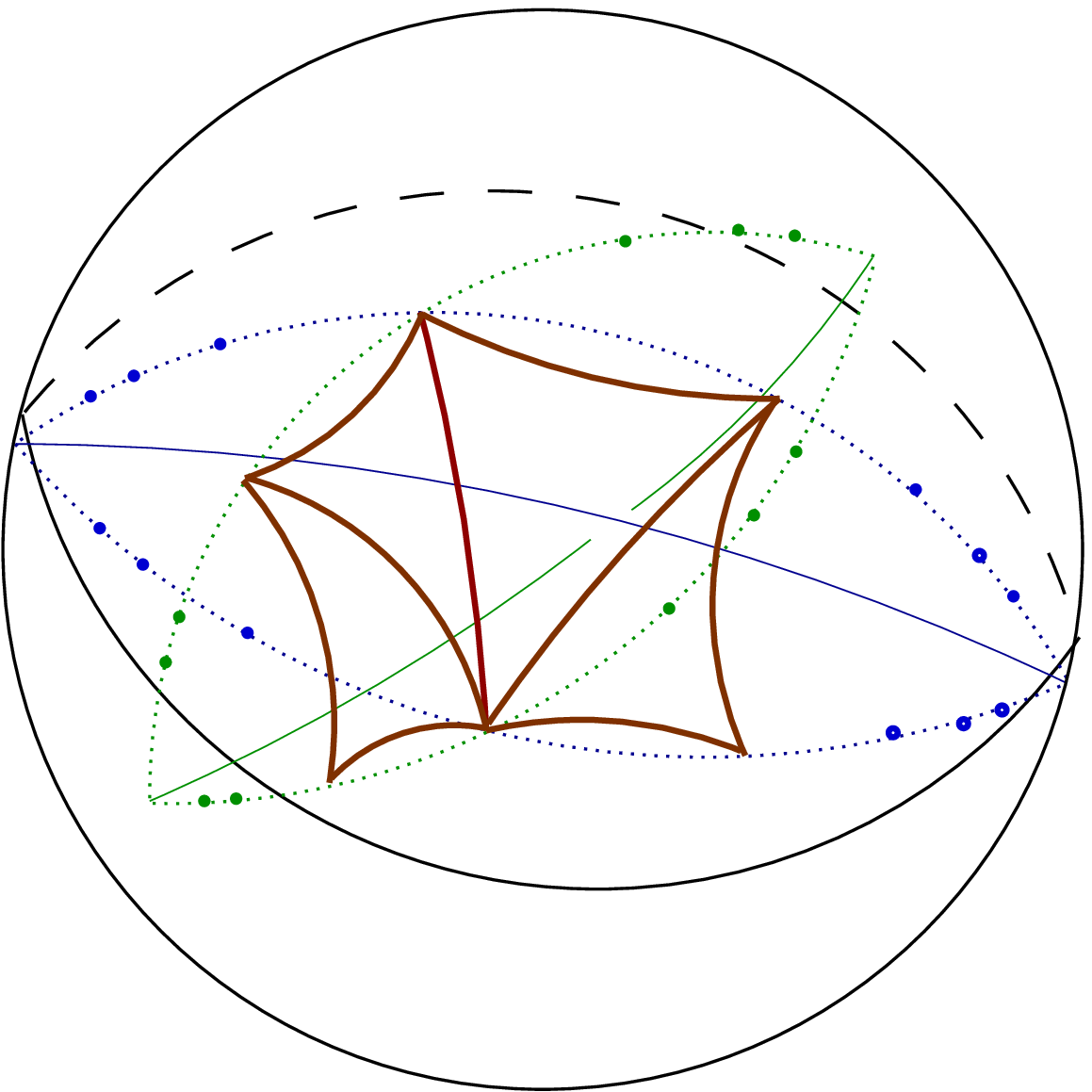
\caption{Construction of fundamental domains of the cylinders $Cyl_1$ and $Cyl_2$ in the Poincar\'e model of $\mathbb{H}^3$.}\label{chbb_dist_bb_constructing_cyl1_cyl2}
\end{center}
\end{figure}
Let us now consider any single point $\widetilde{P}^{+}_{0}\in\mathbb{H}^3$ serving as a pre-image of $P^{+} \in c^{+}_{1}\cap c^{+}_{2}$ in the universal covering $\widetilde{\mathcal{M}}$. Among all the points in the pre-image of $P^{-} \in c^{-}_{1}\cap c^{-}_{2}$ in $\widetilde{\mathcal{M}}$, we choose $\widetilde{P}^{-}_{0}\in\mathbb{H}^3$ to be the closest to $\widetilde{P}^{+}_{0}$ (in case there are several points realizing the minimal distance to $\widetilde{P}^{+}_{0}$, we choose one of them arbitrarily). Denote $\widetilde{P}^{+}_{1}\stackrel{\mathrm{def}}{=}\tilde{\gamma}_{1}.\widetilde{P}^{+}_{0}$, $\widetilde{P}^{-}_{1}\stackrel{\mathrm{def}}{=}\tilde{\gamma}_{1}.\widetilde{P}^{-}_{0}$,
$\widetilde{P}^{+}_{2}\stackrel{\mathrm{def}}{=}\tilde{\gamma}_{2}.\widetilde{P}^{+}_{0}$, $\widetilde{P}^{-}_{2}\stackrel{\mathrm{def}}{=}\tilde{\gamma}_{2}.\widetilde{P}^{-}_{0}$ (recall that for every point $T\in\mathbb{H}^3$ and for every $\tilde{\gamma}\in\mathcal{I}(\mathbb{H}^3)$ the symbol $\tilde{\gamma}.T$ stands for the image of $T$ under the isometry $\tilde{\gamma}$). Then we set the unions of flat hyperbolic triangles $\triangle\widetilde{P}^{+}_{0}\widetilde{P}^{-}_{0}\widetilde{P}^{+}_{1}\cup\triangle\widetilde{P}^{+}_{1}\widetilde{P}^{-}_{1}\widetilde{P}^{-}_{0}$ and $\triangle\widetilde{P}^{+}_{0}\widetilde{P}^{-}_{0}\widetilde{P}^{+}_{2}\cup\triangle\widetilde{P}^{+}_{2}\widetilde{P}^{-}_{2}\widetilde{P}^{-}_{0}$ in $\mathbb{H}^3$ to be fundamental domains of the cylinders $Cyl_1$ and $Cyl_2$ (see Fig.~\ref{chbb_dist_bb_constructing_cyl1_cyl2}).

The fundamental domain $\tilde{c}^{+}_{1}\subset\mathbb{H}^3$ of the curve $c^{+}_{1}$ has the same length $l^{+}_{1}$ as $c^{+}_{1}$. We can choose $\tilde{c}^{+}_{1}$ to connect $\widetilde{P}^{+}_{0}$ and $\widetilde{P}^{+}_{1}$. Hence, the length of the straight (hyperbolic) segment $\widetilde{P}^{+}_{0}\widetilde{P}^{+}_{1}$ is less than or equal to $l^{+}_{1}$. Similarly, $\mathrm{d}_{\mathbb{H}^3}(\widetilde{P}^{-}_{0},\widetilde{P}^{-}_{1})\leq l^{-}_{1}$, $\mathrm{d}_{\mathbb{H}^3}(\widetilde{P}^{+}_{0},\widetilde{P}^{+}_{2})\leq l^{+}_{2}$, and $\mathrm{d}_{\mathbb{H}^3}(\widetilde{P}^{-}_{0},\widetilde{P}^{-}_{2})\leq l^{-}_{2}$. Also, by construction, the midpoints $\widetilde{P}^{mid}_{0}$, $\widetilde{P}^{mid}_{1}$, and $\widetilde{P}^{mid}_{2}$ of the segments $\widetilde{P}^{+}_{0}\widetilde{P}^{-}_{0}$, $\widetilde{P}^{+}_{1}\widetilde{P}^{-}_{1}$, and $\widetilde{P}^{+}_{2}\widetilde{P}^{-}_{2}$ serve as pre-images of the midpoint $P^{mid}$ of the segment $P^{+}P^{-}$ lying in the intersection $Cyl_1 \cap Cyl_2$.

Evidently, $Cyl_1$ and $Cyl_2$ can be prolonged to realize homotopies between the pairs of closed curves $(c^{+}_{1},c^{-}_{1})$ and $(c^{+}_{2},c^{-}_{2})$ as it was announced in our plan, but it will not be needed further.

Let us study properties of the cylinders constructed alike $Cyl_1$ and $Cyl_2$.

\subsection{Properties of the cylinders of the type $Cyl$}
\label{chbb_sec_properties_of_the_cylinders_of_the_type_Cyl}

\textbf{Definition.} A cylinder $Cyl_0$ is said to be \emph{of the type} $Cyl$ if and only if $Cyl_0$ possesses
\begin{itemize}
\item[$1)$] a fundamental domain $FD(Cyl_0)\stackrel{\mathrm{def}}{=}\triangle \widetilde{R}^{+}\widetilde{R}^{-}\widetilde{Q}^{+} \cup\triangle \widetilde{Q}^{+}\widetilde{Q}^{-}\widetilde{R}^{-}$ constructed of two totally geodesic triangles in $\mathbb{H}^3$ such that $\mathrm{d}_{\mathbb{H}^3}(\widetilde{Q}^{+},\widetilde{Q}^{-})
=\mathrm{d}_{\mathbb{H}^3}(\widetilde{R}^{+},\widetilde{R}^{-})$, and
 \item[$2)$] the hyperbolic isometry $\tilde{\gamma}\in\mathcal{I}(\mathbb{H}^3)$ sending the geodesic segment $\widetilde{R}^{+}\widetilde{R}^{-}$ to the geodesic segment $\widetilde{Q}^{+}\widetilde{Q}^{-}$ and such that for every point $\widetilde{R}^{-}_{\sharp}\in
\{\tilde{\gamma}_{\sharp}.{\widetilde{R}}^{-}|\tilde{\gamma}_{\sharp}\in\langle\tilde{\gamma}\rangle\}$ the inequality $\mathrm{d}_{\mathbb{H}^3}(\widetilde{R}^{+},\widetilde{R}^{-})\leq
\mathrm{d}_{\mathbb{H}^3}(\widetilde{R}^{+},\widetilde{R}^{-}_{\sharp})$ holds true (here and below the symbol $\langle\tilde{\gamma}\rangle$ stands for the group generated by the element $\tilde{\gamma}$). Note that $\widetilde{Q}^{-}\in\{\tilde{\gamma}_{\sharp}.{\widetilde{R}}^{-}|\tilde{\gamma}_{\sharp}\in\langle\tilde{\gamma}\rangle\}$ by construction.
\end{itemize}

\begin{figure}[!h]
\begin{center}
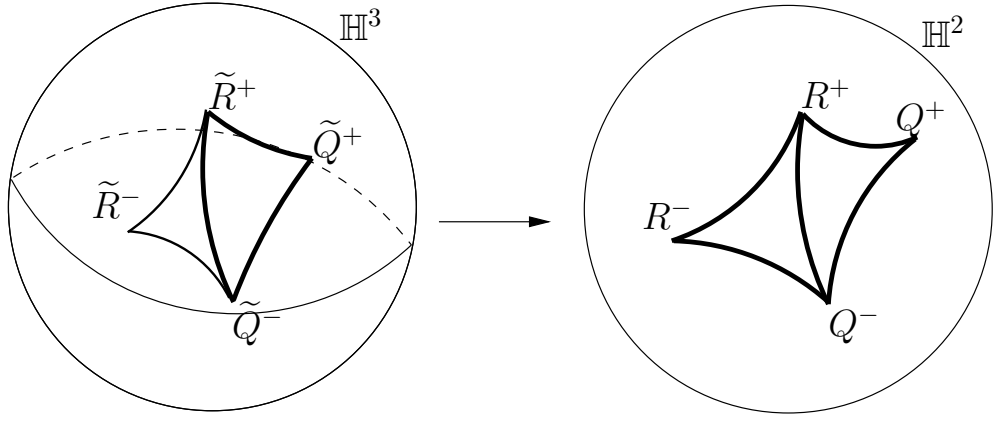
\caption{The quadrilaterals $\widetilde{R}^{+}\widetilde{R}^{-}\widetilde{Q}^{+}\widetilde{Q}^{-}$ in $\mathbb{H}^3$ and ${R}^{+}{R}^{-}{Q}^{+}{Q}^{-}$ in $\mathbb{H}^2$.}\label{chbb_dist_bb_flattening}
\end{center}
\end{figure}

Remark that the metric of $Cyl_0$ induced from the ambient space is hyperbolic. Let us flatten $FD(Cyl_0)$ and obtain a hyperbolic quadrilateral ${R}^{+}{R}^{-}{Q}^{+}{Q}^{-}\subset\mathbb{H}^2$ isometric to $FD(Cyl_0)$ such that the vertices with tildes in $\mathbb{H}^3$ correspond to the vertices of the same name but without tildes in $\mathbb{H}^2$ (see Fig.~\ref{chbb_dist_bb_flattening}).

The quadrilateral ${R}^{+}{R}^{-}{Q}^{+}{Q}^{-}$ serves as a fundamental domain of $Cyl_0$ in its universal covering in $\mathbb{H}^2$. Denote by $\chi_{R}$ and $\chi_{Q}$ the hyperbolic straight lines in $\mathbb{H}^2$ containing the segments ${R}^{+}{R}^{-}$ and ${Q}^{+}{Q}^{-}$ correspondingly. Remark that the connected domain of $\mathbb{H}^2$ between $\chi_{R}$ and $\chi_{Q}$ is actually a fundamental domain of the unbounded hyperbolic cylinder $Cyl^{\circ}_{0}$ containing $Cyl_0$. We will call it $FD(Cyl^{\circ}_{0})$. Indeed, the fundamental group $\pi_{1}(Cyl^{\circ}_{0})=\mathbb{Z}$. Hence, $Cyl^{\circ}_{0}$ possesses a closed geodesic $\chi^{\circ}$ and there is a hyperbolic straight line $\chi$ in $\mathbb{H}^2$ serving as a lift of $\chi^{\circ}$ and related to the isometry $\bar{\chi}$ of $\mathbb{H}^2$ such that $Cyl^{\circ}_{0}=\mathbb{H}^2/\langle\bar{\chi}\rangle$. We show the existence of such geodesic $\chi$ in the following
\begin{lemma}\label{chbb_lemma_existence_of chi}
Consider two nonintersecting geodesics $\chi_{R}$ and $\chi_{Q}$ in $\mathbb{H}^2$ which are not asymptotic, with marked points $R\in\chi_{R}$ and $Q\in\chi_{Q}$. There is a unique hyperbolic straight line $\chi$ in $\mathbb{H}^2$ such that the angles of intersection of $\chi$ with $\chi_{R}$ and $\chi_{Q}$ are equal, and moreover, if we denote $R'\stackrel{\mathrm{def}}{=}\chi_{R}\cap\chi$ and $Q'\stackrel{\mathrm{def}}{=}\chi_{Q}\cap\chi$, then $\mathrm{d}_{\mathbb{H}^2}(R,R')=\mathrm{d}_{\mathbb{H}^2}(Q,Q')$ and the points $R$ and $Q$ lie in the same half-plane with respect to $\chi$.
\end{lemma}

\begin{proof}
Let us consider the Beltrami-Klein model $\mathbb{K}^2$ of the hyperbolic plane $\mathbb{H}^2$. Recall that $\mathbb{K}^2$ is a unit disc in the Euclidean plane $\mathbb{R}^2$ and all geodesics of $\mathbb{K}^2$ are restrictions of Euclidean straight lines on this disc. Without loss of generality the geodesics $\chi_{R}\subset\mathbb{K}^2$ and $\chi_{Q}\subset\mathbb{K}^2$ can be taken symmetric with respect to the axis $Ox$ of the cartesian coordinate system on $\mathbb{R}^2$, both at an arbitrary distance $\zeta$ from $Ox$. Let $\chi_{R}$ lie in the upper half-space of $\mathbb{R}^2$ with respect to $Ox$ and $\chi_{Q}$ lie in the lower half-space of $\mathbb{R}^2$ with respect to $Ox$. At last we fix arbitrary points $R\in\chi_{R}$ and $Q\in\chi_{Q}$.

By construction, every geodesic in $\mathbb{K}^2$ passing through the origin $O$ of the cartesian coordinate system on $\mathbb{R}^2$ either intersects $\chi_{R}$ and $\chi_{Q}$ at the same angle or does not intersect them. Let us consider a family ${\Phi}_{\tau}$ of such geodesics $R_{\tau}Q_{\tau}$ lying between the straight lines $OR$ and $OQ$ where $R_{\tau}\in\chi_{R}$, $Q_{\tau}\in\chi_{Q}$, $\tau$ stands for the hyperbolic distance between $R$ and $R_{\tau}$, and the line $OQ\in{\Phi}_{\tau}$ corresponds to the value $\hat{\tau}$ of the parameter $\tau$.

Note that
\begin{itemize}
\item $R$ and $Q$ lie in the same half-plane with respect to any $R_{\tau}Q_{\tau}\in{\Phi}_{\tau}$.
\item As $\tau$ grows up monotonically from $0$ to $\hat{\tau}$, the distance $\mathrm{d}_{\mathbb{H}^2}(Q,Q_{\tau})$ decreases monotonically from $\mathrm{d}_{\mathbb{H}^2}(Q,Q_{\hat{\tau}})$ to $0$. Hence, there exists a unique $\tau_{0}\in[0,\hat{\tau}]$ such that $\mathrm{d}_{\mathbb{H}^2}(R,R_{\tau_{0}})=\mathrm{d}_{\mathbb{H}^2}(Q,Q_{\tau_{0}})$.
\end{itemize}
We choose $\chi$ to be $R_{\tau_{0}}Q_{\tau_{0}}\in{\Phi}_{\tau}$. $\chi$ is unique since $\tau_{0}$ is unique.
\end{proof}

\begin{remark}\label{chbb_remark_P+P-_is_short}
Let $Set({R}^{-})\stackrel{\mathrm{def}}{=}\{\bar{\chi}_{\sharp}.{R}^{-}
|\bar{\chi}_{\sharp}\in\langle\bar{\chi}\rangle\}$ \emph{(}by construction, ${Q}^{-}\in Set({R}^{-})$\emph{)}. Then for every point ${R}^{-}_{\sharp}\in Set({R}^{-})$ the inequality $\mathrm{d}_{\mathbb{H}^2}({R}^{+},{R}^{-})\leq
\mathrm{d}_{\mathbb{H}^2}({R}^{+},{R}^{-}_{\sharp})$ holds true.
\end{remark}

\begin{proof}
By construction, $\mathrm{d}_{\mathbb{H}^3}(\widetilde{R}^{+},\widetilde{R}^{-})=\mathrm{d}_{\mathbb{H}^2}({R}^{+},{R}^{-})$, and the surfaces $\langle\bar{\chi}\rangle.{R}^{+}{R}^{-}{Q}^{+}{Q}^{-}\subset\mathbb{H}^2$ (which is the union $\bigcup_{\bar{\chi}_{\sharp}\in\langle\bar{\chi}\rangle}\bar{\chi}_{\sharp}.{R}^{+}{R}^{-}{Q}^{+}{Q}^{-}$ of the quadrilaterals $\bar{\chi}_{\sharp}.{R}^{+}{R}^{-}{Q}^{+}{Q}^{-}$ isometric to ${R}^{+}{R}^{-}{Q}^{+}{Q}^{-}$) and $\langle\bar{\chi}\rangle.FD(Cyl_0)\subset\mathbb{H}^3$ are isometric in their intrinsic metrics. Evidently, for any points $\widetilde{T}_{1}$ and $\widetilde{T}_{2}$ in $\langle\bar{\chi}\rangle.FD(Cyl_0)$ it is true that $\mathrm{d}_{\mathbb{H}^3}(\widetilde{T}_{1},\widetilde{T}_{2})\leq
\mathrm{d}^{int}_{\langle\bar{\chi}\rangle.FD(Cyl_0)}(\widetilde{T}_{1},\widetilde{T}_{2})$, where $\mathrm{d}^{int}_{\langle\bar{\chi}\rangle.FD(Cyl_0)}(\cdot,\cdot)$ stands for the intrinsic metric of $\langle\bar{\chi}\rangle.FD(Cyl_0)$. At last, the part~$2)$ of the definition of a cylinder $Cyl_0$ of the type $Cyl$ allows us to conclude that Remark~\ref{chbb_remark_P+P-_is_short} is valid.
\end{proof}

\begin{remark}\label{chbb_remark_P'Q'_in_or_out_of_quadrangle}
Let $R'Q'$ be a segment of the geodesic $\chi\subset\mathbb{H}^2$ between $\chi_{R}$ and $\chi_{Q}$ serving as a fundamental domain of $\chi^{\circ}\subset Cyl^{\circ}_{0}$ on $\chi$ (here $R'\in\chi_{R}$ and $Q'\in\chi_{Q}$). Then either $R'Q'\subset {R}^{+}{R}^{-}{Q}^{+}{Q}^{-}$ or $R'Q'\cap {R}^{+}{R}^{-}{Q}^{+}{Q}^{-}=\emptyset$.
\end{remark}
\begin{proof}
Recall that the points ${R}^{+}$ and ${Q}^{+}$ are pre-images in $\mathbb{H}^2$ of the same point on $Cyl_{0}$, and one can be obtained from another by applying an isometry of $\mathbb{H}^2$ which is an element of the group $\langle\bar{\chi}\rangle$ preserving the straight hyperbolic line $\chi$. Hence, ${R}^{+}$ and ${Q}^{+}$ lie in one half-plane of $\mathbb{H}^2$ with respect to $\chi$ and, by consequence, the segment ${R}^{+}{Q}^{+}$ does not intersect $\chi$. Similarly, ${R}^{-}{Q}^{-}\cap\chi=\emptyset$.

We conclude that if ${R}^{+}{Q}^{+}$ and ${R}^{-}{Q}^{-}$ lie in the same half-plane of $\mathbb{H}^2$ with respect to $\chi$ then $R'Q'\cap {R}^{+}{R}^{-}{Q}^{+}{Q}^{-}=\emptyset$. Otherwise, if ${R}^{+}{Q}^{+}$ and ${R}^{-}{Q}^{-}$ lie in different half-planes with respect to $\chi$, then $R'Q'\subset {R}^{+}{R}^{-}{Q}^{+}{Q}^{-}$.
\end{proof}

\subsection{$h$-neighborhood of a geodesic in $\mathbb{H}^2$}
\label{chbb_sec_h-neighborhood_of_a_geodesic_in_H_2}

In this section, we study hyperbolic quadrilaterals of one special type and half-neighborhoods of geodesics containing one of the sides of our quadrilaterals which are inscribed in and circumscribed about these quadrilaterals. Properties of these objects will be largely used in obtaining bounds on a possible size of cylinders of the type $Cyl$.

The object of our interest is a quadrilateral $O_{R}O_{Q}{R}{Q}\subset\mathbb{H}^2$ with the sides $\mathrm{d}_{\mathbb{H}^2}(O_{R}O_{Q})=l$, $\mathrm{d}_{\mathbb{H}^2}(R,Q)=l'$, and $\mathrm{d}_{\mathbb{H}^2}(O_{R},R)=\mathrm{d}_{\mathbb{H}^2}(O_{Q},Q)=h'$, such that the edges $O_{R}R$ and $O_{Q}Q$ are perpendicular to $O_{R}O_{Q}$. Draw a curve $\gamma_h$ at a distance $h<h'$ from the geodesic containing $O_{R}O_{Q}$ such that $\gamma_h$ intersects $O_{R}R$ and $O_{Q}Q$ at points $T$ and $T'$ correspondingly. Denote a segment of $\gamma_h$ between $O_{R}R$ and $O_{Q}Q$ by $\widehat{TT'}$, and the hyperbolic length of $\widehat{TT'}$ by $l_{h}$.

A direct calculation shows that

\begin{remark}\label{chbb_remark_length_of_TT'}
The following relation holds true:
\begin{equation*}\label{chbb_frm_length_of_TT'}
l_{h}=l\cosinh h.
\end{equation*}
\end{remark}

\begin{remark}\label{chbb_remark_TT'_outside_OO'PP'}
If $h=h'$ then $T$ and $T'$ coincide with $R$ and $Q$, $\widehat{TT'}$ intersects $O_{R}O_{Q}{R}{Q}$ as a solid body only at its ends $R$ and $Q$, and, evidently, $l_{h'}>l'$ (any path connecting two points can not be shorter then a geodesic segment between them).
\end{remark}

\begin{remark}\label{chbb_remark_TT'_inside_OO'PP'}
Suppose that $h'>l'$. If $h\leq h'-l'$ then $\widehat{TT'}\subset O_{R}O_{Q}{R}{Q}$ and $l_{h}<l'$.
\end{remark}

\begin{proof}
Consider hyperbolic balls $B_{l'}(R)$ and $B_{l'}(Q)$ of the radius $l'$ with the centers $R$ and $Q$. These balls contain the segment $RQ$. Also, $B_{l'}(R)$ and $B_{l'}(Q)$ are perpendicular to $O_{R}{R}$ and $O_{Q}{Q}$ correspondingly. By construction, $\widehat{TT'}$ is perpendicular to $O_{R}{R}$ and $O_{Q}{Q}$ as well. Moreover, $\widehat{TT'}$ is a convex curve. Hence, $\widehat{TT'}$ lies outside the interior of $B_{l'}(R)$ and $B_{l'}(Q)$ for $h\leq h'-l'$. It means that the geodesic segment $RQ$ does not intersect $\widehat{TT'}$, and $\widehat{TT'}\subset O_{R}O_{Q}{R}{Q}$.

Denote by $O_{R}O_{Q}\widehat{TT'}$ the convex domain in $\mathbb{H}^2$ bounded by the segments $O_{R}T$, $O_{R}O_{Q}$, $O_{Q}T'$ and the curve $\widehat{TT'}$. By construction, the orthogonal projection of $RQ$ onto $O_{R}O_{Q}\widehat{TT'}$ is $\widehat{TT'}$. Since the orthogonal projection on the boundary of a convex hyperbolic domain is contracting \cite[p.~9]{chbb_BaGrSc1985} (see also \cite[II.1.3.4, p.~124]{chaaa_CaEpGr2006}), we get $l_{h}<l'$.
\end{proof}

Also, we need
\begin{lemma}\label{chbb_lemma_h_ort_int}
Let us consider a quadrilateral $O_{R}O_{Q}{R}{Q}$ as in Section~\ref{chbb_sec_h-neighborhood_of_a_geodesic_in_H_2} with the fixed length $l_{RQ}$ of the edge $RQ$. There is a constant
\begin{equation*}\label{chbb_frm_h_ort_int_by_l_RQ-sit_1_orth}
h^{ort}_{int}=l_{RQ}+\arch\frac{e^{l_{RQ}}l_{RQ}^{2}}{\varepsilon_{3}^{2}}.
\end{equation*}
such that if the length $h_{RQ}$ of the sides $O_{R}{R}$ and $O_{Q}{Q}$ is greater than $h^{ort}_{int}$ then the length of the path $\widehat{{T}_{R}{T}_{Q}}$ at the distance $h_{T}\stackrel{\mathrm{def}}{=}h_{RQ}/2$ from $O_{R}O_{Q}$ connecting the midpoints ${T}_{R}$ and ${T}_{Q}$ of $O_{R}{R}$ and $O_{Q}{Q}$ is smaller than the Margulis constant $\varepsilon_{3}$.
\end{lemma}
\begin{proof}
Denote by $l_{O}$ the length of $O_{R}O_{Q}$. Once $l_{RQ}$ is fixed, suppose that $h_{RQ}$ can be arbitrarily big, in particular, bigger than $l_{RQ}$.

There are points ${T}_{R}'\in O_{R}{R}$ and ${T}_{Q}'\in O_{Q}{Q}$ at the distance $h_{T}'$ from $O_{R}$ and $O_{Q}$ correspondingly, such that the length of the path $\widehat{{T}_{R}'{T}_{Q}'}$ as in Section~\ref{chbb_sec_h-neighborhood_of_a_geodesic_in_H_2} is equal to $\varepsilon_{3}$. By Remark~\ref{chbb_remark_length_of_TT'},
\begin{equation}\label{chbb_frm_def_h_T_'-sit_1_orth}
l_{O}\cosinh h_{T}'=\varepsilon_{3}.
\end{equation}
Indeed, if ${T}_{R}'$ and ${T}_{Q}'$ do not exist then
\begin{equation}\label{chbb_frm_l_O>e3-sit_1_orth}
l_{O}>\varepsilon_{3}.
\end{equation}
By Remarks~\ref{chbb_remark_length_of_TT'} and~\ref{chbb_remark_TT'_inside_OO'PP'} applied to the quadrilateral $O_{R}O_{Q}{R}{Q}$,
\begin{equation}\label{chbb_frm_h_RQ_by_l_O_and_l_RQ-sit_1_orth}
l_{O}\cosinh (h_{RQ}-l_{RQ})<l_{RQ}.
\end{equation}
Mixing~(\ref{chbb_frm_l_O>e3-sit_1_orth}) and~(\ref{chbb_frm_h_RQ_by_l_O_and_l_RQ-sit_1_orth}), we get
\begin{equation*}\label{chbb_frm1_h_RQ_by_e3_and_l_RQ-sit_1_orth}
\varepsilon_{3}\cosinh (h_{RQ}-l_{RQ})<l_{RQ},
\end{equation*}
\begin{equation*}\label{chbb_frm2_h_RQ_by_e3_and_l_RQ-sit_1_orth}
h_{RQ}<l_{RQ}+\arch\frac{l_{RQ}}{\varepsilon_{3}},
\end{equation*}
which leads us to a contradiction with the unboundedness of $h_{RQ}$.

The length of $\widehat{{T}_{R}{T}_{Q}}$ is less than the length $\varepsilon_{3}$ of $\widehat{{T}_{R}'{T}_{Q}'}$  when the inequality
\begin{equation}\label{chbb_frm_h_T'>_h_RQ/2-sit_1_orth}
h_{T}'>h_{T}\Big(=\frac{h_{RQ}}{2}\Big)
\end{equation}
is satisfied, which is equivalent to the validity of
\begin{equation*}\label{chbb_frm_cosh_h_T'>cosh_h_RQ/2-sit_1_orth}
\cosinh h_{T}'>\cosinh \frac{h_{RQ}}{2},
\end{equation*}
and, by~(\ref{chbb_frm_def_h_T_'-sit_1_orth}), is also equivalent to
\begin{equation}\label{chbb_frm_e3/l_O>cosh_h_RQ/2-sit_1_orth}
\frac{\varepsilon_{3}}{l_{O}}>\cosinh \frac{h_{RQ}}{2}.
\end{equation}

Due to the following property of the hyperbolic cosine: $\cosinh 2x={\cosinh}^2 x+{\sinush}^2 x$, we see that
\begin{equation*}\label{chbb_frm_cosh_2h>cosh^2_h-sit_1_orth}
{\cosinh}^2 \Big(\frac{h_{RQ}}{2}\Big)\leq\cosinh h_{RQ}.
\end{equation*}
Hence, the validity of the formula
\begin{equation}\label{chbb_frm2_e3/l_O>cosh_h_RQ/2-sit_1_orth}
\cosinh{h_{RQ}}<\frac{\varepsilon_{3}^{2}}{l_{O}^{2}}
\end{equation}
implies the validity of~(\ref{chbb_frm_e3/l_O>cosh_h_RQ/2-sit_1_orth}).

Let us exclude $l_{O}$ from~(\ref{chbb_frm2_e3/l_O>cosh_h_RQ/2-sit_1_orth}) with the help of~(\ref{chbb_frm_h_RQ_by_l_O_and_l_RQ-sit_1_orth}).

At first, we perform a series of modifications of~(\ref{chbb_frm_h_RQ_by_l_O_and_l_RQ-sit_1_orth}). By the formula for the hyperbolic cosine of the sum of two angles, we get
\begin{equation*}\label{chbb_frm3_h_RQ_by_l_O_and_l_RQ-sit_1_orth}
\cosinh h_{RQ} \cosinh l_{RQ}-\sinush h_{RQ} \sinush l_{RQ}<\frac{l_{RQ}}{l_{O}}.
\end{equation*}
Then, as $\sinush x>0$ for each $x>0$, and because $\cosinh x>\sinush x$ and $\cosinh x>0$ for all $x\in\mathbb{R}$, we obtain
\begin{equation*}\label{chbb_frm4_h_RQ_by_l_O_and_l_RQ-sit_1_orth}
\cosinh h_{RQ} (\cosinh l_{RQ}-\sinush l_{RQ})<\frac{l_{RQ}}{l_{O}},
\end{equation*}
and the definitions of the hyperbolic sine and cosine,
\begin{equation}\label{chbb_frm_sinh_and_cosh_defs-sit_1_orth}
\sinush x=\frac{e^{x}-e^{-x}}{2} \quad\text{and}\quad \cosinh x=\frac{e^{x}+e^{-x}}{2},
\end{equation}
imply
\begin{equation*}\label{chbb_frm5_h_RQ_by_l_O_and_l_RQ-sit_1_orth}
\cosinh h_{RQ}<\frac{e^{l_{RQ}}l_{RQ}}{l_{O}}.
\end{equation*}
It means that the validity of the formula
\begin{equation}\label{chbb_frm1_l_O_by_l_RQ-sit_1_orth}
\frac{e^{l_{RQ}}l_{RQ}}{l_{O}}<\frac{\varepsilon_{3}^{2}}{l_{O}^{2}}
\end{equation}
implies the validity of~(\ref{chbb_frm2_e3/l_O>cosh_h_RQ/2-sit_1_orth}).
We rewrite the condition~(\ref{chbb_frm1_l_O_by_l_RQ-sit_1_orth}) in a more convenient form:
\begin{equation}\label{chbb_frm2_l_O_by_l_RQ-sit_1_orth}
l_{O}<\frac{\varepsilon_{3}^{2}}{e^{l_{RQ}}l_{RQ}}.
\end{equation}
By~(\ref{chbb_frm_h_RQ_by_l_O_and_l_RQ-sit_1_orth}), we know that
\begin{equation*}\label{chbb_frm6_h_RQ_by_l_O_and_l_RQ-sit_1_orth}
l_{O}<\frac{l_{RQ}}{\cosinh (h_{RQ}-l_{RQ})}.
\end{equation*}
Hence, the validity of
\begin{equation}\label{chbb_frm1_h_RQ_by_l_RQ-sit_1_orth}
\frac{l_{RQ}}{\cosinh (h_{RQ}-l_{RQ})}<\frac{\varepsilon_{3}^{2}}{e^{l_{RQ}}l_{RQ}}
\end{equation}
implies the validity of~(\ref{chbb_frm1_l_O_by_l_RQ-sit_1_orth}).

We can now conclude that the condition
\begin{equation*}\label{chbb_frm2_h_RQ_by_l_RQ-sit_1_orth}
h_{RQ}>h^{ort}_{int}
\end{equation*}
obtained from~(\ref{chbb_frm1_h_RQ_by_l_RQ-sit_1_orth}) implies~(\ref{chbb_frm_h_T'>_h_RQ/2-sit_1_orth}).
\end{proof}

\subsection{Fundamental domains of $Cyl_1$ and $Cyl_2$ in $\mathbb{H}^2$}
\label{chbb_sec_fund_domains_of_Cyl_1_and_Cyl_2_in H^2}

Following the construction of a fundamental domain of a cylinder of the type $Cyl$ in $\mathbb{H}^2$ from Section~\ref{chbb_sec_properties_of_the_cylinders_of_the_type_Cyl}, we define for the cylinder $Cyl_1$ its fundamental domain ${P}^{+}_{0}{P}^{-}_{0}{P}^{+}_{1}{P}^{-}_{1}\subset\mathbb{H}^2_1$, where $\mathbb{H}^2_1$ is just a copy of the hyperbolic plane $\mathbb{H}^2$. We denote by $\chi_{{P}_{0}}$ and $\chi_{{P}_{1}}$ the hyperbolic straight lines in $\mathbb{H}^2_1$ containing the segments ${P}^{+}_{0}{P}^{-}_{0}$ and ${P}^{+}_{1}{P}^{-}_{1}$ correspondingly. Following the content of Section~\ref{chbb_sec_h-neighborhood_of_a_geodesic_in_H_2}, we find the hyperbolic segment $O_{0}O_{1}\subset\mathbb{H}^2_1$ corresponding to the element $\gamma_1$ of the fundamental group $\pi_1(\mathcal{M}^{\circ})$ (see Section~\ref{chbb_sec_construction_of_the_cylinders_Cyl_1_and_Cyl_2}) with the points $O_{0}\in\chi_{{P}_{0}}$ and $O_{1}\in\chi_{{P}_{1}}$.

Similarly, we define the quadrilateral ${P}^{+}_{0}{P}^{-}_{0}{P}^{+}_{2}{P}^{-}_{2}\subset\mathbb{H}^2_2$ to be a fundamental domain of the cylinder $Cyl_2$, where $\mathbb{H}^2_2$ is another copy of $\mathbb{H}^2$. Denote by $\chi_{{P}_{0}}$ and $\chi_{{P}_{2}}$ the geodesics in $\mathbb{H}^2_2$ containing ${P}^{+}_{0}{P}^{-}_{0}$ and ${P}^{+}_{2}{P}^{-}_{2}$ correspondingly. We also find the hyperbolic segment $O_{0}O_{2}\subset\mathbb{H}^2_2$ corresponding to $\gamma_2\in\pi_1(\mathcal{M}^{\circ})$ with the points $O_{0}\in\chi_{{P}_{0}}$ and $O_{2}\in\chi_{{P}_{2}}$.

An attentive reader has already remarked the following abuse of notation: the geodesic $\chi_{{P}_{0}}$ with the points ${P}^{+}_{0}$, ${P}^{-}_{0}$, and $O_{0}$ on it lie both in $\mathbb{H}^2_1$ and $\mathbb{H}^2_2$ as if these copies $\mathbb{H}^2_1$ and $\mathbb{H}^2_2$ of the hyperbolic plane intersect at $\chi_{{P}_{0}}$. It is very logic since the segment ${P}^{+}_{0}{P}^{-}_{0}\subset\chi_{{P}_{0}}$ corresponds to the segment ${P}^{+}{P}^{-}$ in the intersection of the cylinders $Cyl_1$ and $Cyl_2$ related to $\mathbb{H}^2_1$ and $\mathbb{H}^2_2$.

%\begin{figure}
%\input{./figures/dist_bb_situations/dist_bb_situation1.pstex_t}
%\hfill
%\input{./figures/dist_bb_situations/dist_bb_situation2.pstex_t} \\
%\parbox[t]{7cm}{\caption{The quadrilateral ${P}^{+}_{0}{P}^{-}_{0}{P}^{+}_{i}{P}^{-}_{i}$, $i=1,2$, in Situation %1.}\label{dist_bb_situation1}} \hfill
%\parbox[t]{7cm}{\caption{The quadrilateral ${P}^{+}_{0}{P}^{-}_{0}{P}^{+}_{i}{P}^{-}_{i}$, $i=1,2$, in Situation %2.}\label{dist_bb_situation2}}
%\end{figure}

We are now prepared to prove Theorem~\ref{chbb_theorem_distance_between_boundaries}. In order to do this, according to Remark~\ref{chbb_remark_P'Q'_in_or_out_of_quadrangle} we must consider two separate situations.
\begin{itemize}

\item[\textbf{Situation 1.}] If for both cylinders $Cyl_1$ and $Cyl_2$ their fundamental domains ${P}^{+}_{0}{P}^{-}_{0}{P}^{+}_{1}{P}^{-}_{1}\subset\mathbb{H}^2_1$ and ${P}^{+}_{0}{P}^{-}_{0}{P}^{+}_{2}{P}^{-}_{2}\subset\mathbb{H}^2_1$ contain the segments $O_{0}O_{1}$ and $O_{0}O_{2}$ correspondingly (see Fig.~\ref{dist_bb_situation1}), then the distance between the surfaces $\mathcal{S}^{+}$ and $\mathcal{S}^{-}$ from the statement of Theorem~\ref{chbb_theorem_distance_between_boundaries} is bounded from above due to the Margulis lemma.

   Indeed, recall that $P^{mid}$ is the midpoint of the segment ${P}^{+}{P}^{-}\subset Cyl_1\cap Cyl_2$, then the midpoints $P^{mid}_{0}$, $P^{mid}_{1}$, and $P^{mid}_{2}$ of the segments ${P}^{+}_{0}{P}^{-}_{0}\subset\chi_{{P}_{0}}$, ${P}^{+}_{1}{P}^{-}_{1}\subset\chi_{{P}_{1}}$, and ${P}^{+}_{2}{P}^{-}_{2}\subset\chi_{{P}_{2}}$ are the pre-images of $P^{mid}$ in ${P}^{+}_{0}{P}^{-}_{0}{P}^{+}_{1}{P}^{-}_{1}\subset\mathbb{H}^2_1$ or ${P}^{+}_{0}{P}^{-}_{0}{P}^{+}_{2}{P}^{-}_{2}\subset\mathbb{H}^2_2$.  Following the content of Section~\ref{chbb_sec_h-neighborhood_of_a_geodesic_in_H_2}, we construct the paths $\widehat{P^{mid}_{0}P^{mid}_{1}}\subset\mathbb{H}^2_1$ and $\widehat{P^{mid}_{0}P^{mid}_{2}}\subset\mathbb{H}^2_2$ connecting $P^{mid}_{0}$ with $P^{mid}_{1}$ and $P^{mid}_{2}$, and lying at the distance $\mathrm{d}_{\mathbb{H}^2}(P^{mid}_{0},O_{0})$ from $O_{0}O_{1}$ and $O_{0}O_{2}$. We will demonstrate that, once the distance between $\mathcal{S}^{+}$ and $\mathcal{S}^{-}$ (consequently, the hyperbolic length of ${P}^{+}{P}^{-}$) is bigger then a constant depending on ${l}^{+}_{1}$, ${l}^{-}_{1}$, ${l}^{+}_{2}$, and ${l}^{-}_{2}$ (see Section~\ref{chbb_sec_construction_of_the_cylinders_Cyl_1_and_Cyl_2} for definitions), then two intersecting homotopically different curves in $\mathcal{M}$ with fundamental domains $\widehat{P^{mid}_{0}P^{mid}_{1}}\subset\mathbb{H}^2_1$ and $\widehat{P^{mid}_{0}P^{mid}_{2}}\subset\mathbb{H}^2_2$ have the lengths less than the Margulis constant $\varepsilon_{3}$, which is impossible.

\begin{figure}[!h]
\begin{center}
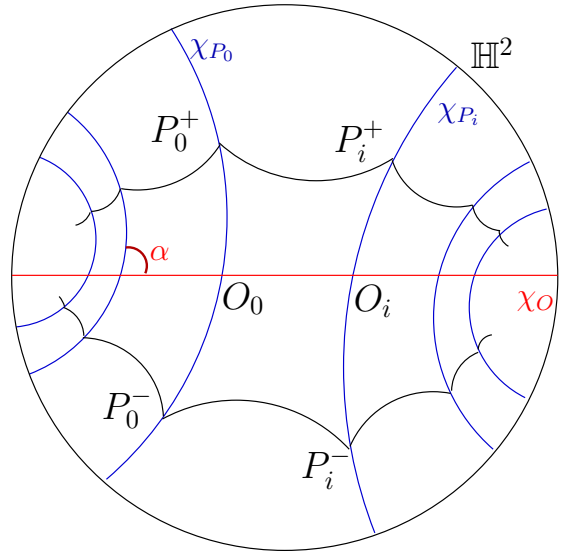
\caption{The quadrilateral ${P}^{+}_{0}{P}^{-}_{0}{P}^{+}_{i}{P}^{-}_{i}$, $i=1,2$, in Situation 1.}\label{dist_bb_situation1}
\end{center}
\end{figure}

\item[\textbf{Situation 2.}] If for at least one of the cylinders $Cyl_1$ or $Cyl_2$ the corresponding segment $O_{0}O_{1}$ or $O_{0}O_{2}$ does not intersect ${P}^{+}_{0}{P}^{-}_{0}{P}^{+}_{1}{P}^{-}_{1}$ or ${P}^{+}_{0}{P}^{-}_{0}{P}^{+}_{2}{P}^{-}_{2}$ (see Fig.~\ref{dist_bb_situation2}), then we will prove that the hyperbolic length of the segment ${P}^{+}{P}^{-}\subset Cyl_1\cap Cyl_2$ (and, hence, the distance between $\mathcal{S}^{+}$ and $\mathcal{S}^{-}$) is necessarily bounded by a constant depending on either ${l}^{+}_{1}$ and ${l}^{-}_{1}$, or ${l}^{+}_{2}$ and ${l}^{-}_{2}$.
\end{itemize}

\begin{figure}[!h]
\begin{center}
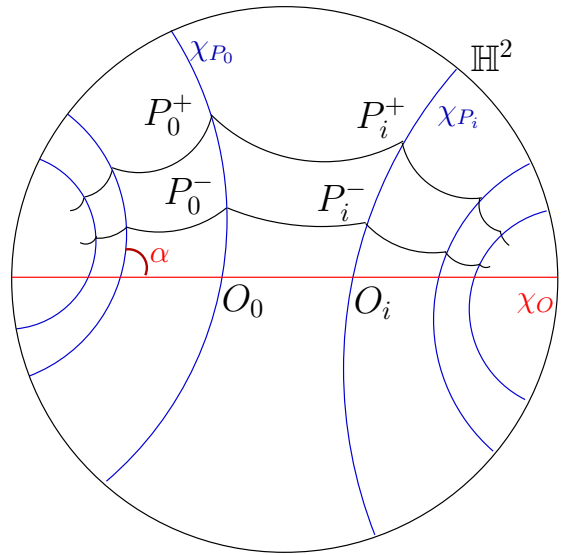
\caption{The quadrilateral ${P}^{+}_{0}{P}^{-}_{0}{P}^{+}_{i}{P}^{-}_{i}$, $i=1,2$, in Situation 2.}\label{dist_bb_situation2}
\end{center}
\end{figure}

It is now time to study

\subsection{Distance between boundary components of a cylinder of the type $Cyl$}
\label{chbb_sec_proof_of_main_theorem_in_general_case}

Let a quadrilateral ${R}^{+}_{0}{R}^{-}_{0}{R}^{+}_{1}{R}^{-}_{1}\subset\mathbb{H}^2$ with $h\stackrel{\mathrm{def}}{=}\mathrm{d}_{\mathbb{H}^2}({R}^{+}_{0},{R}^{-}_{0})=\mathrm{d}_{\mathbb{H}^2}({R}^{+}_{1},{R}^{-}_{1})$, $l^{+}\stackrel{\mathrm{def}}{=}\mathrm{d}_{\mathbb{H}^2}({R}^{+}_{0},{R}^{+}_{1})$, and $l^{-}\stackrel{\mathrm{def}}{=}\mathrm{d}_{\mathbb{H}^2}({R}^{-}_{0},{R}^{-}_{1})$ be a fundamental domain in $\mathbb{H}^2$ of a cylinder $Cyl_0$ of the type $Cyl$. Denote by $\chi_{{R}_{0}}$ and $\chi_{{R}_{1}}$ the hyperbolic straight lines in $\mathbb{H}^2$ containing the segments ${R}^{+}_{0}{R}^{-}_{0}$ and ${R}^{+}_{1}{R}^{-}_{1}$ correspondingly. Then, by Lemma~\ref{chbb_lemma_existence_of chi} applied to the points ${R}^{+}_{0}\in\chi_{{R}_{0}}$ and ${R}^{+}_{1}\in\chi_{{R}_{1}}$ there is a a unique hyperbolic straight line $\chi_{O}\subset\mathbb{H}^2$ intersecting $\chi_{{R}_{0}}$ at a point $O_{0}$, $\chi_{{R}_{1}}$ at a point $O_{1}$, such that ${R}^{+}_{0}$ and ${R}^{+}_{1}$ lie in the same half-plane with respect to $\chi_{O}$, $h^{+}\stackrel{\mathrm{def}}{=}{\mathrm{d}}_{\mathbb{H}^2}({R}^{+}_{0},O_{0})={\mathrm{d}}_{\mathbb{H}^2}({R}^{+}_{1},O_{1})$, and the angles of intersection $\angle(\chi_{O},\chi_{{R}_{0}})$ and $\angle(\chi_{O},\chi_{{R}_{1}})$ are equal to some $\alpha\in(0,\pi/2)$. Denote also $h^{-}\stackrel{\mathrm{def}}{=}{\mathrm{d}}_{\mathbb{H}^2}({R}^{-}_{0},O_{0})={\mathrm{d}}_{\mathbb{H}^2}({R}^{-}_{1},O_{1})$ and $l_{O}\stackrel{\mathrm{def}}{=}{\mathrm{d}}_{\mathbb{H}^2}(O_{0},O_{1})$.

Let the hyperbolic isometry $\bar{\chi}_{O}$ of $\mathbb{H}^2$ send $O_{0}$ to $O_{1}$ leaving the geodesic $\chi_{O}$ invariant. Note that $\bar{\chi}_{O}$ sends also ${R}^{+}_{0}$ to ${R}^{+}_{1}$ and ${R}^{-}_{0}$ to ${R}^{-}_{1}$. We define points ${R}^{+}_{i}\stackrel{\mathrm{def}}{=}{\bar{\chi}}^{i}_{O}.{R}^{+}_{0}$, ${R}^{-}_{i}\stackrel{\mathrm{def}}{=}{\bar{\chi}}^{i}_{O}.{R}^{-}_{0}$, and $O_{i}\stackrel{\mathrm{def}}{=}{\bar{\chi}}^{i}_{O}.{O}_{0}$ for $i\in\mathbb{Z}$, where the symbol ${\bar{\chi}}^{i}_{O}$ stands for the isometry $\bar{\chi}_{O}$ applied $i$ times when $i$ is a positive integer, and for the inverse isometry ${\bar{\chi}}^{-1}_{O}$ applied $-i$ times when $i<0$. Denote by $\chi_{{R}_{i}}$ the hyperbolic straight line containing the segment ${R}^{+}_{i}{R}^{-}_{i}$, $i\in\mathbb{Z}$. Construct the curves ${\nu}_{+}\stackrel{\mathrm{def}}{=}{\bigcup}_{i\in\mathbb{Z}}{R}^{+}_{i}{R}^{+}_{i+1}$ and ${\nu}_{-}\stackrel{\mathrm{def}}{=}{\bigcup}_{i\in\mathbb{Z}}{R}^{-}_{i}{R}^{-}_{i+1}$ of the geodesic segments ${R}^{+}_{i}{R}^{+}_{i+1}$ and ${R}^{-}_{i}{R}^{-}_{i+1}$, $i\in\mathbb{Z}$. Remark that for each $i\in\mathbb{Z}$ the quadrilateral ${R}^{+}_{i}{R}^{-}_{i}{R}^{+}_{i+1}{R}^{-}_{i+1}\subset\mathbb{H}^2$ serves as a fundamental domain of the cylinder $Cyl_0$ in $\mathbb{H}^2$, and the connected domain between the curves ${\nu}_{+}$ and ${\nu}_{-}$ of the hyperbolic plane is a universal covering of $Cyl_0$ in $\mathbb{H}^2$. By construction, $\mathrm{d}_{\mathbb{H}^2}({R}^{+}_{i},{R}^{-}_{i})=h$, ${\mathrm{d}}_{\mathbb{H}^2}({R}^{+}_{i},O_{i})=h^{+}$, ${\mathrm{d}}_{\mathbb{H}^2}({R}^{-}_{i},O_{i})=h^{-}$, $\mathrm{d}_{\mathbb{H}^2}({R}^{+}_{i},{R}^{+}_{i+1})=l^{+}$, $\mathrm{d}_{\mathbb{H}^2}({R}^{-}_{i},{R}^{-}_{i+1})=l^{-}$, $\angle(\chi_{O},\chi_{{R}_{i}})=\alpha$, $i\in\mathbb{Z}$.

Let us construct a family of hyperbolic straight lines $\chi^{+}_{i}$ passing through ${R}^{+}_{i}$ and orthogonal to $\chi_{O}$, $i\in\mathbb{Z}$. Define the points of intersection $O^{+}_{i}\stackrel{\mathrm{def}}{=}\chi^{+}_{i}\cap\chi_{O}$, $T^{-}_{i}\stackrel{\mathrm{def}}{=}\chi^{+}_{i}\cap\nu_{-}$, $i\in\mathbb{Z}$. Note that, by construction, the connected sets $\Xi^{+}_{i}$ bounded by $\chi^{+}_{i+1}$, $\nu_{+}$, $\chi^{+}_{i}$, and $\nu_{-}$ are fundamental domains of the cylinder $Cyl_0$ in $\mathbb{H}^2$, $i\in\mathbb{Z}$.

\begin{remark}\label{chbb_remark1_two_fundamental_domains_of_Cyl_0}
The geodesic segment ${R}^{+}_{i+1}{R}^{-}_{i+1}$ lies inside the fundamental domain $\Xi^{+}_{i}\subset\mathbb{H}^2$ of a cylinder $Cyl_0$ of the type $Cyl$; on the other hand, the geodesic segment ${R}^{+}_{i}T^{-}_{i}$ lies inside the fundamental domain ${R}^{+}_{i}{R}^{-}_{i}{R}^{+}_{i+1}{R}^{-}_{i+1}\subset\mathbb{H}^2$ of the same cylinder $Cyl_0$, $i\in\mathbb{Z}$.
\end{remark}

\begin{proof}
Since for every integer $i$ the hyperbolic straight lines $\chi^{+}_{i}$ are orthogonal to the geodesic $\chi_{O}$ corresponding to the closed geodesic $\chi^{\circ}$ of the unbounded cylinder $Cyl^{\circ}_{0}=\mathbb{H}^2/\langle\bar{\chi}_{O}\rangle$ which contains $Cyl_0$ (see also Section~\ref{chbb_sec_properties_of_the_cylinders_of_the_type_Cyl}), the projection on $Cyl_0$ of a path $\xi\subset\Xi^{+}_{i}$ connecting any point $P^{u}$ of the upper boundary $\partial\Xi^{+}_{i}\cap\nu_{+}(={R}^{+}_{i}{R}^{+}_{i+1})$ of $\Xi^{+}_{i}$ with any point $P^{l}$ of its lower boundary $\partial\Xi^{+}_{i}\cap\nu_{-}$ does not make a full turn around $Cyl_0$.

Let us fix $i\in\mathbb{Z}$. As $\Xi^{+}_{i}\subset\mathbb{H}^2$ is a fundamental domain of $Cyl_0$, the lower boundary $\partial\Xi^{+}_{i}\cap\nu_{-}$ of $\Xi^{+}_{i}$ must contain at least one and at most two points of the family $\{{R}^{-}_{j}\in\mathbb{H}^2|j\in\mathbb{Z}\}$ corresponding to one point on $Cyl_0$. Consider the point ${R}^{-}_{i+1}$ of this family. By Remark~\ref{chbb_remark_P+P-_is_short}, the length of the segment ${R}^{+}_{i+1}{R}^{-}_{i+1}$ is the smallest one among the lengths of all the segments ${R}^{+}_{i+1}{R}^{-}_{j}$, $j\in\mathbb{Z}$. Hence, the projection on $Cyl_0$ of ${R}^{+}_{i+1}{R}^{-}_{i+1}$  does not make a full turn around $Cyl_0$ (otherwise, there would be a path shorter than ${R}^{+}_{i+1}{R}^{-}_{i+1}$ among the segments ${R}^{+}_{i+1}{R}^{-}_{j}$, $j\in\mathbb{Z}$). Since $\alpha\in(0,\pi/2)$, we conclude that ${R}^{+}_{i+1}{R}^{-}_{i+1}\subset\Xi^{+}_{i}$. Similarly, ${R}^{+}_{i}{R}^{-}_{i}\subset\Xi^{+}_{i-1}$. Hence, ${R}^{+}_{i}T^{-}_{i}\subset{R}^{+}_{i}{R}^{-}_{i}{R}^{+}_{i+1}{R}^{-}_{i+1}$.
\end{proof}

Similarly, we construct a family of hyperbolic straight lines $\chi^{-}_{i}$ passing through ${R}^{-}_{i}$ and orthogonal to $\chi_{O}$, $i\in\mathbb{Z}$, and define the points of intersection $O^{-}_{i}\stackrel{\mathrm{def}}{=}\chi^{-}_{i}\cap\chi_{O}$, $T^{+}_{i}\stackrel{\mathrm{def}}{=}\chi^{-}_{i}\cap\nu_{+}$, $i\in\mathbb{Z}$. By construction, the connected sets $\Xi^{-}_{i}$ bounded by $\chi^{-}_{i+1}$, $\nu_{+}$, $\chi^{-}_{i}$, and $\nu_{-}$ are fundamental domains of the cylinder $Cyl_0$ in $\mathbb{H}^2$ and, by analogy with Remark~\ref{chbb_remark1_two_fundamental_domains_of_Cyl_0}, the following statement holds true.

\begin{remark}\label{chbb_remark2_two_fundamental_domains_of_Cyl_0}
The geodesic segment ${R}^{+}_{i}{R}^{-}_{i}$ lies inside the fundamental domain $\Xi^{-}_{i}\subset\mathbb{H}^2$ of a cylinder $Cyl_0$ of the type $Cyl$; on the other hand, the geodesic segment ${R}^{-}_{i+1}T^{+}_{i+1}$ lies inside the fundamental domain ${R}^{+}_{i}{R}^{-}_{i}{R}^{+}_{i+1}{R}^{-}_{i+1}\subset\mathbb{H}^2$ of the same cylinder $Cyl_0$, $i\in\mathbb{Z}$.
\end{remark}

Also, define $h^{+}_{O}\stackrel{\mathrm{def}}{=}\mathrm{d}_{\mathbb{H}^2}({R}^{+}_{i},{O}^{+}_{i})$, $h^{-}_{O}\stackrel{\mathrm{def}}{=}\mathrm{d}_{\mathbb{H}^2}({R}^{-}_{i},{O}^{-}_{i})$, and note that ${\mathrm{d}}_{\mathbb{H}^2}(O_{i},O_{i+1})={\mathrm{d}}_{\mathbb{H}^2}(O^{+}_{i},O^{+}_{i+1})
={\mathrm{d}}_{\mathbb{H}^2}(O^{-}_{i},O^{-}_{i+1})=l_{O}$, $i\in\mathbb{Z}$.

\subsubsection{Consideration of Situation~1}
\label{chbb_sec_situation_1_general_case}

In this section, we demonstrate

\begin{lemma}\label{chbb_lemma_situation1_gen}
Let a cylinder of the type $Cyl$ contain a closed geodesic and possess a fundamental domain ${R}^{+}_{0}{R}^{+}_{1}{R}^{-}_{0}{R}^{-}_{0}\subset\mathbb{H}^2$. Define by $l^{+}$ and $l^{-}$ the lengths of the sides ${R}^{+}_{0}{R}^{+}_{1}$ and ${R}^{-}_{0}{R}^{-}_{1}$, and by $h$ the length of ${R}^{+}_{0}{R}^{-}_{0}$ and ${R}^{+}_{1}{R}^{-}_{1}$. Then the condition
\begin{equation*}\label{chbb_frm1_h_condition-sit_1_gen}
h\geq2\max\bigg{\{}\arch\bigg{[}\cosinh{l}^{+}\cosinh \bigg{(}l^{+}+\arch\frac{e^{l^{+}}(l^{+})^{2}}{\varepsilon_{3}^{2}}\bigg{)}\bigg{]},
\end{equation*}
\begin{equation}\label{chbb_frm2_h_condition-sit_1_gen}
\arch\bigg{[}\cosinh{l}^{-}\cosinh \bigg{(}l^{-}+\arch\frac{e^{l^{-}}(l^{-})^{2}}{\varepsilon_{3}^{2}}\bigg{)}\bigg{]}\bigg{\}}.
\end{equation}
guarantees that there is a path in ${R}^{+}_{0}{R}^{+}_{1}{R}^{-}_{0}{R}^{-}_{0}$ connecting the midpoints of ${R}^{+}_{0}{R}^{-}_{0}$ and ${R}^{+}_{1}{R}^{-}_{1}$, and such that its length is smaller than the Margulis constant $\varepsilon_{3}$.
\end{lemma}

As we consider Situation~1, we suppose that $O_{i}\in{R}^{-}_{i}{R}^{+}_{i}$ for $i\in\mathbb{Z}$ and, consequently,
\begin{equation}\label{chbb_frm_h=h-_+_h+_-sit_1_gen}
h=h^{-}+h^{+}.
\end{equation}

For all $i\in\mathbb{Z}$, let us denote the midpoint of the segment ${R}^{+}_{i}{R}^{-}_{i}$ by ${R}^{mid}_{i}$, the midpoints of ${R}^{+}_{i}{O}_{i}$ and ${R}^{-}_{i}{O}_{i}$ by ${R}^{mid+}_{i}$ and ${R}^{mid-}_{i}$, the midpoints of ${R}^{+}_{i}{O}^{+}_{i}$ and ${R}^{-}_{i}{O}^{-}_{i}$ by ${O}^{mid+}_{i}$ and ${O}^{mid-}_{i}$. Denote the distances from the points ${R}^{mid}_{i}$ to the straight hyperbolic line $\chi_{O}$ by $d$, from ${R}^{mid+}_{i}$ to $\chi_{O}$ by ${d}^{+}$, from ${R}^{mid-}_{i}$ to $\chi_{O}$ by ${d}^{-}$ and note that, by construction, the distances from the points ${O}^{mid+}_{i}$ to $\chi_{O}$ are equal to $h^{+}_{O}/2$ and from the points ${O}^{mid-}_{i}$ to $\chi_{O}$ are equal to $h^{-}_{O}/2$, $i\in\mathbb{Z}$.

Denote by $\hat{\chi}$ a curve in $\mathbb{H}^2$ at the distance $d$ from $\chi_{O}$ and passing through the points ${R}^{mid}_{i}$ for all $i$ integers; by $\hat{\chi}^{+}_{R}$ a curve in $\mathbb{H}^2$ at the distance $d^{+}$ from $\chi_{O}$ and passing through the points ${R}^{mid+}_{i}$; by $\hat{\chi}^{-}_{R}$ a curve in $\mathbb{H}^2$ at the distance $d^{-}$ from $\chi_{O}$ and passing through the points ${R}^{mid-}_{i}$; by $\hat{\chi}^{+}_{O}$ a curve in $\mathbb{H}^2$ at the distance $h^{+}_{O}/2$ from $\chi_{O}$ and passing through the points ${O}^{mid+}_{i}$; by $\hat{\chi}^{-}_{O}$ a curve in $\mathbb{H}^2$ at the distance $h^{-}_{O}/2$ from $\chi_{O}$ and passing through the points ${O}^{mid-}_{i}$, $i\in\mathbb{Z}$.

\begin{remark}\label{chbb_remark_d_neighborhood_inside_h_O/2_neighborhood}
In the notation defined above, the inequalities
\begin{equation}\label{chbb_frm_d_<_h_O/2-sit_1_gen}
{d}^{+}\leq\frac{h^{+}_{O}}{2}\quad\text{and}\quad{d}^{-}\leq\frac{h^{-}_{O}}{2}
\end{equation}
hold true.
\end{remark}

\begin{proof}
Define by $\hat{R}^{mid+}_{0}$ the orthogonal projection of the point ${R}^{mid+}_{0}$ on $\chi_{O}\subset\mathbb{H}^2$ and consider the hyperbolic triangles $\triangle{O}_{0}{O}^{+}_{0}{R}^{+}_{0}$ and $\triangle{O}_{0}\hat{R}^{mid+}_{0}{R}^{mid+}_{0}$. Recall that $\mathrm{d}_{\mathbb{H}^2}({R}^{+}_{0},{O}^{+}_{0})=h^{+}_{O}$, $\mathrm{d}_{\mathbb{H}^2}({R}^{mid+}_{0},\hat{R}^{mid+}_{0})=d^{+}$, ${\mathrm{d}}_{\mathbb{H}^2}({R}^{+}_{0},O_{0})=h^{+}$,
${\mathrm{d}}_{\mathbb{H}^2}({R}^{mid+}_{0},O_{0})=h^{+}/2$,
$\angle{R}^{+}_{0}{O}_{0}{O}^{+}_{0}=\angle{R}^{mid+}_{0}{O}_{0}\hat{R}^{mid+}_{0}=\alpha$, and
$\angle{O}_{0}{O}^{+}_{0}{R}^{+}_{0}=\angle{O}_{0}\hat{R}^{mid+}_{0}{R}^{mid+}_{0}=\pi/2$.

Applying Hyperbolic Law of Sines to $\triangle{O}_{0}{O}^{+}_{0}{R}^{+}_{0}$ and $\triangle{O}_{0}\hat{R}^{mid+}_{0}{R}^{mid+}_{0}$, we obtain the formulas
\begin{equation*}\label{chbb_frm1_law_of_sines_h_O_+_by_h_+_-sit_1_gen}
\frac{\sin\alpha}{\sinush h^{+}_{O}}=\frac{\sin\frac{\pi}{2}}{\sinush h^{+}}
\end{equation*}
and
\begin{equation*}\label{chbb_frm1_law_of_sines_d_+_by_h_+_-sit_1_gen}
\frac{\sin\alpha}{\sinush d^{+}}=\frac{\sin\frac{\pi}{2}}{\sinush\frac{h^{+}}{2}},
\end{equation*}
or, after simplification,
\begin{equation}\label{chbb_frm2_law_of_sines_h_O_+_by_h_+_-sit_1_gen}
\sinush h^{+}_{O}=\sin\alpha \sinush h^{+}
\end{equation}
and
\begin{equation}\label{chbb_frm2_law_of_sines_d_+_by_h_+_-sit_1_gen}
\sinush d^{+}=\sin\alpha\sinush\frac{h^{+}}{2}.
\end{equation}

Note that when the formula
\begin{equation}\label{chbb_frm_sinh_d_<sinh_h_O/2-sit_1_gen}
\sinush{d}^{+}\leq\sinush\frac{h^{+}_{O}}{2}
\end{equation}
holds true, the first relation in~(\ref{chbb_frm_d_<_h_O/2-sit_1_gen}) is satisfied.

By~(\ref{chbb_frm2_law_of_sines_d_+_by_h_+_-sit_1_gen}), (\ref{chbb_frm_sinh_d_<sinh_h_O/2-sit_1_gen}) is equivalent to
\begin{equation}\label{chbb_frm_sinh_h_+_by_sinh_h_O/2_and_alpha-sit_1_gen}
\sin\alpha\sinush\frac{h^{+}}{2}\leq\sinush\frac{h^{+}_{O}}{2}.
\end{equation}
Due to the following property of the hyperbolic sine: $\sinush 2x=2 \sinush x \cosinh x$, from~(\ref{chbb_frm2_law_of_sines_h_O_+_by_h_+_-sit_1_gen}) we get
\begin{equation}\label{chbb_frm3_law_of_sines_h_O_+_by_h_+_-sit_1_gen}
2\sinush\frac{h^{+}_{O}}{2}\cosinh\frac{h^{+}_{O}}{2}=2\sin\alpha\sinush\frac{h^{+}}{2}\cosinh\frac{h^{+}}{2}
\end{equation}
As $h^{+}_{O}\leq h^{+}$ by construction and the function $\cosinh x$ is monotonically increasing for $x\geq0$, then it is true that $\cosinh(h^{+}_{O}/2)\leq\cosinh(h^{+}/2)$ and, by~(\ref{chbb_frm2_law_of_sines_h_O_+_by_h_+_-sit_1_gen}), we obtain
\begin{equation}\label{chbb_frm4_law_of_sines_h_O_+_by_h_+_-sit_1_gen}
\sinush\frac{h^{+}_{O}}{2}\cosinh\frac{h^{+}}{2}\geq\sin\alpha\sinush\frac{h^{+}}{2}\cosinh\frac{h^{+}}{2}.
\end{equation}
Simplifying~(\ref{chbb_frm4_law_of_sines_h_O_+_by_h_+_-sit_1_gen}), we see that the condition~(\ref{chbb_frm_sinh_h_+_by_sinh_h_O/2_and_alpha-sit_1_gen}) is satisfied. Hence, the first inequality in~(\ref{chbb_frm_d_<_h_O/2-sit_1_gen}) holds true.

The validity of the second relation in~(\ref{chbb_frm_d_<_h_O/2-sit_1_gen}) we prove by the same method.
\end{proof}

Together with constructions made above, Remark~\ref{chbb_remark_d_neighborhood_inside_h_O/2_neighborhood} means geometrically that the curve $\hat{\chi}$ lies inside the connected domain of the hyperbolic plane bounded by the curves $\hat{\chi}^{+}_{R}$ and $\hat{\chi}^{-}_{R}$ which is embedded into the connected domain bounded by $\hat{\chi}^{+}_{O}$ and $\hat{\chi}^{-}_{O}$ which is embedded, in its turn, into the connected domain bounded by ${\nu}_{+}$ and ${\nu}_{-}$.

By Remark~\ref{chbb_remark_length_of_TT'}, the length of the path $\widehat{{R}^{mid}_{i}{R}^{mid}_{i+1}}$ connecting the points ${R}^{mid}_{i}$ and ${R}^{mid}_{i+1}$ on the curve $\hat{\chi}$ is $\hat{l}=l_{O}\cosinh d$, the length of the path $\widehat{{R}^{mid+}_{i}{R}^{mid+}_{i+1}}\subset\hat{\chi}^{+}_{R}$ connecting the points ${R}^{mid+}_{i}$ and ${R}^{mid+}_{i+1}$ is $\hat{l}^{+}_{R}=l_{O}\cosinh d^{+}$, the length of the path $\widehat{{R}^{mid-}_{i}{R}^{mid-}_{i+1}}\subset\hat{\chi}^{-}_{R}$ connecting the points ${R}^{mid-}_{i}$ and ${R}^{mid-}_{i+1}$ is $\hat{l}^{-}_{R}=l_{O}\cosinh d^{-}$, the length of the path $\widehat{{O}^{mid+}_{i}{O}^{mid+}_{i+1}}\subset\hat{\chi}^{+}_{O}$ connecting the points ${O}^{mid+}_{i}$ and ${O}^{mid+}_{i+1}$ is $\hat{l}^{+}_{O}=l_{O}\cosinh (h^{+}_{O}/2)$, and the length of the path $\widehat{{O}^{mid-}_{i}{O}^{mid-}_{i+1}}\subset\hat{\chi}^{-}_{O}$ connecting the points ${O}^{mid-}_{i}$ and ${O}^{mid-}_{i+1}$ is $\hat{l}^{-}_{O}=l_{O}\cosinh (h^{-}_{O}/2)$, $i\in\mathbb{Z}$.

Assume that ${R}^{mid}_{i}\in{R}^{+}_{i}{O}_{i}$, $i\in\mathbb{Z}$. According to Remark~\ref{chbb_remark_d_neighborhood_inside_h_O/2_neighborhood}, we have
\begin{equation}\label{chbb_frm+_range_of_l_-sit_1_gen}
l_{O}\leq\hat{l}\leq\hat{l}^{+}_{R}\leq\hat{l}^{+}_{O}\leq{l}^{+}.
\end{equation}
Otherwise ${R}^{mid}_{i}\in{R}^{-}_{i}{O}_{i}$, $i\in\mathbb{Z}$ and
\begin{equation}\label{chbb_frm-_range_of_l_-sit_1_gen}
l_{O}\leq\hat{l}\leq\hat{l}^{-}_{R}\leq\hat{l}^{-}_{O}\leq{l}^{-}
\end{equation}
(remind that we consider Situation~1).
Hence, if we prove that for $h$ big enough $\hat{l}^{+}_{O}<\varepsilon_{3}$ and $\hat{l}^{-}_{O}<\varepsilon_{3}$, then $\hat{l}<\varepsilon_{3}$ and the projection of the path $\widehat{{R}^{mid}_{i}{R}^{mid}_{i+1}}\subset\mathbb{H}^2$ on the cylinder $Cyl_0$ is a closed curve which is shorter than the Margulis constant $\varepsilon_{3}$ and which passes through the midpoint ${R}^{mid}$ of the segment ${R}^{+}{R}^{-}\subset Cyl_0$ corresponding to ${R}^{+}_{i}{R}^{-}_{i}\subset\mathbb{H}^2$, $i\in\mathbb{Z}$.

First, fixing $l^{+}$ let us find a condition on $h^{+}$ which will guarantee $\hat{l}^{+}_{O}$ to be less than $\varepsilon_{3}$.

By Remark~\ref{chbb_remark1_two_fundamental_domains_of_Cyl_0}, the geodesic segment ${R}^{+}_{0}T^{-}_{0}$ lies inside the fundamental domain ${R}^{+}_{0}{R}^{-}_{0}{R}^{+}_{1}{R}^{-}_{1}\subset\mathbb{H}^2$. Hence, the point ${O}^{+}_{0}$ of intersection of ${R}^{+}_{0}T^{-}_{0}$ with $\chi_{O}$ belongs to the geodesic segment ${O}_{0}{O}_{1}$.

Denote $l_{{O}^{+}_{0}{O}_{0}}\stackrel{\mathrm{def}}{=}\mathrm{d}_{\mathbb{H}^2}({O}^{+}_{0},{O}_{0})$ and consider the right-angled triangle $\triangle{O}_{0}{O}^{+}_{0}{R}^{+}_{0}$. Hyperbolic Pythagorean Theorem implies:
\begin{equation}\label{chbb_frm_Pythagorean_thm_O0_O0+_-sit_1_gen}
\cosinh h^{+}=\cosinh h^{+}_{O}\cosinh l_{{O}^{+}_{0}{O}_{0}}.
\end{equation}
Since ${O}_{0}{O}^{+}_{0}\subset{O}_{0}{O}_{1}$, the inequality $l_{{O}^{+}_{0}{O}_{0}}\leq l_{O}$ holds true and, together with~(\ref{chbb_frm_Pythagorean_thm_O0_O0+_-sit_1_gen}) gives us
\begin{equation*}\label{chbb_frm1_cosh_h+_inequality_-sit_1_gen}
\cosinh h^{+}\leq\cosinh h^{+}_{O}\cosinh l_{O},
\end{equation*}
and, by~(\ref{chbb_frm+_range_of_l_-sit_1_gen}),
\begin{equation*}\label{chbb_frm2_cosh_h+_inequality_-sit_1_gen}
\cosinh h^{+}\leq\cosinh h^{+}_{O}\cosinh {l}^{+},
\end{equation*}
or, in other form,
\begin{equation}\label{chbb_frm3_cosh_h+_inequality_-sit_1_gen}
\cosinh h^{+}_{O}\geq\frac{\cosinh h^{+}}{\cosinh {l}^{+}}.
\end{equation}

It means that, once we take $h^{+}$ to satisfy the condition
\begin{equation}\label{chbb_frm_condition_on_cosh_h+_by_l+_-sit_1_gen}
\cosinh h^{+}\geq\cosinh{l}^{+}\cosinh \bigg{(}l^{+}+\arch\frac{e^{l^{+}}(l^{+})^{2}}{\varepsilon_{3}^{2}}\bigg{)},
\end{equation}
then, according to~(\ref{chbb_frm3_cosh_h+_inequality_-sit_1_gen}),
\begin{equation*}\label{chbb_frm_condition_on_cosh_hO+_by_l+_-sit_1_gen}
h^{+}_{O}\geq l^{+}+\arch\frac{e^{l^{+}}(l^{+})^{2}}{\varepsilon_{3}^{2}},
\end{equation*}
and, by Lemma~\ref{chbb_lemma_h_ort_int} applied to the quadrilateral ${O}^{+}_{0}{O}^{+}_{1}{R}^{+}_{0}{R}^{+}_{1}$, we conclude that
\begin{equation}\label{chbb_frm_hat_l+O_<_margulis-sit_1_gen}
\hat{l}^{+}_{O}\leq\varepsilon_{3}.
\end{equation}

Similarly, if we take $h^{-}$ to verify the inequality
\begin{equation}\label{chbb_frm_condition_on_cosh_h-_by_l-_-sit_1_gen}
\cosinh h^{-}\geq\cosinh{l}^{-}\cosinh \bigg{(}l^{-}+\arch\frac{e^{l^{-}}(l^{-})^{2}}{\varepsilon_{3}^{2}}\bigg{)},
\end{equation}
then
\begin{equation}\label{chbb_frm_hat_l-O_<_margulis-sit_1_gen}
\hat{l}^{-}_{O}\leq\varepsilon_{3}.
\end{equation}

Finally, let the condition~(\ref{chbb_frm2_h_condition-sit_1_gen}) be satisfied.
Supposing $h^{+}\geq h^{-}$, we have $\widehat{{R}^{mid}_{0}{R}^{mid}_{1}}\subset{O}^{+}_{0}{O}^{+}_{1}{R}^{+}_{0}{R}^{+}_{1}$ and, by~(\ref{chbb_frm_h=h-_+_h+_-sit_1_gen}), the inequality~(\ref{chbb_frm_condition_on_cosh_h+_by_l+_-sit_1_gen}) holds true, which implies (\ref{chbb_frm_hat_l+O_<_margulis-sit_1_gen}) and, due to~(\ref{chbb_frm+_range_of_l_-sit_1_gen}), leads as to the validity of the condition
\begin{equation}\label{chbb_frm_hat_l<_margulis-sit_1_gen}
\hat{l}\leq\varepsilon_{3}.
\end{equation}
On the other hand, if $h^{+}<h^{-}$ then $\widehat{{R}^{mid}_{0}{R}^{mid}_{1}}\subset{O}^{-}_{0}{O}^{-}_{1}{R}^{-}_{0}{R}^{-}_{1}$ and, by~(\ref{chbb_frm_h=h-_+_h+_-sit_1_gen}), the inequality~(\ref{chbb_frm_condition_on_cosh_h-_by_l-_-sit_1_gen}) holds true,
which implies (\ref{chbb_frm_hat_l-O_<_margulis-sit_1_gen}) and, due to~(\ref{chbb_frm-_range_of_l_-sit_1_gen}), leads as to the validity of~(\ref{chbb_frm_hat_l<_margulis-sit_1_gen}).

Lemma~\ref{chbb_lemma_situation1_gen} is proved.

\subsubsection{Consideration of Situation~2}
\label{chbb_sec_situation_2_general_case}

\begin{lemma}\label{chbb_lemma_situation2_gen}
Let a cylinder of the type $Cyl$ do not contain a closed geodesic and possess a fundamental domain ${R}^{+}_{0}{R}^{+}_{1}{R}^{-}_{0}{R}^{-}_{0}\subset\mathbb{H}^2$. Define by $l^{+}$ and $l^{-}$ the lengths of the sides ${R}^{+}_{0}{R}^{+}_{1}$ and ${R}^{-}_{0}{R}^{-}_{1}$, and by $h$ the length of ${R}^{+}_{0}{R}^{-}_{0}$ and ${R}^{+}_{1}{R}^{-}_{1}$. Then
\begin{equation*}\label{chbb_frm_from_lemma_h_by l+-_-sit_2_gen}
h<\max\bigg{\{}\bigg{(}l^{+}+l^{-}+\ln\frac{2l^{+}}{l^{-}}\bigg{)}, \bigg{(}l^{+}+l^{-}+\ln\frac{2l^{-}}{l^{+}}\bigg{)}\bigg{\}}.
\end{equation*}
\end{lemma}

\begin{proof}
We will use notation developed in Section~\ref{chbb_sec_proof_of_main_theorem_in_general_case}. In these terms, the fact that a cylinder of the type $Cyl$ does not contain a closed geodesic means that the segment $O_{0}O_{1}$ lies outside the fundamental domain ${R}^{+}_{0}{R}^{+}_{1}{R}^{-}_{0}{R}^{-}_{0}\subset\mathbb{H}^2$ of the cylinder.

First, we suppose that $h^{+}\geq h^{-}$, then
\begin{equation}\label{chbb_frm_h=h+_-h-_-sit_2_gen}
h=h^{+}-h^{-},
\end{equation}
which distinguishes Situation~2 from Situation~1 (compare~(\ref{chbb_frm_h=h+_-h-_-sit_2_gen}) with~(\ref{chbb_frm_h=h-_+_h+_-sit_1_gen})).

Denote
\begin{equation}\label{chbb_frm_hO=hO+_-hO-_-sit_2_gen}
h_{O}\stackrel{\mathrm{def}}{=}h^{+}_{O}-h^{-}_{O},
\end{equation}
construct a curve $\hat{\chi}^{-}\subset\mathbb{H}^2$ at the distance $h^{-}_{O}$ from $\chi_{O}$ and passing through the points ${R}^{-}_{i}$, and define the points of intersection $K^{-}_{i}\stackrel{\mathrm{def}}{=}\chi^{+}_{i}\cap\hat{\chi}^{-}$, $i\in\mathbb{Z}$. By construction, the lengths $l_{{R}^{+}_{i}K^{-}_{i}}$ and $l_{{O}^{+}_{i}K^{-}_{i}}$ of the segments ${R}^{+}_{i}K^{-}_{i}\subset{R}^{+}_{i}{O}^{+}_{i}$ and ${O}^{+}_{i}K^{-}_{i}\subset{R}^{+}_{i}{O}^{+}_{i}$ are equal to
\begin{equation}\label{chbb_frm_lengths_segments_of_chi+i_-sit_2_gen}
l_{{R}^{+}_{i}K^{-}_{i}}=h_{O}\quad\text{and}\quad l_{{O}^{+}_{i}K^{-}_{i}}=h^{-}_{O},
\end{equation}
$i\in\mathbb{Z}$. Define also the path $\widehat{{R}^{-}_{i}{K}^{-}_{i}}$ connecting the points ${R}^{-}_{i}$ and ${K}^{-}_{i}$ on the curve $\hat{\chi}^{-}$, $i\in\mathbb{Z}$.

By Remark~\ref{chbb_remark1_two_fundamental_domains_of_Cyl_0}, the geodesic segment ${R}^{+}_{0}K^{-}_{0}\subset{R}^{+}_{0}T^{-}_{0}$ lies inside the fundamental domain ${R}^{+}_{0}{R}^{-}_{0}{R}^{+}_{1}{R}^{-}_{1}\subset\mathbb{H}^2$. Hence, the path $\widehat{{R}^{-}_{i}{K}^{-}_{i}}$ is contained in the hyperbolic ball $B_{{R}^{-}_{0}}(l^{-})$ (also, we see that the segment ${R}^{-}_{0}{R}^{-}_{1}$ is a radius of $B_{{R}^{-}_{0}}(l^{-})$), and the length $l_{{R}^{-}_{0}K^{-}_{0}}$ of the segment ${R}^{-}_{0}K^{-}_{0}\subset{R}^{-}_{0}{R}^{-}_{1}$ satisfies the following inequality:
\begin{equation}\label{chbb_frm_up_bound_of_length_of_R-0_K-0-sit_2_gen}
l_{{R}^{-}_{0}K^{-}_{0}}\leq l^{-}.
\end{equation}

Applying the triangle inequality to $\triangle {R}^{+}_{0}{R}^{-}_{0}K^{-}_{0}$, we get:
\begin{equation*}\label{chbb_frm1_triangle_inequality-sit_2_gen}
h\leq l_{{R}^{-}_{0}K^{-}_{0}}+l_{{R}^{+}_{0}K^{-}_{0}},
\end{equation*}
and, by (\ref{chbb_frm_lengths_segments_of_chi+i_-sit_2_gen}) and~(\ref{chbb_frm_up_bound_of_length_of_R-0_K-0-sit_2_gen}),
\begin{equation}\label{chbb_frm2_triangle_inequality-sit_2_gen}
h\leq l^{-}+h_{O}.
\end{equation}

Let us now estimate the parameter $h_{O}$ from above.

Given the quadrilateral ${O}^{-}_{0}{O}^{-}_{1}{R}^{-}_{0}{R}^{-}_{1}$, Remarks~\ref{chbb_remark_length_of_TT'} and~\ref{chbb_remark_TT'_outside_OO'PP'} imply
\begin{equation*}\label{chbb_frm2_hO-_by_l_O_and_l-_-sit_2_gen}
l_{O}\cosinh h^{-}_{O}>l^{-},
\end{equation*}
then, by the definition of the hyperbolic cosine~(\ref{chbb_frm_sinh_and_cosh_defs-sit_1_orth}), we have
\begin{equation*}\label{chbb_frm3_hO-_by_l_O_and_l-_-sit_2_gen}
\frac{e^{h^{-}_{O}}+e^{-h^{-}_{O}}}{2}>\frac{l^{-}}{l_{O}},
\end{equation*}
and, as $e^{h^{-}_{O}}\geq e^{-h^{-}_{O}}$ for $h^{-}_{O}\geq0$, we obtain
\begin{equation}\label{chbb_frm4_hO-_by_l_O_and_l-_-sit_2_gen}
e^{h^{-}_{O}}>\frac{l^{-}}{l_{O}}.
\end{equation}

If $h^{+}_{O}\leq l^{+}$ then, by~(\ref{chbb_frm_hO=hO+_-hO-_-sit_2_gen}),
\begin{equation}\label{chbb_frm_hO<l+_-sit_2_gen}
h_{O}\leq l^{+}
\end{equation}
as well.

Assume that $h^{+}_{O}>l^{+}$. By Remarks~\ref{chbb_remark_length_of_TT'} and~\ref{chbb_remark_TT'_inside_OO'PP'} applied to the quadrilateral ${O}^{+}_{0}{O}^{+}_{1}{R}^{+}_{0}{R}^{+}_{1}$, we get
\begin{equation*}\label{chbb_frm_hO+_by_l_O_and_l+_-sit_2_gen}
l_{O}\cosinh (h^{+}_{O}-l^{+})<l^{+},
\end{equation*}
and, by~(\ref{chbb_frm_hO=hO+_-hO-_-sit_2_gen}),
\begin{equation*}\label{chbb_frm1_hO_and_hO-_by_l_O_and_l+_-sit_2_gen}
l_{O}\cosinh (h^{-}_{O}+h_{O}-l^{+})<l^{+},
\end{equation*}
then the definition of the hyperbolic cosine~(\ref{chbb_frm_sinh_and_cosh_defs-sit_1_orth}) gives us
\begin{equation*}\label{chbb_frm2_hO_and_hO-_by_l_O_and_l+_-sit_2_gen}
e^{h^{-}_{O}}e^{h_{O}}e^{-l^{+}}+e^{-h^{-}_{O}}e^{-h_{O}}e^{l^{+}}<\frac{2l^{+}}{l_{O}}.
\end{equation*}
Let us weaken the obtained inequality:
\begin{equation*}\label{chbb_frm2'_hO_and_hO-_by_l_O_and_l+_-sit_2_gen}
e^{h^{-}_{O}}e^{h_{O}}e^{-l^{+}}<\frac{2l^{+}}{l_{O}},
\end{equation*}
and, together with~(\ref{chbb_frm4_hO-_by_l_O_and_l-_-sit_2_gen}), we get
\begin{equation*}\label{chbb_frm3_hO_and_hO-_by_l_O_and_l+_-sit_2_gen}
\frac{l^{-}}{l_{O}}e^{h_{O}}e^{-l^{+}}<\frac{2l^{+}}{l_{O}},
\end{equation*}
\begin{equation*}\label{chbb_frm4_hO_and_hO-_by_l_O_and_l+_-sit_2_gen}
e^{h_{O}}<\frac{2l^{+}}{l^{-}}e^{l^{+}},
\end{equation*}
\begin{equation}\label{chbb_frm5_hO_and_hO-_by_l_O_and_l+_-sit_2_gen}
h_{O}<l^{+}+\ln\frac{2l^{+}}{l^{-}}.
\end{equation}

Note that the inequality~(\ref{chbb_frm_hO<l+_-sit_2_gen}) is stronger than~(\ref{chbb_frm5_hO_and_hO-_by_l_O_and_l+_-sit_2_gen}). Mixing~(\ref{chbb_frm2_triangle_inequality-sit_2_gen}) and (\ref{chbb_frm5_hO_and_hO-_by_l_O_and_l+_-sit_2_gen}) we get:
\begin{equation}\label{chbb_frm1_h_up_bound-sit_2_gen}
h<l^{-}+l^{+}+\ln\frac{2l^{+}}{l^{-}}.
\end{equation}

Supposing $h^{+}< h^{-}$, we just need to interchange the upper indices $+$ and $-$ in the formula~(\ref{chbb_frm1_h_up_bound-sit_2_gen}):
\begin{equation*}\label{chbb_frm2_h_up_bound-sit_2_gen}
h<l^{-}+l^{+}+\ln\frac{2l^{-}}{l^{+}}.
\end{equation*}
\end{proof}

\subsection{Finalizing the proof of Theorem~\ref{chbb_theorem_distance_between_boundaries}}
\label{chbb_sec_finalizing_the_proof_of_main_theorem}

Consider some points $P^{+}\in c^{+}_{1}\cap c^{+}_{2}$ and $P^{-}\in c^{-}_{1}\cap c^{-}_{2}$. As in Section~\ref{chbb_sec_construction_of_the_cylinders_Cyl_1_and_Cyl_2}, construct the cylinders $Cyl_1$ and $Cyl_2$ of the type $Cyl$ homotopically equivalent to the pairs of curves $(c^{+}_{1},c^{-}_{1})$ and $(c^{+}_{2},c^{-}_{2})$, with the upper boundaries of the lengths $l^{+}_{1}$ and $l^{+}_{2}$, with the lower boundaries of the lengths $l^{-}_{1}$ and $l^{-}_{2}$, and such that the hyperbolic geodesic segment $P^{+}P^{-}\subset\mathcal{M}^{\circ}$ lies in the intersection $Cyl_1 \cap Cyl_2$.

If Situation~2 is realized for at least one of the cylinders $Cyl_1$ and $Cyl_2$, than Lemma~\ref{chbb_lemma_situation2_gen} implies that
\begin{equation*}\label{chbb_frm_dist_S+_S-_sit2-finalizing}
d(\mathcal{S}^{+},\mathcal{S}^{-})<\max\bigg{\{}\bigg{(}l^{+}_{1}+l^{-}_{1}+\ln\frac{2l^{+}_{1}}{l^{-}_{1}}\bigg{)}, \bigg{(}l^{+}_{1}+l^{-}_{1}+\ln\frac{2l^{-}_{1}}{l^{+}_{1}}\bigg{)},\bigg{(}l^{+}_{2}+l^{-}_{2}+
\ln\frac{2l^{+}_{2}}{l^{-}_{2}}\bigg{)}, \bigg{(}l^{+}_{2}+l^{-}_{2}+\ln\frac{2l^{-}_{2}}{l^{+}_{2}}\bigg{)}\bigg{\}}.
\end{equation*}

Otherwise, Situation~1 is realized for both cylinders $Cyl_1$ and $Cyl_2$ and, once we suppose
\begin{equation*}\label{chbb_frm1_dist_S+_S-_sit1-finalizing}
d(\mathcal{S}^{+},\mathcal{S}^{-})<2\max\bigg{\{}\arch\bigg{[}\cosinh{l}^{+}_{1}\cosinh \bigg{(}l^{+}_{1}+\arch\frac{e^{l^{+}_{1}}(l^{+}_{1})^{2}}{\varepsilon_{3}^{2}}\bigg{)}\bigg{]},
\end{equation*}
\begin{equation*}\label{chbb_frm2_dist_S+_S-_sit1-finalizing}
\arch\bigg{[}\cosinh{l}^{-}_{1}\cosinh \bigg{(}l^{-}_{1}+\arch\frac{e^{l^{-}_{1}}(l^{-}_{1})^{2}}{\varepsilon_{3}^{2}}\bigg{)}\bigg{]},
\arch\bigg{[}\cosinh{l}^{+}_{2}\cosinh \bigg{(}l^{+}_{2}+\arch\frac{e^{l^{+}_{2}}(l^{+}_{2})^{2}}{\varepsilon_{3}^{2}}\bigg{)}\bigg{]},
\end{equation*}
\begin{equation*}\label{chbb_frm3_dist_S+_S-_sit1-finalizing}
\arch\bigg{[}\cosinh{l}^{-}_{2}\cosinh \bigg{(}l^{-}_{2}+\arch\frac{e^{l^{-}_{2}}(l^{-}_{2})^{2}}{\varepsilon_{3}^{2}}\bigg{)}\bigg{]}\bigg{\}},
\end{equation*}
by Lemma~\ref{chbb_lemma_situation1_gen}, there are curves $cur_1 \subset Cyl_1$ and $cur_2 \subset Cyl_2$ with the lengths less than the Margulis constant $\varepsilon_{3}$, both passing through the midpoint of the segment $P^{+}P^{-}$. Thus, we come to a contradiction with Margulis Lemma.

Theorem~\ref{chbb_theorem_distance_between_boundaries} is proved. $\square$

\section*{Acknowledgements}
First of all, the author expresses his sincere gratitude to his PhD co-advisor Jean-Marc Schlenker who stated the problem studied in this paper, and guided the research with great patience and attention.

As well, the author would like to thank very much Cyril Lecuire for fruitful discussions around Margulis lemma; Gregory McShane for an attentive reading of the PhD thesis, for the valuable comments and suggestions, as well as for the idea of generalization of Theorem~\ref{chaaa_theorem_manifolds_with_polyhedral_metric_on_boundary} to the case of all Alexandrov metrics of curvature $K\geq1$; and also, Olivier Guichard and Pierre Py for their advice in the approximation of Alexandrov metrics by regular ones.

\bibliographystyle{alpha}
\bibliography{slutskiy2014}

\begin{thebibliography}{CEG06}

\bibitem[Ale45]{chaaa_ADA1945}
A.D. Alexandroff.
\newblock {Complete convex surfaces in Lobachevskian space.}
\newblock {\em Izv. Akad. Nauk SSSR, Ser. Mat.}, 9:113--120, 1945.

\bibitem[Ale06]{chaaa_ADA2006}
A.~D. Alexandrov.
\newblock {\em Intrinsic Geometry of Convex Surfaces}.
\newblock Selected Works: Part II. Chapman and Hall/CRC, Berlin, 2006.

\bibitem[AVS93]{chaaa_Vinberg1988}
D.~V. Alekseevskij, {\`E}.~B. Vinberg, and A.~S. Solodovnikov.
\newblock Geometry of spaces of constant curvature.
\newblock In {\em Geometry, {II}}, volume~29 of {\em Encyclopaedia Math. Sci.},
  pages 1--138. Springer, Berlin, 1993.

\bibitem[BGS85]{chbb_BaGrSc1985}
Werner Ballmann, Mikhael Gromov, and Viktor Schroeder.
\newblock {\em Manifolds of nonpositive curvature}, volume~61 of {\em Progress
  in Mathematics}.
\newblock Birkh\"auser Boston Inc., Boston, MA, 1985.

\bibitem[BP03]{chaaa_BP2003}
Riccardo Benedetti and Carlo Petronio.
\newblock {\em Lectures on hyperbolic geometry}.
\newblock Universitext. Springer, Berlin, 2003.

\bibitem[Bro01]{chaaa_Brock2001}
Jeffrey~F. Brock.
\newblock {Iteration of mapping classes and limits of hyperbolic 3-manifolds.}
\newblock {\em Invent. Math.}, 143(3):523--570, 2001.

\bibitem[CEG06]{chaaa_CaEpGr2006}
R.~D. Canary, D.~B.~A. Epstein, and P.~L. Green.
\newblock Notes on notes of {T}hurston [mr0903850].
\newblock In {\em Fundamentals of hyperbolic geometry: selected expositions},
  volume 328 of {\em London Math. Soc. Lecture Note Ser.}, pages 1--115.
  Cambridge Univ. Press, Cambridge, 2006.
\newblock With a new foreword by Canary.

\bibitem[Die60]{chaaa_Dieudonne1960}
J.~Dieudonn{\'e}.
\newblock {\em Foundations of modern analysis}.
\newblock Pure and Applied Mathematics, Vol. X. Academic Press, New York, 1960.

\bibitem[Gho02]{chaaa_Gho2002}
Mohammad Ghomi.
\newblock The problem of optimal smoothing for convex functions.
\newblock {\em Proc. Amer. Math. Soc.}, 130(8):2255--2259 (electronic), 2002.

\bibitem[Gro86]{chaaa_Gro1986}
Mikhael Gromov.
\newblock {\em {Partial differential relations.}}
\newblock {Ergebnisse der Mathematik und ihrer Grenzgebiete. 3. Folge, Bd. 9.
  Berlin etc.: Springer-Verlag. ix+363 pp.; DM 148.00 }, 1986.

\bibitem[Lab92]{chaaa_Lab1992}
Fran\c{c}ois Labourie.
\newblock {Prescribed metrics on the boundary of hyperbolic manifolds of
  dimension 3. (M\'etriques prescrites sur le bord des vari\'et\'es
  hyperboliques de dimension 3.)}.
\newblock {\em J. Differ. Geom.}, 35(3):609--626, 1992.

\bibitem[Mat04]{chaaa_Mat2004}
Katsuhiko Matsuzaki.
\newblock {Indecomposable continua and the limit sets of Kleinian groups.}
\newblock {Abikoff, William (ed.) et al., In the tradition of Ahlfors and Bers,
  III. Proceedings of the 3rd Ahlfors-Bers colloquium, Storrs, CT, USA, October
  18--21, 2001. Providence, RI: American Mathematical Society (AMS).
  Contemporary Mathematics 355, 321-332 (2004).}, 2004.

\bibitem[MS09]{chaaa_MS2009}
Sergiu Moroianu and Jean-Marc Schlenker.
\newblock Quasi-{F}uchsian manifolds with particles.
\newblock {\em J. Differential Geom.}, 83(1):75--129, 2009.

\bibitem[MT98]{chaaa_MaTa1998}
Katsuhiko Matsuzaki and Masahiko Taniguchi.
\newblock {\em Hyperbolic manifolds and {K}leinian groups}.
\newblock Oxford Mathematical Monographs. The Clarendon Press Oxford University
  Press, New York, 1998.
\newblock Oxford Science Publications.

\bibitem[Ota96]{chaaa_Otal1996}
Jean-Pierre Otal.
\newblock Le th\'eor\`eme d'hyperbolisation pour les vari\'et\'es fibr\'ees de
  dimension 3.
\newblock {\em Ast\'erisque}, (235):x+159, 1996.

\bibitem[Ota03]{chbb_Otal2001}
Jean-Pierre Otal.
\newblock Les g\'eod\'esiques ferm\'ees d'une vari\'et\'e hyperbolique en tant
  que n\oe uds.
\newblock In {\em Kleinian groups and hyperbolic 3-manifolds ({W}arwick,
  2001)}, volume 299 of {\em London Math. Soc. Lecture Note Ser.}, pages
  95--104. Cambridge Univ. Press, Cambridge, 2003.

\bibitem[Pog73]{chaaa_Pogor1973}
A.V. Pogorelov.
\newblock {\em {Extrinsic geometry of convex surfaces. Translated from the
  Russian by Israel Program for Scientific Translations.}}
\newblock {Translations of Mathematical Monographs. Vol. 35. Providence, R.I.:
  American Mathematical Society (AMS). VI, 669 p. \$ 40.50 }, 1973.

\bibitem[Ric12]{chaaa_Ri2012}
Thomas Richard.
\newblock {\em Flot de Ricci sans borne sup\'erieure sur la courbure et
  g\'eom\'etrie de certains espaces m\'etriques}.
\newblock PhD Thesis. Fourier Institute, Grenoble, 2012.

\bibitem[Sch06]{chaaa_Sch2006}
Jean-Marc Schlenker.
\newblock Hyperbolic manifolds with convex boundary.
\newblock {\em Inventiones mathematicae}, 163:109--169, 2006.

\bibitem[Shi93]{chaaa_Shi1993}
Katsuhiro Shiohama.
\newblock {\em An introduction to the geometry of {A}lexandrov spaces},
  volume~8 of {\em Lecture Notes Series}.
\newblock Seoul National University, Research Institute of Mathematics, Global
  Analysis Research Center, Seoul, 1993.

\bibitem[Slu13]{chaaa_Slu2013}
Dmitriy Slutskiy.
\newblock {\em M\'etriques poly\`edrales sur les bords de vari\'et\'es
  hyperboliques convexes et flexibilit\'e des poly\`edres hyperboliques}.
\newblock PhD Thesis. Univesit\'e Paul Sabatier, Toulouse, 2013.

\end{thebibliography}

%\bibliographystyle{amsplain}
%\begin{thebibliography}{10}

%\bibitem {A} T. Aoki, \textit{Calcul exponentiel des op\'erateurs
%microdifferentiels d'ordre infini.} I, Ann. Inst. Fourier (Grenoble)
%

%\bibitem {B} R. Brown, \textit{On a conjecture of Dirichlet},
%Amer. Math. Soc., Providence, RI, 1993.

%\bibitem {D} R. A. DeVore, \textit{Approximation of functions},
%Proc. Sympos. Appl. Math., vol. 36,
%Amer. Math. Soc., Providence, RI, 1986, pp. 34--56.

%\end{thebibliography}

\end{document}